\documentclass{article}[12pt]
\usepackage{wrapfig,cancel,authblk}
\usepackage{lscape}
\usepackage{rotating}
\usepackage{graphicx}
\usepackage{caption}
\usepackage{amsmath}
\usepackage{lineno,hyperref,amsfonts,enumerate,amsthm,appendix,lipsum}
\usepackage{color,amssymb,mathtools,lineno,url}
\usepackage{framed,setspace,subfigure}
\usepackage[dvipsnames]{xcolor}
\usepackage[margin=1.5in]{geometry}
\usepackage[utf8]{inputenc}
\usepackage[english]{babel}
\usepackage{bm}
\newtheorem{theorem}{\bf Theorem}[section]
\newtheorem{lem}{Lemma}

\newcommand{\drop}[1]{}
\newcommand{\norm}[1]{\left\lVert#1\right\rVert}
\begin{document}
\title{Solution of planar elastic stress problems using stress basis functions}
\author[1*]{Sankalp Tiwari}
\affil[1]{\begin{small} Mechanical and Aerospace Engineering, Indian Institute of Technology Hyderabad \end {small}}

\author[2]{Anindya Chatterjee}
\affil[2]{\begin{small} Mechanical Engineering, Indian Institute of Technology Kanpur \end{small}}
\affil[*]{\begin{small} Corresponding author, Email: snklptwr@gmail.com \end{small}}

\date{}
\maketitle

\begin{abstract}
The use of global displacement basis functions to solve boundary-value problems in linear elasticity is well established. No prior work uses a global {\em stress tensor} basis for such solutions. We present two such methods for solving stress problems in linear elasticity. In both methods, we split the sought stress $\bm{\sigma}$ into two parts, where neither part is required to satisfy strain compatibility. The first part, $\bm{\sigma}_p$, is {\em any} stress in equilibrium with the loading. The second part, $\bm{\sigma}_h$, is a self-equilibrated stress field on the unloaded body. In both methods, $\bm{\sigma}_h$ is expanded using tensor-valued global stress basis functions developed elsewhere. In the first method, the coefficients in the expansion are found by minimizing the strain energy based on the well-known complementary energy principle. For the second method, which is restricted to planar homogeneous isotropic bodies, we show that we merely need to minimize the squared $L^2$ norm of the {\em trace} of stress. For demonstration, we solve eight stress problems involving sharp corners, multiple-connectedness, non-zero net force and/or moment on an internal hole, body force, discontinuous surface traction, material inhomogeneity, and anisotropy. The first method presents a new application of a known principle. The second method presents a hitherto unreported principle, to the best of our knowledge.
\end{abstract}

\section{Introduction}
Variational methods, both displacement- and stress-based, have long been used to solve boundary value problems in solid mechanics \cite{washizu}. In displacement-based methods, direct expansion of the displacement using global {\em vector-valued} functions (e.g., free vibration modes) is common \cite{timoshenko}. However, in stress-based methods that use global expansions, the variational problem is first cast in terms of the Airy stress function, which is then expanded using the first few terms of a scalar polynomial or trigonometric series \cite{timoshenko}. To the best of our knowledge, there is no prior demonstration of using a global {\em tensor-valued} complete stress basis to solve boundary value problems in solid mechanics. 
 
In this paper, we consider linear elastic stress problems, by which we mean traction is prescribed on the whole boundary, and we are interested in finding only stresses. In our solution approach, the sought stress is split into a `particular' part $\bm{\sigma}_p$ and a `homogeneous' part $\bm{\sigma}_h$, neither of which is required to satisfy strain compatibility. $\bm{\sigma}_p$ is {\em any} stress in equilibrium with the loading, while $\bm{\sigma}_h$ is a self-equilibrated traction-free stress. Subsequently, $\bm{\sigma}_h$ is expanded in a series using tensor-valued global basis functions developed by us earlier \cite{tiwari}.

The first contribution of this paper lies in determining the coefficients in the above expansion by minimizing the strain energy of $\bm{\sigma}_p+\bm{\sigma}_h$, following the principle of minimum complementary energy. With increasing numbers of terms in the expansion, the computed approximation converges to the true stress in the strain energy norm. 

The academic contribution of this paper is a new variational principle for the special case of planar homogeneous isotropic bodies. We show that the true stress minimizes the squared $L^2$ norm of the trace of stress over the set of all stresses in equilibrium with the loading. The principle is useful for two reasons: firstly, it yields a computationally cheaper formulation than that obtained through minimization of strain energy; secondly, it shows that in planar homogeneous isotropic bodies, solely the dilatational part of the stress is enough to resolve the issue of strain compatibility. We will discuss these issues in detail in due course.

For a demonstration of the two solution methods mentioned, we solve eight planar stress problems in linear elasticity. These problems incorporate sharp corners, multiple-connectedness, non-zero net force and/or moment on an internal hole, body force, discontinuous surface traction, material inhomogeneity, and anisotropy.

We close this introduction with a brief description of the notation used in the paper. The dot product between two tensors of the same order represents total tensor contraction. Using Einstein's summation convention,
\[
\bm{A}\cdot \bm{B}= 
\begin{cases}
A_i B_i& \mbox{if } \,\, \bm{A} \,\, \mbox{and} \,\, \bm{B} \,\, \mbox{are vectors}, \\
A_{ij} B_{ij}& \mbox{if } \,\, \bm{A} \,\, \mbox{and} \,\, \bm{B} \,\, \mbox{are second-order tensors}, \\
A_{ijk} B_{ijk}& \mbox{if } \,\, \bm{A} \,\, \mbox{and} \,\, \bm{B} \,\, \mbox{are third-order tensors},
\end{cases}
\]
where $A_i$,$A_{ij}$,$A_{ijk}$ etc.\ are the Cartesian components of the tensor $\bm{A}$ (likewise for $\bm{B}$).
For a second-order tensor $\bm{\sigma}$, $\mbox{div} \bm{\sigma}$  represents $\sigma_{ij,j} \bm{e}_i $, where a subscript following a comma denotes a partial derivative and $\bm{e}_i$ are the unit Cartesian basis vectors. For a vector $\bm{v}$, $\bm{\sigma} \bm{v}$ represents $\sigma_{ij}v_j \bm{e}_i$.
The dyadic product $\bm{u} \otimes \bm{v}$ for vectors $\bm{u}$ and $\bm{v}$ is defined by its action on a vector $\bm{w}$ as $(\bm{u} \otimes \bm{v})\bm{w}=(\bm{v}\cdot \bm{w})\bm{u}$.
Accordingly, $\text{curl}\,\bm{\sigma} = \chi_{ijk}\sigma_{lj,k} \bm{e}_i \otimes \bm{e}_l$, where $\chi_{ijk}$ is the permutation tensor.
The sum of planar components of $\bm{\sigma}$ is denoted as $\bar{\sigma}=\sigma_{xx}+\sigma_{yy}=\sigma_{rr}+\sigma_{\theta\theta}$. 

The $L^2$ inner product between two scalar fields $f_1$ and $f_2$ on a domain $\Omega\in \mathbb{R}^2$ with area measure $dA$ is 
$$ (f_1,f_2)=\int_{\Omega} f_1 f_2 \, dA,$$
and the $L^2$ inner product between two second-order tensor fields $\bm{\sigma}_1$ and $\bm{\sigma}_2$ is
$$ \left( \bm{\sigma}_1,\bm{\sigma}_2 \right)=\int_{\Omega} \sigma_{1ij} \sigma_{2ij} \, dA.$$
Unless stated otherwise, we will refer to these $L^2$ inner products simply as inner products. The order of the tensor field in the inner product (scalar, vector or second-order tensor) will be clear from the context.

\section{Two variational principles: one old, one new} \label{variation}
We now present the two variational principles to be used in the next section. The first principle, called the `strain energy principle' henceforth, is based on strain energy minimization and applies to all linear elastic bodies. The second principle, henceforth referred to as the `planar trace principle,' is based on minimizing the squared $L^2$ norm of the trace of stress and applies to only planar homogeneous isotropic bodies\footnote{By planar, we mean that we will only consider plane strain problems. In general, the case of plane stress is the same as plane strain under a redefinition of the elastic parameters; see Timoshenko and Goodier \cite{timoshenko}. So, our results will also hold for plane stress problems.}. The planar trace principle is new to the best of our knowledge.

\subsection{Strain energy principle} \label{kala}
We consider a linear elastic body $\mathcal{B}$ subjected to a sufficiently regular self-equilibrated surface traction $\bm{\tau}$ prescribed on the whole boundary. Let the set of all square-integrable stresses in equilibrium with $\bm{\tau}$ be
\begin{equation}
\label{edefQ}
\mathcal{Q}=\left\{\bm{\sigma} \big| ~ \bm{\sigma}\in \,\text{Sym},~\int_{\Omega} \bm{\sigma}\cdot \bm{\sigma}\,dA<\infty,~\text{div}\,\bm{\sigma}=\bm{0}, ~ \left.\bm{\sigma n}\right|_{\partial \Omega}=\bm{\tau} \right\},
\end{equation}
where $\Omega$ is an open, bounded, sufficiently regular domain in $\mathbb{R}^2$ with boundary $\partial \Omega$, $\bm{n}$ is the unit outward normal to $\partial\Omega$, and ``Sym'' denotes the set of all symmetric second-order tensor fields defined over $\Omega$. 
The strain energy of this body 
\begin{equation*}
\mathcal{E}(\bm{\sigma})=\int_{\Omega} \bm{C}^{-1}\bm{\sigma}\cdot \bm{\sigma} \, dA,
\end{equation*}
where $\bm{C}$ is the fourth-order stiffness tensor possessing major and minor symmetries \cite{steigmann}. We then have the following lemma.

\begin{lem} [Strain energy principle] \label{SEPr}
The true stress minimizes $\mathcal{E}$ over $\mathcal{Q}$.
\end{lem}

\begin{proof}
We will use the calculus of variations \cite{hilbert}. We begin by seeking a {\em stationary point} of $\mathcal{E}$ over $\mathcal{Q}$. We incorporate the pointwise constraint $\text{div}\,\bm{\sigma}=\bm{0}$ through a spatially varying Lagrange multiplier vector $\bm{\mu}$, and extremize the functional
\begin{equation}
\tilde{\mathcal{E}}(\bm{\sigma},\bm{\mu})=\int_{\Omega}  \left(\bm{C}^{-1}\bm{\sigma}\cdot \bm{\sigma} - \bm{\mu}\cdot \text{div}\,\bm{\sigma} \right)\, dA
\label{Etild}
\end{equation}
over the set 
\begin{equation}
\tilde{\mathcal{Q}}=\left\{\bm{\sigma} \big| ~ \bm{\sigma}\in \,\text{Sym},~\int_{\Omega} \bm{\sigma}\cdot \bm{\sigma}\,dA<\infty,~ \left.\bm{\sigma n}\right|_{\partial \Omega}=\bm{\tau} \right\}.
\label{tildQ}
\end{equation}
Let $(\hat{\bm{\sigma}},\hat{\bm{\mu}})$ be a stationary point, and $\bm{\eta}$ be an infinitesimal perturbation in $\hat{\bm{\sigma}}$ such that $\hat{\bm{\sigma}} + \bm{\eta}\in \tilde{\mathcal{Q}}$. Then,
$$ \tilde{\mathcal{E}}(\bm{\sigma}+\bm{\eta},\bm{\mu}) = \tilde{\mathcal{E}}(\bm{\sigma},\bm{\mu})$$
up to first-order terms in $\bm{\eta}$, where we have dropped the hats to avoid clutter. Expanding the above using Eq.~\ref{Etild}, using the major symmetry of $\bm{C}^{-1}$ (which follows from the major symmetry of $\bm{C}$) to write $\bm{C}^{-1}\bm{\eta}\cdot \bm{\sigma}=\bm{C}^{-1}\bm{\sigma}\cdot\bm{\eta}$, and neglecting the second-order terms in $\bm{\eta}$, we obtain
$$ \int_{\Omega} \left(2\bm{C}^{-1}\bm{\sigma} \cdot \bm{\eta} - \bm{\mu}\cdot \text{div}\,\bm{\eta} \right)\, dA=0. $$
Upon using the divergence theorem and noting that $\bm{\eta}\bm{n}=\bm{0}$ on $\partial\Omega$ (since the traction is prescribed on the whole boundary), we obtain
$$\int_{\Omega} \left(2\bm{C}^{-1}\bm{\sigma} + \nabla \bm{\mu}\right)\cdot \bm{\eta}\, dA=0.$$
Since $\bm{\eta}$ is an arbitrary symmetric infinitesimal perturbation,
$$2\bm{C}^{-1}\bm{\sigma} + \nabla \bm{\mu}=\bm{R}$$
for some skew symmetric second-order tensor $\bm{R}$. Adding the above equation to its transpose and simplifying, we obtain
\begin{equation}
\bm{C}^{-1}\bm{\sigma} =- \nabla_s \bm{\mu}/2,
\label{sivaSt}
\end{equation}
where $\nabla_s$ denotes the symmetric part of the gradient. Thus, the strain $\bm{C}^{-1}\bm{\sigma}$ corresponding to the stationary point $\bm{\sigma}$ is compatible, with the corresponding displacement being $-\bm{\mu}/2$. 

The variation of $\bm{\mu}$ yields
$$ \text{div}\, \bm{\sigma}=\bm{0}.$$
Also, since $\bm{\sigma}\in \tilde{\mathcal{Q}}$ (Eq.~\ref{tildQ}), $\bm{\sigma n}=\bm{\tau}$. Thus, $\bm{\sigma}$ is in equilibrium with the traction $\bm{\tau}$.
 
From the uniqueness theorems of linear elasticity, it is known that if a stress is in equilibrium with a given traction and has a compatible strain, it {\em is} the solution to the given stress problem \cite{payne}. So, we conclude that the stationary point of $\mathcal{E}$ over $\mathcal{Q}$ is the true stress. In fact, the positive-definiteness of $\mathcal{E}$ coupled with the uniqueness of the elastic solution implies that $\bm{\sigma}$ minimizes $\mathcal{E}$ over $\mathcal{Q}$ \cite{washizu}.
\end{proof}

Note that the strain energy principle holds for {\em all} linear elastic bodies: homogeneous or inhomogeneous, isotropic or anisotropic, simply- or multiply-connected, planar or spatial.

Also note that this principle is a special case of the `principle of minimum complementary energy,' which states that the true stress minimizes the complementary energy
\begin{equation*}
U_c=\mathcal{E}_c-W_c
\end{equation*}
over $\mathcal{Q}$, where $\mathcal{E}_c$ is the complementary strain energy and $W_c$ is the complementary work. For linear elasticity, $\mathcal{E}_c=\mathcal{E}$, and for stress problems, $W_c=0$ since the portion of the boundary on which displacements are prescribed is a null set (e.g., see Page 31 of \cite{washizu}). Accordingly, $U_c=\mathcal{E}$.

We next consider the special case of planar homogeneous isotropic bodies for which the planar trace principle holds.

\subsection{Planar trace principle} \label{HomIso}
Let us denote the sum of the normal planar components of a stress $\bm{\sigma}$ as $\bar{\sigma}$, i.e., $\bar{\sigma}=\sigma_{xx}+\sigma_{yy}=\sigma_{rr}+\sigma_{\theta\theta}$. Henceforth, we will refer to $\bar{\sigma}$ as the {\em planar trace} of $\bm{\sigma}$. We define the trace energy\footnote{Trace of $\bm{\sigma}$ is $\sigma_{xx}+\sigma_{yy}+\sigma_{zz}$, while $\bar{\sigma}$ is $\sigma_{xx}+\sigma_{yy}$. However, for plane strain problems in homogeneous linear isotropic elasticity considered in this section, trace of $\bm{\sigma}$ is equal to $(1+\nu) \bar{\sigma}$. So, the minimization of the squared $L^2$ norms of these two quantities is equivalent. For this reason, we call $\mathcal{T}$ the trace energy.}
\begin{equation}
\label{L2e}
\mathcal{T}(\bm{\sigma})=\int_{\Omega} \bar{\sigma}^2 \, dA.
\end{equation}
We then have the following theorem.

\begin{theorem}[Planar trace principle] \label{PTP}
For a planar homogeneous isotropic body that obeys linear elasticity and is subjected to a traction with zero net force on each internal hole (if any), the true stress minimizes $\mathcal{T}$ over $\mathcal{Q}$.
\end{theorem}
\begin{proof}
Again, we begin by seeking the {\em stationary points} of $\mathcal{T}$ over $\mathcal{Q}$. To that end, we incorporate the pointwise constraint $\text{div}\,\bm{\sigma}=\bm{0}$ through a Lagrange multiplier vector $\bm{\mu}$, and extremize
\begin{equation}
\tilde{\mathcal{T}}(\bm{\sigma},\bm{\mu})=\int_{\Omega} \left(\bar{\sigma}^2- \bm{\mu}\cdot \text{div}\,\bm{\sigma} \right)\, dA
\label{tildT}
\end{equation}
over $\tilde{\mathcal{Q}}$ (Eq.~\ref{tildQ}).
Let $(\hat{\bm{\sigma}},\hat{\bm{\mu}})$ be a stationary point, and $\bm{\eta}$ denote an infinitesimal perturbation in $\hat{\bm{\sigma}}$ such that $\hat{\bm{\sigma}} + \bm{\eta}\in \tilde{\mathcal{Q}}$. Then,
$$ \tilde{\mathcal{T}}(\bm{\sigma}+\bm{\eta},\bm{\mu}) = \tilde{\mathcal{T}}(\bm{\sigma},\bm{\mu})$$
up to first-order terms in $\bm{\eta}$, where we have dropped the hats. Expanding the above using Eq.~\ref{tildT}, and neglecting the second-order terms in $\bm{\eta}$, we obtain
$$ \int_{\Omega} \left(2\bar{\sigma} \bar{\eta} - \bm{\mu}\cdot \text{div}\,\bm{\eta} \right)\, dA=0. $$
Proceeding as before, we obtain the Euler-Lagrange equation
\begin{equation}
\bar{\sigma}\bm{I} =- \nabla_s \bm{\mu}/2.
\label{siva}
\end{equation}
Matching components in the above, we obtain
\begin{equation}
\bar{\sigma} = \sigma_{xx}+\sigma_{yy}=-\frac{\mu_{x,x}}{2}=-\frac{\mu_{y,y}}{2},~~~ \text{and}~~~ \mu_{x,y}+\mu_{y,x}=0.
\label{mali}
\end{equation}
Eliminating $\bm{\mu}$, we obtain
\begin{equation}
\Delta \bar{\sigma}=0.
\label{traceComp2}
\end{equation}
We multiply the above equation by $1-\nu$, and rewrite it as
\begin{equation*}
(1-\nu)\left(\frac{\partial^2 \sigma_{xx}}{\partial y^2}+\frac{\partial^2 \sigma_{yy}}{\partial x^2}\right)-\nu\left(\frac{\partial^2 \sigma_{xx}}{\partial x^2}+\frac{\partial^2 \sigma_{yy}}{\partial y^2}\right)+\left(\frac{\partial^2 \sigma_{xx}}{\partial x^2}+\frac{\partial^2 \sigma_{yy}}{\partial y^2}\right)=0,
\end{equation*}
where $\nu$ is Poisson's ratio. Using the equilibrium equations, we write the last term in brackets on the left-hand side as
\begin{equation}
\frac{\partial^2\sigma_{xx}}{\partial x^2}+\frac{\partial^2\sigma_{yy}}{\partial y^2}=\frac{1}{2}\left(\frac{\partial^2\sigma_{xx}}{\partial x^2}+\frac{\partial^2\sigma_{yy}}{\partial y^2} - 2 \frac{\partial^2\sigma_{xy}}{\partial x \partial y}\right).
\label{transf}
\end{equation}
Finally, using the plane-strain constitutive relations
\begin{equation}
\sigma_{xx}=\frac{Y\left\{ \epsilon_{xx}(1-\nu)+\nu \epsilon_{yy}\right\}}{(1+\nu)(1-2\nu)},~~~\sigma_{yy}=\frac{Y\left\{\epsilon_{yy}(1-\nu)+\nu \epsilon_{xx}\right\}}{(1+\nu)(1-2\nu)}, ~~~ \sigma_{xy}=\frac{Y\epsilon_{xy}}{1+\nu},
\label{ss}
\end{equation}
where $Y$ is Young's modulus, we obtain
\begin{equation}
\frac{\partial^2 \epsilon_{xx}}{\partial y^2} + \frac{\partial^2 \epsilon_{yy}}{\partial x^2} - 2\frac{\partial^2 \epsilon_{xy}}{\partial x \partial y}= 0.
\label{weak2}
\end{equation}
Note that Eq.~\ref{weak2} is the strain compatibility equation for planar simply-connected bodies \cite{barber}. 

The variation of the Lagrange multiplier $\bm{\mu}$ gives
$$ \text{div}\, \bm{\sigma}=\bm{0}.$$
Also, since $\bm{\sigma}\in \tilde{\mathcal{Q}}$ (Eq.~\ref{tildQ}), $\bm{\sigma n}=\bm{\tau}$. Thus, $\bm{\sigma}$ is in equilibrium with the traction $\bm{\tau}$.

We have thus shown that for a simply-connected body, the stationary point $\bm{\sigma}$ is in equilibrium with $\bm{\tau}$ and satisfies strain compatibility. Therefore, it must be the sought stress solution. We conclude that the extremization of $\mathcal{E}$ and $\mathcal{T}$ are equivalent for simply-connected bodies. Again, the positive definiteness of $\mathcal{T}$ implies
that $\bm{\sigma}$ {\em minimizes} $\mathcal{T}$ over $\mathcal{Q}$.

We now consider multiply-connected bodies, for which Eq.~\ref{traceComp2} is a necessary but not sufficient condition for strain compatibility. If there are $n$ holes in the body, then the following additional conditions must be enforced for the compatibility of strain \cite{boley}:
\begin{equation}
\oint_{S_r} \left\{\bm{C}^{-1}\bm{\sigma} -  \left(\bm{X}- \hat{\bm{x}}\right) \times \text{curl}\, \bm{C}^{-1}\bm{\sigma}\right\} \,  d\hat{\bm{x}}=\bm{0},
\label{loopapp}
\end{equation}
where $S_r$ is {\em any} closed loop surrounding the $r^{\text{th}}$ hole, and $\bm{X}$ is an arbitrary fixed point. We need to enforce the above conditions, called C\'esaro's integral conditions, by considering one loop for each hole.

Thus, to prove Theorem \ref{PTP} for multiply-connected bodies, we must show that the minimizer of $\mathcal{T}$ over $\mathcal{Q}$ satisfies C\'esaro's integral conditions provided the net force on each hole is individually zero. 
\begin{figure}[t!]
	\centering
	\includegraphics [width=0.8\linewidth]{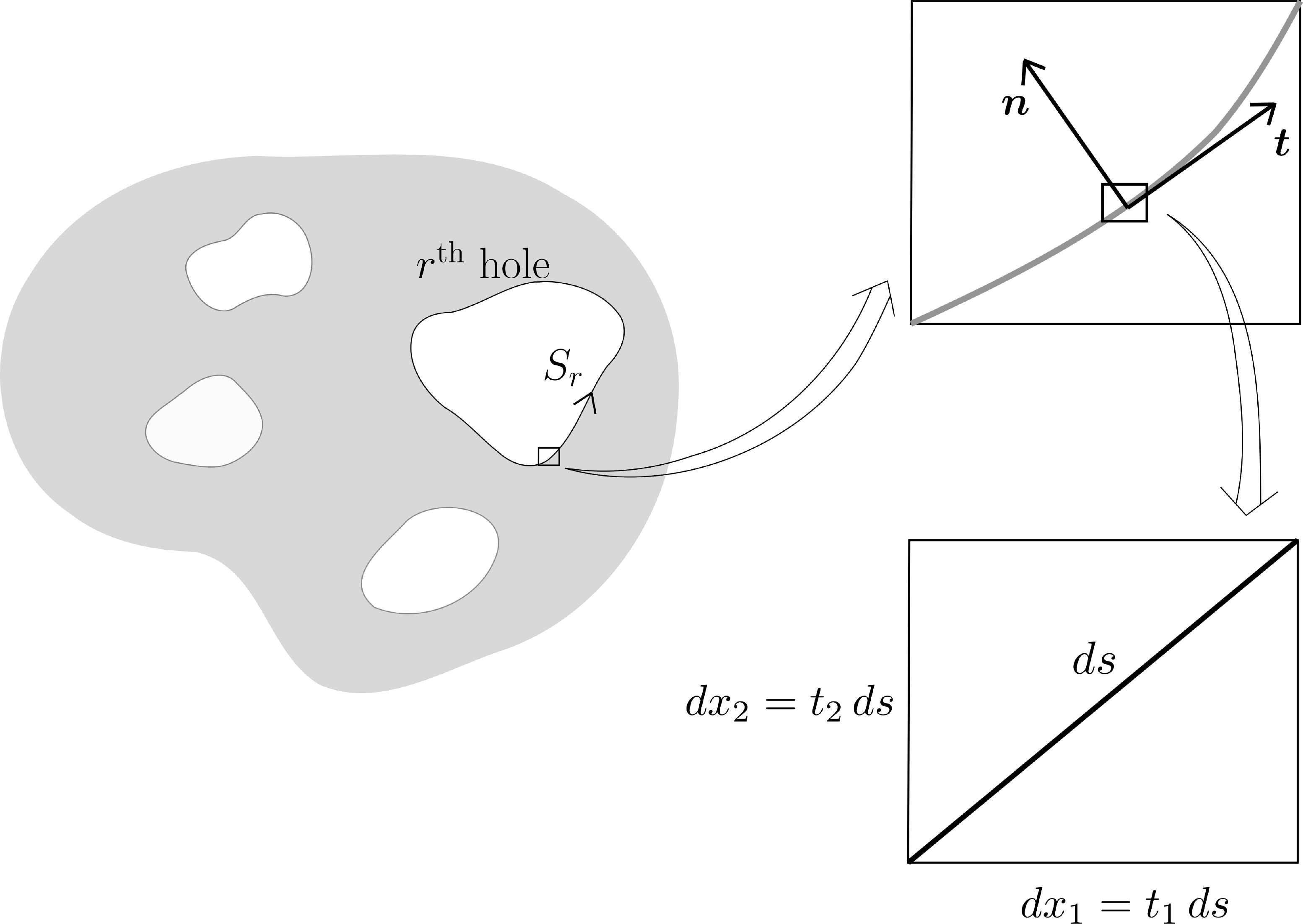}
\caption{A multiply-connected body.}
\label{multi}
\end{figure}

For convenience, we will use indicial notation. The integral around the $r^{\text{th}}$ hole (Figure \ref{multi}) in Eq.~\ref{loopapp} when expressed in indicial notation is
\begin{equation}
F_i=\oint_{S_r} U_{ij} \,  dx_j,
\label{Ii}
\end{equation}
where
\begin{equation}
U_{ij}=\epsilon_{ij} + \left(X_{l}- x_{l}\right) I_{ilj}; ~~~ I_{ilj}=\epsilon_{ij,l}-\epsilon_{lj,i}.
\label{Iilj}
\end{equation}
We must show that $F_i$ evaluated for the trace energy minimizer is zero. 

For plane strain, $\epsilon_{i3}=\epsilon_{3i}=\epsilon_{ij,3}=0$. Further, $\epsilon_{ij}=\epsilon_{ji}$. Using these, we can easily obtain that
\begin{equation*}
\begin{aligned}
&I_{111}=0,~~~I_{112}=0,~~~I_{121}=\epsilon_{11,2}-\epsilon_{12,1}, ~~~I_{122}=\epsilon_{12,2}-\epsilon_{22,1},\\
&I_{211}=\epsilon_{12,1}-\epsilon_{11,2},~~~I_{212}=\epsilon_{22,1}-\epsilon_{12,2},~~~I_{221}=0, ~~~I_{222}=0.
\end{aligned}
\end{equation*}
Substituting the above expressions in the first of Eq.~\ref{Iilj}, we obtain $U_{ij}$. We can simplify $U_{ij}$ using the equilibrium equations
\begin{equation*}
\sigma_{11,1}+\sigma_{12,2}=0,~~~\sigma_{12,1}+\sigma_{22,2}=0,
\end{equation*}
and the stress-strain relations
\begin{equation}
\epsilon_{11}=\left(1-\nu^2\right)\frac{\sigma_{11}}{Y}-\nu(1+\nu)\frac{\sigma_{22}}{Y},~~~\epsilon_{22}=\left(1-\nu^2\right)\frac{\sigma_{22}}{Y}-\nu(1+\nu)\frac{\sigma_{11}}{Y}, ~~~ \epsilon_{12}=(1+\nu)\frac{\sigma_{12}}{Y}.
\label{ss}
\end{equation}
For instance, 
\begin{equation*}
\begin{aligned}
U_{11}&=\epsilon_{11}+(X_1-x_1)I_{111}+(X_2-x_2)I_{121}\\
&=\epsilon_{11}+(X_2-x_2)(\epsilon_{11,2}-\epsilon_{12,1})\\
&=\epsilon_{11}+(X_2-x_2)\frac{1+\nu}{Y}\left[ \left\{\sigma_{11}(1-\nu) - \nu \sigma_{22}\right\}_{,2}-\sigma_{12,1}\right]\\
&=\epsilon_{11}+(X_2-x_2)\frac{(1+\nu)(1-\nu)}{Y} \left( \sigma_{11,2}+\sigma_{22,2}\right)\\
&=\epsilon_{11}+(X_2-x_2)\frac{1-\nu}{1-2\nu}\bar{\epsilon}_{,2}\\
&=\epsilon_{11}+(X_2-x_2)c\, \bar{\epsilon}_{,2},
\end{aligned}
\end{equation*}
where $\bar{\epsilon}=\epsilon_{11}+\epsilon_{22}$ and 
$$ c=\frac{1-\nu}{1-2\nu}.$$
It is similarly obtained that
\begin{equation*}
U_{12}=\epsilon_{12}-(X_2-x_2)c\, \bar{\epsilon}_{,1},~~~U_{21}=\epsilon_{12}-(X_1-x_1)c\, \bar{\epsilon}_{,2},~~~U_{22}=\epsilon_{22}+(X_1-x_1)c\, \bar{\epsilon}_{,1}.
\end{equation*}
Since we are in 2D, we need not consider the out-of-plane $U_{ij}$ components (since $dx_{3}=0$).

From Eq.~\ref{Ii}, the first C\'esaro integral around the $r^{\text{th}}$ hole is
\begin{equation*}
F_1=\oint_{S_r} \left( U_{11} \,  dx_1 + U_{12} \,  dx_2\right) .
\end{equation*}
Substituting the expressions for $U_{ij}$ from the preceding equation,
\begin{equation}
\begin{aligned}
F_1&=\oint_{S_r} \left[ \left\{\epsilon_{11}+(X_2-x_2)c\, \bar{\epsilon}_{,2} \right\}\,  dx_1 + \left\{\epsilon_{12}-(X_2-x_2)c\, \bar{\epsilon}_{,1}\right\} \,  dx_2\right]\\
 &=\oint_{S_r} \left(\epsilon_{11}\,  dx_1 + \epsilon_{12}\,  dx_2\right) + c \oint_{S_r} \left\{(X_2-x_2)\,\left( \bar{\epsilon}_{,2} \,  dx_1 - \bar{\epsilon}_{,1}\,  dx_2\right) \right\}.
\end{aligned}
\label{kamini}
\end{equation}
Consider the first of the above integrals
\begin{equation*}
F_{1a}=\oint_{S_r} \left(\epsilon_{11}\,  dx_1 + \epsilon_{12}\,  dx_2\right). 
\end{equation*}
Using the stress-strain relations (Eqs.~\ref{ss}) in the above, we obtain
\begin{equation*}
\begin{aligned}
F_{1a}&=\oint_{S_r} \left\{\left(1-\nu^2\right)\frac{\sigma_{11}}{Y}\,  dx_1 -\nu(1+\nu)\frac{\sigma_{22}}{Y}\,dx_1+ (1+\nu)\frac{\sigma_{12}}{Y}\,  dx_2\right\}\\
&=\frac{1+\nu}{Y}\oint_{S_r} \left\{\left(1-\nu\right)\sigma_{11}\,  dx_1 -\nu \sigma_{22}\, dx_1+ \sigma_{12}\,  dx_2\right\}.
\end{aligned} 
\end{equation*}
Since $S_r$ can be {\em any} loop surrounding the $r^{\text{th}}$ hole, we consider it to be the surface of the hole itself, as shown in the left panel of Figure \ref{multi}. If the length of the infinitesimal arc element along the hole surface is $ds$, and the unit tangent vector to the surface is $\bm{t}$, then $dx_1 = t_1\,ds$ and $dx_2=t_2\, ds$ (bottom-right panel in Figure \ref{multi}). Accordingly,
\begin{equation}
F_{1a}=\frac{1+\nu}{Y}\oint_{S_r} \left\{\left(1-\nu\right)\sigma_{11}\,  t_1- \nu \sigma_{22}\, t_1+ \sigma_{12}\,  t_2\right\}\,ds. 
\label{tanha}
\end{equation}
If the traction on the hole is $\bm{\tau}_r$,  then
\begin{equation*}
\sigma_{11}n_1+\sigma_{12}n_2 = \tau_{r1},~~~\sigma_{12}n_1+\sigma_{22}n_2 = \tau_{r2}.
\end{equation*}
Since $\bm{n}$ and $\bm{t}$ are orthonormal, $n_1=-t_2$ and $n_2=t_1$. Substituting these in the above equation, we obtain
\begin{equation*}
-\sigma_{11}t_2+\sigma_{12}t_1 = \tau_{r1},~~~-\sigma_{12}t_2+\sigma_{22}t_1 = \tau_{r2}.
\end{equation*}
Using the second of above, we eliminate $\sigma_{12}$ from Eq.~\ref{tanha} to obtain
\begin{equation}
\begin{aligned}
F_{1a}&=\frac{1+\nu}{Y}\oint_{S_r} \left\{\left(1-\nu\right)\sigma_{11}\,  t_1- \nu \sigma_{22}\, t_1+ \sigma_{22}\,t_1-\tau_{r2}\right\}\,ds\\
& = \frac{(1+\nu)(1-\nu)}{Y}\oint_{S_r} \bar{\sigma}\, t_1\, ds - \frac{1+\nu}{Y}\oint_{S_r} \tau_{r2}\, ds\\
& = c\oint_{S_r} \bar{\epsilon}\, t_1\, ds - \frac{1+\nu}{Y}\oint_{S_r} \tau_{r2}\, ds\\
&=c\oint_{S_r} \bar{\epsilon}\, dx_1 - \frac{1+\nu}{Y}\oint_{S_r} \tau_{r2}\, ds. 
\end{aligned}
\label{f1a}
\end{equation}
Now consider the second integral in Eq.~\ref{kamini},
\begin{equation}
F_{1b}=c\oint_{S_r} (X_2-x_2)\,\left( \bar{\epsilon}_{,2} \,  dx_1 - \bar{\epsilon}_{,1}\,  dx_2\right) .
\label{rah}
\end{equation}
Recall that minimization of trace energy yielded a Lagrange multiplier $\bm{\mu}$ satisfying (see Eq.~\ref{mali})
\begin{equation*}
\bar{\sigma}=\frac{Y\bar{\epsilon}}{(1-2\nu)(1+\nu)} = -\frac{\mu_{1,1}}{2}=-\frac{\mu_{2,2}}{2},~~~ \text{and}~~~ \mu_{1,2}+\mu_{2,1}=0.
\end{equation*}
Consider $\bar{\epsilon}_{,2}$. Using the above relations, we have
\begin{equation*}
\frac{-2Y\bar{\epsilon}_{,2}}{(1-2\nu)(1+\nu)}=\frac{(\mu_{1,1}+\mu_{2,2})_{,2}}{2}=\frac{\mu_{1,12}+\mu_{2,22}}{2}=\frac{\mu_{1,12}+\mu_{1,12}}{2}=(\mu_{1,2})_{,1}.
\end{equation*}
Note that the existence of $\bm{\mu}$ from Eq.~\ref{siva} guarantees that the order of the derivatives can be interchanged as above. Similarly,
\begin{equation*}
\begin{split}
\frac{-2Y\bar{\epsilon}_{,1}}{(1-2\nu)(1+\nu)}=\frac{(\mu_{1,1}+\mu_{2,2})_{,1}}{2}=\frac{\mu_{1,11}+\mu_{2,21}}{2}=\frac{\mu_{2,21}+\mu_{2,21}}{2}\\
=(\mu_{2,1})_{,2}=-(\mu_{1,2})_{,2}.
\end{split}
\end{equation*}
Letting $\displaystyle \frac{-2Y}{(1-2\nu)(1+\nu)}=\frac{1}{k}$, we have
\begin{equation*}
\bar{\epsilon}_{,2} \,  dx_1 - \bar{\epsilon}_{,1}\,  dx_2 = k\left\{ (\mu_{1,2})_{,1} \,  dx_1 + \mu_{1,2})_{,2} \,  dx_2 \right\}=k d\left( \mu_{1,2}\right).
\end{equation*}
Using the above, we write the integrand in Eq.~\ref{rah} using the chain rule of differentiation as
\begin{equation*}
(X_2-x_2)\,\left( \bar{\epsilon}_{,2} \,  dx_1 - \bar{\epsilon}_{,1}\,  dx_2\right) = k (X_2-x_2) d\left( \mu_{1,2}\right) = k d\left\{(X_2-x_2) \left( \mu_{1,2}\right)\right\}+k\mu_{1,2}d{x_2},
\end{equation*}
where we have used the fact that $\bm{X}$ is a fixed point. Substituting this in Eq.~\ref{rah}, we obtain
\begin{equation}
F_{1b}=ck\oint_{S_r} d\left\{(X_2-x_2) \left( \mu_{1,2}\right)\right\} + ck  \oint_{S_r} \mu_{1,2}d{x_2}=ck  \oint_{S_r} \mu_{1,2}d{x_2}.
\label{F2fin}
\end{equation}
Finally, using Eqs.~\ref{f1a} and \ref{F2fin}, we have
\begin{equation*}
\begin{aligned}
F_{1}&=F_{1a}+F_{1b}= c\oint_{S_r} \bar{\epsilon}\, dx_1 - \frac{1+\nu}{Y}\oint_{S_r} \tau_{r2}\, ds + ck  \oint_{S_r} \mu_{1,2}d{x_2} \\
&= ck\oint_{S_r} \left(\frac{\mu_{1,1}+\mu_{2,2}}{2}\right)\, dx_1 + ck  \oint_{S_r} \mu_{1,2}d{x_2}- \frac{1+\nu}{Y}\oint_{S_r} \tau_{r2}\, ds\\
&=ck\oint_{S_r} \mu_{1,1}\, dx_1 + ck  \oint_{S_r} \mu_{1,2}d{x_2}- \frac{1+\nu}{Y}\oint_{S_r} \tau_{r2}\, ds\\
& = ck\oint_{S_r} d\left(\mu_{1}\right)- \frac{1+\nu}{Y}\oint_{S_r} \tau_{r2}\, ds\\
&=- \frac{1+\nu}{Y}\oint_{S_r} \tau_{r2}\, ds. 
\end{aligned}
\end{equation*}
We can similarly show that 
$$F_2=ck\oint_{S_r} d\left(\mu_{2}\right)- \frac{1+\nu}{Y}\oint_{S_r} \tau_{r1}\, ds=- \frac{1+\nu}{Y}\oint_{S_r} \tau_{r1}\, ds.$$
If we assume that the net force corresponding to $\bm{\tau}_r$ is zero, i.e.,
\begin{equation*}
\oint_{S_r} \bm{\tau}_{r}\, ds = \bm{0},
\end{equation*}
we have $F_1=F_2=0.$

Thus, we have shown that the minimizer of the trace energy satisfies the global compatibility of strain if the net force on each hole is zero. Since it also satisfies the local strain compatibility, and domain and boundary equilibria, it must be the true stress. This completes the proof of Theorem \ref{PTP}.
\end{proof}

The planar trace principle is new to the best of our knowledge. It is to be noted that even for homogeneous isotropic elastic bodies, there are two independent elastic constants, say Young's modulus and Poisson's ratio. The remarkable thing about this new variational principle is that it is not merely independent of Young's modulus, which can be scaled out, but it is also independent of Poisson's ratio.

We now summarize the two principles derived in this section. \\
\noindent (a) \underline{Strain energy principle}: In any linear elastic stress problem, the minimizer of the strain energy over the set of all stresses in equilibrium with the loading is the solution.\\
\noindent (b) \underline{Planar trace principle}: In the special case of planar homogeneous isotropic bodies, the minimizer of the trace energy of Eq.\ (\ref{L2e}) over the same set of stresses is the solution for simply-connected bodies; the same is true for a multiply-
 body too provided the net force on each hole is individually zero.

In the next section, we use these variational principles to develop our method of solving stress problems in linear elasticity.

\section{Solution methods} \label{formulation}
We now present our approach to solving the following problem. A planar linear elastic body $\mathcal{B}$ is subjected to a self-equilibrated traction $\bm{\tau}$ prescribed on the whole boundary. We wish to find a sequence $(\bm{\sigma}^N)$ of approximate stresses converging to the true stress $\bm{\sigma}$ in the strain energy norm \cite{steigmann}
\begin{equation}
\norm{\bm{\sigma}}_{\mathcal{E}}=\left(\int_{\Omega} \bm{C}^{-1}\bm{\sigma} \cdot \bm{\sigma}\, dA\right)^{0.5}.
\label{senorm}
\end{equation}

We begin by finding any one stress $\bm{\sigma}_p$ in equilibrium with $\bm{\tau}$, without regard to strain compatibility, i.e., $\bm{\sigma}_p$ satisfies
\begin{equation}
\begin{aligned}
\text{div}\, \bm{\sigma}_p &= \bm{0} ~~~ \text{in} ~~~ \Omega, \\
\bm{\sigma}_p \bm{n} &= \bm{\tau} ~~~ \text{on} ~~~ \partial\Omega.
\end{aligned}
\label{BVPsigmap}
\end{equation}
Note that there are infinitely many $\bm{\sigma}_p$ satisfying Eqs.~\ref{BVPsigmap}. For simple geometries and surface traction distributions, it may be easy to obtain a candidate $\bm{\sigma}_p$. As an example of finding such a $\bm{\sigma}_p$ which does not satisfy strain compatibility, consider the square block shown in the left panel of Figure \ref{Demo}, occupying $0\leq x\leq 1$, $0\leq y \leq 1$. It is subjected to constant pressure $p$ on the middle halves of the top and bottom edges. An obvious candidate $\bm{\sigma}_p$ is
\begin{equation}
\displaystyle
\sigma_{xxp}=0, ~~~ \sigma_{xyp}=0, ~~~\\
\sigma_{yyp}=
\begin{cases}
-p, ~~~~ 1/4\leq x \leq 3/4,\\
0, ~~~~~~\, \text{otherwise.}
\end{cases}
\label{casesintro}
\end{equation}
A color plot of $\sigma_{yyp}$ is shown in the right panel of Figure \ref{Demo}. 
\begin{figure}[t!]
    \begin{subfigure}
	\centering
	\hbox{\hspace{1cm}}\includegraphics [width=0.35\linewidth]{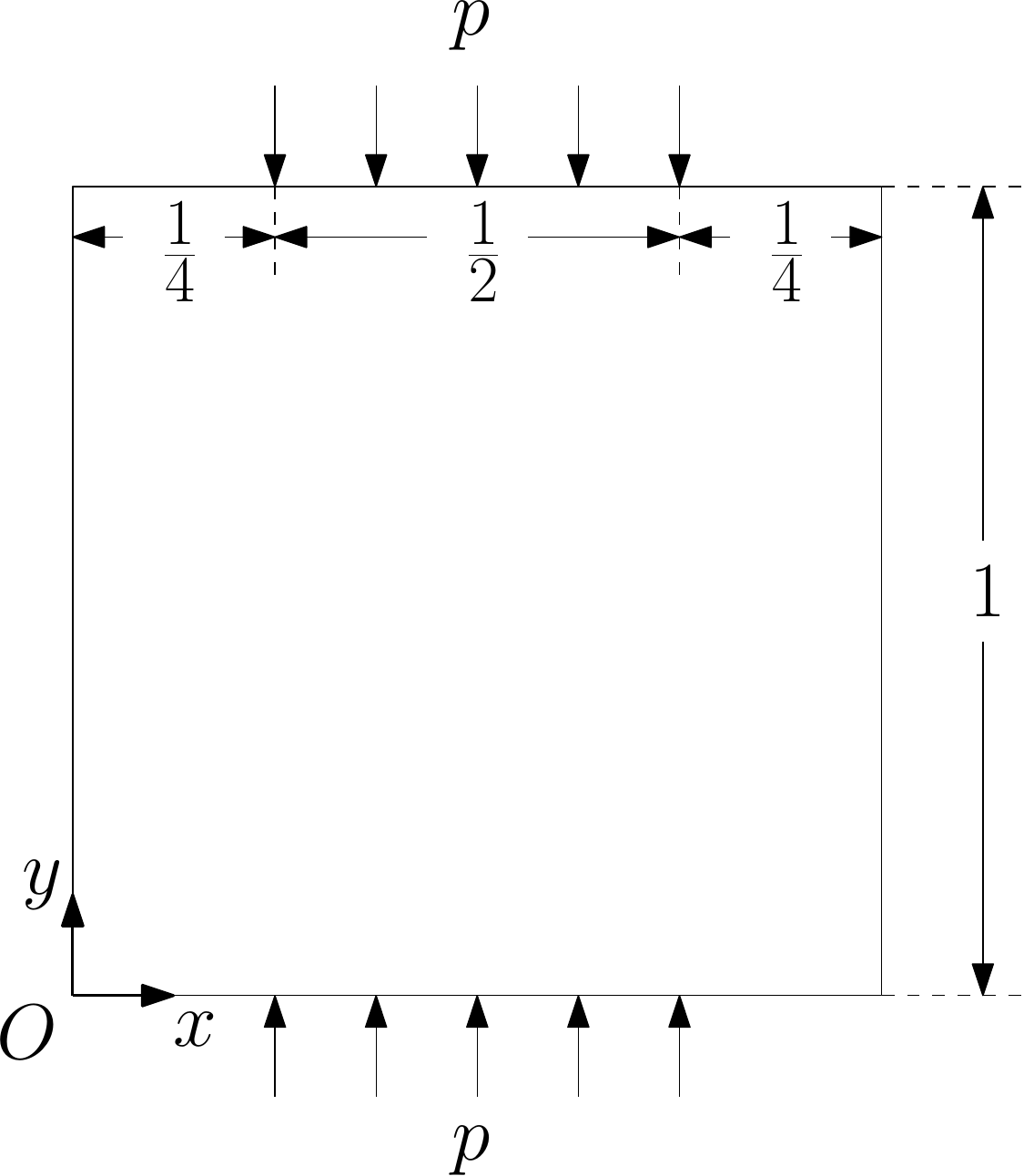}
    \end{subfigure}
    \begin{subfigure}
	\centering
	\hbox{\hspace{1.5cm}}\includegraphics [width=0.5\linewidth]{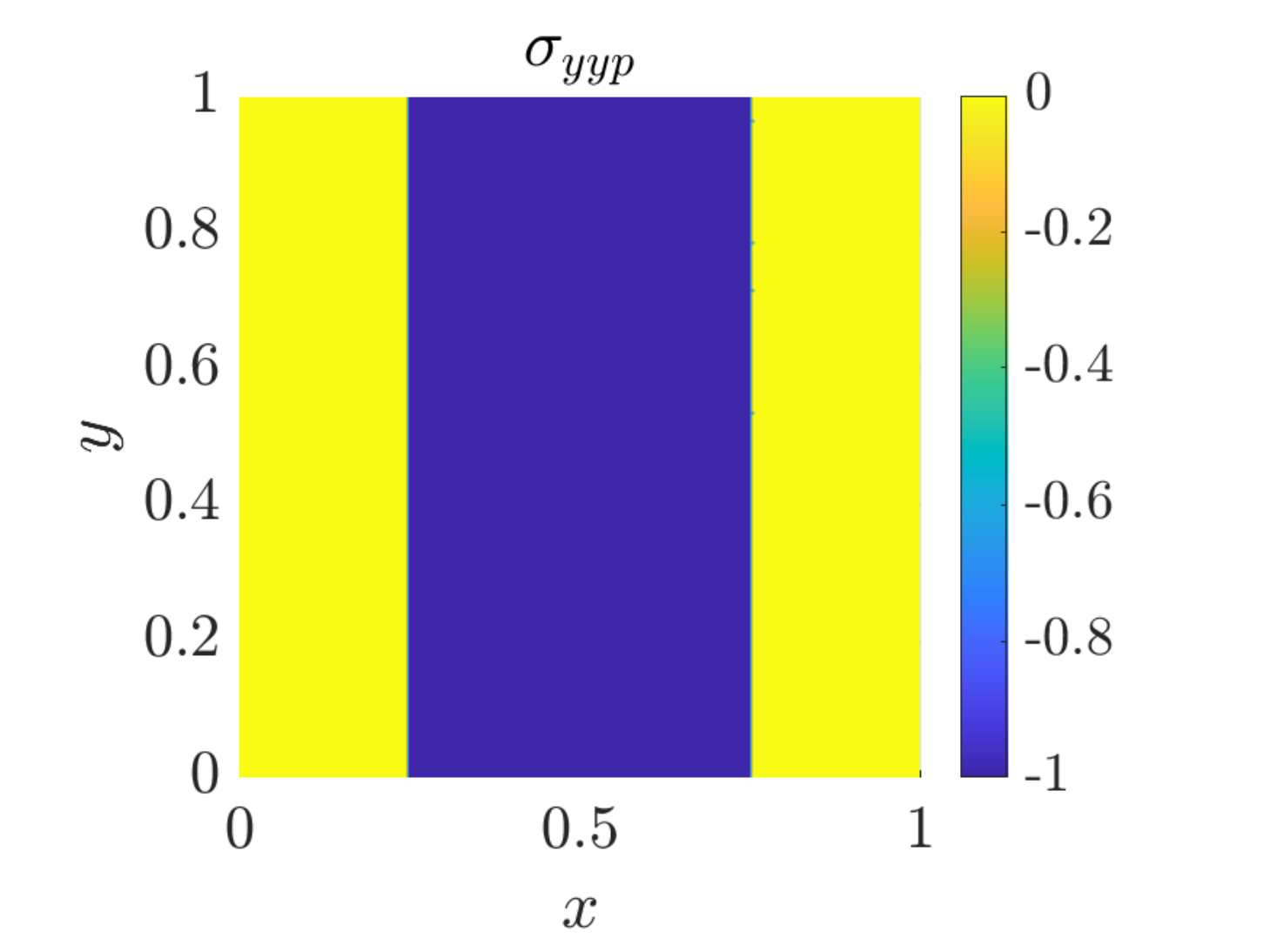}
    \end{subfigure}
\caption{An example where finding a $\bm{\sigma}_p$ is easy. Left: problem statement; right: color plot of $\sigma_{yyp}.$ Note that although $\sigma_{yyp}$ is discontinuous, $\text{div}\, \bm{\sigma}_p$ exists since it involves only the $y$ derivative of $\sigma_{yyp}$.}
\label{Demo}
\end{figure}

A more analytical way to find a $\bm{\sigma}_p$ is to assume a specific functional form of an Airy stress function $\psi_p$ with free parameters fitted to satisfy
traction boundary conditions. Note that we need only {\em one} such Airy stress function, and not a whole basis: we already have a basis of tensor-valued stress fields. We will demonstrate this for axisymmetric examples in the next section. We now proceed by assuming that a $\bm{\sigma}_p$ has been obtained. The primary academic contributions of this paper relate to the calculations undertaken after a $\bm{\sigma}_p$ has been obtained.

Recall that $\mathcal{Q}$ of Eq.\ (\ref{edefQ}) is the set of all stress fields in equilibrium with the traction $\bm{\tau}$. Therefore, if $\bm{\sigma}_p$ is subtracted from each element of $\mathcal{Q}$, we obtain the set 
\begin{equation*}
\mathcal{Q}_R=\left\{\bm{\sigma} \big| ~ \bm{\sigma}\in \,\text{Sym},~\int_{\Omega} \bm{\sigma}\cdot \bm{\sigma}\,dA<\infty,~\text{div}\,\bm{\sigma}=\bm{0}, ~ \left.\bm{\sigma n}\right|_{\partial \Omega}=\bm{0} \right\}
\end{equation*}
of all square-integrable self-equilibrated traction-free stress fields on $\Omega$. Such stress fields were dealt with in detail in \cite{tiwari}, where they were called {\em residual stresses.} The primary academic contribution of \cite{tiwari} is summarized below.

In \cite{tiwari}, we proved that the solutions to the eigenvalue problem
\begin{equation}
\begin{array}{cccl}
-\Delta \bm{\sigma} + \nabla_s \bm{\mu}  = \lambda \bm{\sigma} & \text{ and } & \textup{div} \, \bm{\sigma} = \bm{0} & \text{ in} \hspace{2mm} \Omega,\\
\bm{\sigma n} = \bm{0} & \text{ and } & \nabla_n \bm{\sigma} \cdot (\bm{t} \otimes \bm{t}) = 0 & \text{ on} \hspace{2mm} \partial \Omega
\end{array}
\label{three_eqns}
\end{equation}
form a basis in the $L^2$ norm for the set of all residual stresses of interest to us.  The definitions of $\nabla_s \bm{\mu}$, $\lambda$ and $\nabla_n \bm{\sigma} \cdot (\bm{t} \otimes \bm{t})$, along with a brief description of the minimization problem that yields the above eigenvalue problem, are presented in Appendix \ref{abrief}. 

The eigenfunctions $\bm{\phi}_i$ of Eq.~\ref{three_eqns} are orthogonal in the $L^2$ inner product. They can be computed on arbitrary geometries using the finite element method, and we have presented $\bm{\phi}_i$ computed on an annulus and a rectangle (taken from \cite{tiwari}) in Appendix \ref{abrief}.

We now return to the development of our solution method. Since $\bm{\phi}_i$ span $\mathcal{Q}_R$, each element of $\mathcal{Q}_R$ can be written as $\displaystyle \lim_{N\to \infty} \sum_{j=1}^N a_j \bm{\phi}_j$ for some real-valued sequence $(a_j)$. Consequently, the problem of minimization of $\mathcal{E}$ over $\mathcal{Q}$ is equivalent to finding the real-valued sequence $(a_i)$ that minimizes 
\begin{equation}
\hat{\mathcal{E}}(a_1,a_2,\cdots)=\displaystyle \int_{\Omega} \bm{C}^{-1} \left(\bm{\sigma}_p+\sum_{j=1}^{\infty} a_j \bm{\phi}_j\right)\cdot \left(\bm{\sigma}_p+\sum_{k=1}^{\infty} a_k \bm{\phi}_k\right)\, dA.
\label{hatE}
\end{equation}
Similarly, the problem of minimizing $\mathcal{T}$ over $\mathcal{Q}$ is equivalent to finding the real-valued sequence $(a_i)$ that minimizes 
\begin{equation}
\label{that}
\hat{\mathcal{T}}(a_1,a_2,\cdots)=\displaystyle \int_{\Omega} \left(\bar{\sigma}_p+\sum_{j=1}^{\infty} a_j \bar{\phi}_j\right)^2\, dA. 
\end{equation}
We now compute the coefficients $a_j$. We first consider general linear elastic bodies. Subsequently, we consider planar homogeneous isotropic linear elastic bodies. For such bodies, the new minimum principle holds, and calculations are simpler.

\subsection{Method based on the strain energy principle}
For general linear elastic bodies, the solution $\bm{\sigma}$ corresponds to the real-valued sequence $(a_i)$ that minimizes $\hat{\mathcal{E}}$ (Eq.~\ref{hatE}). Accordingly, we set the gradient of $\hat{\mathcal{E}}$ with respect to $a_i$ to zero, and obtain 
\begin{equation}
\int_{\Omega} \left(\sum_{j=1}^{\infty} a_j \bm{C}^{-1}\bm{\phi}_j  \right)\cdot \bm{\phi}_i \, dA =-\int_{\Omega} \bm{C}^{-1}\bm{\sigma}_p \cdot \bm{\phi}_i \, dA ~~~ \forall ~i.
\label{gleb}
\end{equation}
Truncating the expansion in the above equation to $N$ terms, and considering $1\leq i \leq N$, we set up a linear system of $N$ equations
\begin{equation}
Ma=f,
\label{simpMaf}
\end{equation}
where
\begin{equation*}
M(i,j)=\int_{\Omega}  \bm{C}^{-1} \bm{\phi}_j \cdot \bm{\phi}_i \, dA ~~~~~ \text{and} ~~~~~ f(i)=-\int_{\Omega} \bm{C}^{-1} \bm{\sigma}_p \cdot \bm{\phi}_i \, dA.
\end{equation*}
Using the above expressions, we compute the coefficient vector $a=M^{-1}f$ and find an approximate stress 
$$\bm{\sigma}^N = \bm{\sigma}_p + \sum_{i=1}^N a_i \bm{\phi}_i. $$
Note that in general, $M$ is {\em not} a diagonal matrix\footnote{Alternatively, by changing the normalization constraint in the original minimization problem, the eigenfunctions $\bm{\phi}_i$ can be made to satisfy the orthogonality conditions $\int_{\Omega} \bm{C}^{-1}\bm{\phi}_j \cdot \bm{\phi}_i\, dA=\delta_{ij}$, where $\delta_{ij}$ is the Kronecker-delta function. $M$ will then be an identity matrix, and $a=f$. However, such basis functions will depend on $\bm{C}$ and will have to be computed afresh each time $\bm{C}$ is changed.}. Also, note that the major symmetry of $\bm{C}^{-1}$ implies that $M$ is a symmetric matrix.

The true stress
\begin{equation*}
\bm{\sigma}=\lim_{N\to \infty} \bm{\sigma}^N,
\end{equation*}
where the convergence is in the strain energy norm\footnote{\label{fn1} Since the strain energy and $L^2$ norms are equivalent \cite{steigmann}, $\bm{\sigma}\in L^2$. Moreover, since $\bm{\phi}_i$ span $\mathcal{Q}_R$ in the $L^2$ norm \cite{tiwari}, it follows that $(\bm{\sigma}^N)$ converges to $\bm{\sigma}$ in the $L^2$ norm. Finally, the equivalence of the strain energy and $L^2$ norms implies that $(\bm{\sigma}^N)$ converges to $\bm{\sigma}$ in the strain energy norm.}. Further, within $\bm{\sigma}=\bm{\sigma}_p+\bm{\sigma}_h$,
\begin{equation*}
\bm{\sigma}_h=\lim_{N\to \infty} \sum_{j=1}^{N} a_j \bm{\phi}_j.
\end{equation*} 

We now turn to the special case of planar homogeneous isotropic bodies for which the planar trace principle holds.

\subsection{Method based on the planar trace principle}
For homogeneous isotropic bodies, we will minimize $\hat{\mathcal{T}}$ of Eq.\ (\ref{that}). Setting the gradient of $\hat{\mathcal{T}}$ with respect to $a_i$ to zero, we have
\begin{equation*}
\int_{\Omega} \left(\sum_{j=1}^{\infty} a_j \bar{\phi}_j  \right) \bar{\phi}_i \, dA =-\int_{\Omega} \bar{\sigma}_p \bar{\phi}_i \, dA ~~~ \forall ~i.
\end{equation*}
Truncating the expansion in the above equation to $N$ terms, and considering $1\leq i \leq N$, we obtain a linear system of $N$ equations
\begin{equation}
Ma=f,
\label{simpMaf}
\end{equation}
where
\begin{equation*}
M(i,j)=\int_{\Omega} \bar{\phi}_i \bar{\phi}_j \, dA ~~~~~ \text{and} ~~~~~ f(i)=-\int_{\Omega} \bar{\sigma}_p \bar{\phi}_i \, dA.
\end{equation*}
The above expressions can be simplified as follows. In \cite{tiwari}, we showed that the residual stress basis functions $\bm{\phi}_i$ are orthogonal in the $L^2$ inner product. That orthogonality also holds for residual stress planar traces $\bar{\phi}_i$, as follows.
\begin{theorem} \label{ipTraces}
The $L^2$ inner product of any two planar residual stresses $\bm{\sigma}_1$ and $\bm{\sigma}_2$ is equal to the $L^2$ inner product of their planar traces, i.e.,
$$ \int_{\Omega} \bm{\sigma}_1 \cdot \bm{\sigma}_2 \, dA = \int_{\Omega} \bar{\sigma}_1\, \bar{\sigma}_2 \, dA.$$
\end{theorem}
\begin{proof}
See Appendix \ref{ip}.
\end{proof}
Since the $\bar{\phi}_i$ are orthogonal in the $L^2$ inner product, we adjust their norms so that they become {\em orthonormal}. As a result, $M$ becomes an $N \times N$ identity matrix, and
\begin{equation}
a=f.
\label{af}
\end{equation}
The corresponding approximate solution is
\begin{equation*}
\bm{\sigma}^N= \bm{\sigma}_p + \sum_{j=1}^{N} a_j \bm{\phi}_j =\bm{\sigma}_p + \sum_{j=1}^{N} \left(- \int_{\Omega} \bar{\sigma}_p \bar{\phi}_j \, dA\right) \bm{\phi}_j,
\end{equation*}
and the true stress
\begin{equation*}
\bm{\sigma}=\lim_{N\to \infty} \bm{\sigma}^N,
\end{equation*}
where the convergence is in the strain energy norm (see Footnote \ref{fn1}).
Here, within $\bm{\sigma}=\bm{\sigma}_p+\bm{\sigma}_h,$
\begin{equation*}
\bm{\sigma}_h=\lim_{N\to \infty} \sum_{j=1}^{N} \left(- \int_{\Omega} \bar{\sigma}_p \bar{\phi}_j \, dA\right) \bm{\phi}_j.
\end{equation*}

This calculation, based on the trace of the stress, is more straightforward because the coefficient matrix need not be inverted. We can compute each coefficient in the expansion independently. However, this calculation is restricted to planar homogeneous isotropic elastic bodies.

\begin{figure}[t!]
	\centering
	\includegraphics [width=0.75\linewidth]{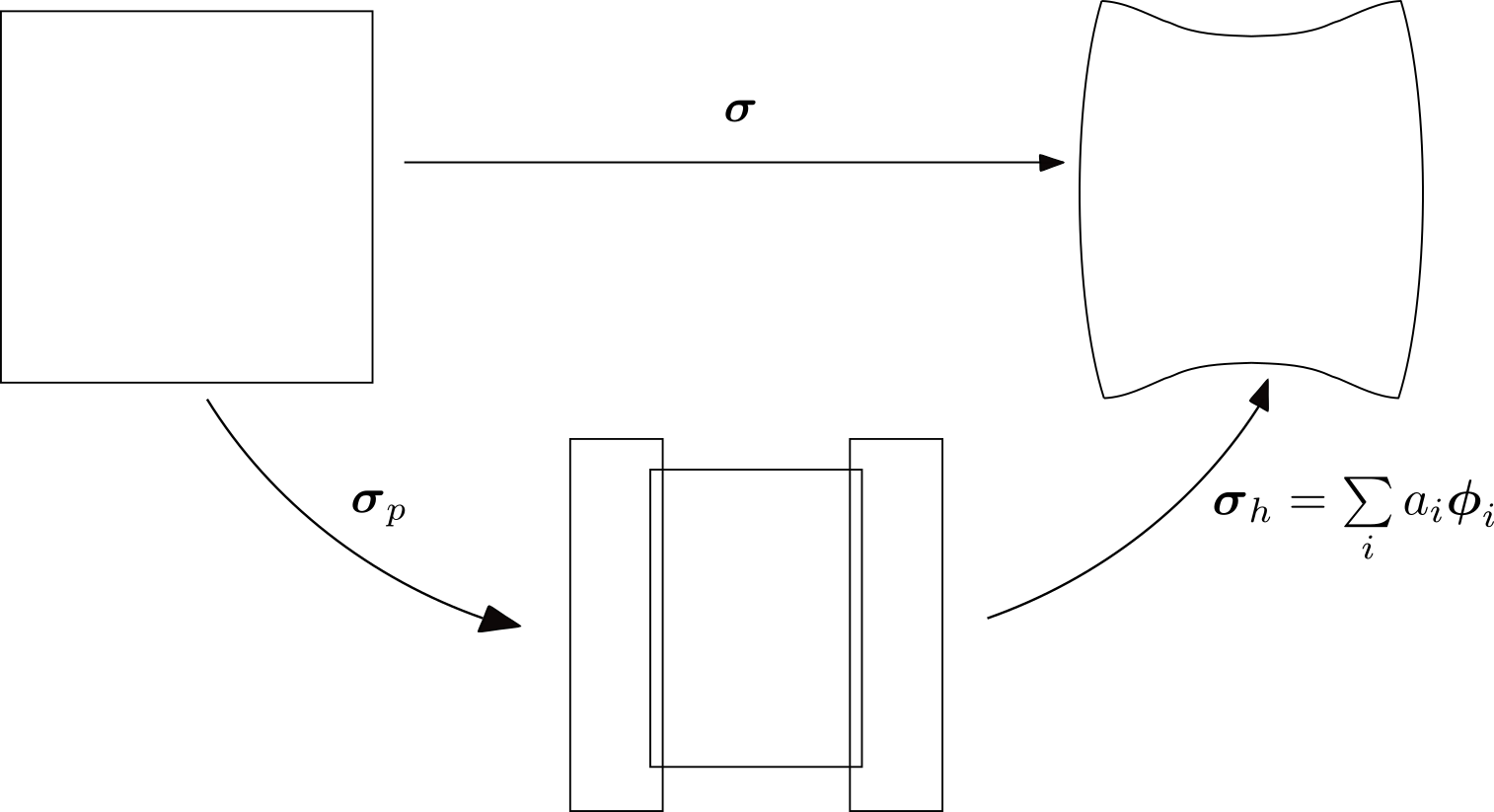}
\caption{A geometric interpretation of the splitting $\bm{\sigma}=\bm{\sigma}_p + \bm{\sigma}_h$ for the problem of the pressurized square block in Figure \ref{Demo}. Though $\bm{\sigma}_p$ and $\bm{\sigma}_h$ are individually incompatible, their sum gives the correct solution.}
\label{decomposition}
\end{figure}
A geometric interpretation of the splitting $\bm{\sigma}=\bm{\sigma}_p + \bm{\sigma}_h$ is given in Figure \ref{decomposition}. Note that in the methods using the principle of minimum complementary energy, a similar splitting of the Airy stress function into a traction-satisfying part $\psi_p$ and a traction-free part $\psi_h$, and a subsequent expansion of $\psi_h$ in suitable global functions has been employed in prior works (e.g., see Section 51 in \cite{timoshenko}). However, analytically obtaining such expansion functions is possible only for special geometries such as rectangles or annuli. Our stress basis functions $\bm{\phi}_i$, being computationally oriented, apply through finite element analysis to {\em any} sufficiently-regular geometry \cite{tiwari}. Moreover, for multiply-connected bodies, the Airy stress function exists only when the net force {\em and} moment on each hole is individually zero \cite{gurtin}, while our method based on the strain energy principle has no such restrictions. The trace energy principle does not work if the net force on each hole is not zero, but it does not require that the net moment on each hole also be zero.

We close this section by summarizing our solution methods. To find a sequence of approximate stresses $(\bm{\sigma}^N)$ in a linear elastic body which is acted upon by a traction $\bm{\tau}$, we first find {\em any} (usually incompatible\footnote{If $\bm{\sigma_p}$ is in equilibrium with $\bm{\tau}$ and is compatible, then it {\em is} the true stress.}) stress $\bm{\sigma}_p$ in equilibrium with $\bm{\tau}$. Then from the strain energy principle, the coefficients in
\begin{equation}
\bm{\sigma}^N=\bm{\sigma}_p+\sum_{j=1}^{N} a_j \bm{\phi}_j
\label{hetero}
\end{equation}
are given by $a=M^{-1}f$ with
\begin{equation}
\begin{aligned}
M(i,j)=\int_{\Omega}  \bm{C}^{-1} \bm{\phi}_j \cdot \bm{\phi}_i \, dA & ~~~~~ \text{and}\\
 f(i)=-\int_{\Omega} \bm{C}^{-1} \bm{\sigma}_p \cdot \bm{\phi}_i \, dA &, ~~~~~~ 1\leq i \leq N, ~ 1\leq j \leq N.
\end{aligned}
\nonumber
\end{equation}
For the special case of planar homogeneous isotropic bodies, we use the planar trace principle wherein the coefficients are obtained without matrix inversion, and 
\begin{equation}
\bm{\sigma}^N=\bm{\sigma}_p+\sum_{j=1}^{N} \left(- \int_{\Omega} \bar{\sigma}_p \bar{\phi}_j \, dA\right) \bm{\phi}_j,
\label{homo}
\end{equation}
provided the net force on each internal hole in the body is zero.

In the next section, we will use Eqs.~\ref{hetero} and \ref{homo} to solve eight planar stress problems. 

\section{Examples} \label{examples}
\subsection{Homogeneous isotropic bodies}
In this section, we solve six stress problems involving planar homogeneous isotropic bodies obeying linear elasticity. A summary of these examples is listed below. 
\begin{itemize}
\item In the first example, we consider an annulus pressurized uniformly from the inside.
\item In the second example, we study a square block subjected to discontinuous traction.
\item In the third example, we consider an arbitrarily shaped body containing a uniformly pressurized circular hole. 
\item In the fourth example, we extend the planar trace principle to incorporate body forces arising from a potential, and use it to find the stress in a square block resting on a table under gravity.
\item In the fifth example, we find the stress in an annulus subjected to a traction with a chosen azimuthal wavenumber and corresponding to non-zero net force on the hole, with the chosen wavenumber rendering the problem essentially 1D.
\item In the sixth example, we consider a general multiply-connected body having a hole subjected to (1) non-zero net force and zero net moment, (2) zero net force and zero net moment, and (3) zero net force and non-zero net moment. 
\end{itemize}

In the first four examples, we will use the planar trace principle. In the fifth and sixth examples, we will see that when the net force on the hole is non-zero, the planar trace principle does not work; however, the strain energy principle does, consistent with the theory developed in Section \ref{variation}.

All of these six examples correspond to the plane-strain case, with Young's modulus $Y=1$ and Poisson's ratio $\nu=0.33$, in any appropriate units. 

\subsubsection{Example 1: Internally pressurized annulus}
We consider an annulus centered at the origin with inner and outer radii of $r_a=0.1$ and $r_b=0.3$, respectively. It is subjected to pressure $p=1$ on the inner boundary, while the outer boundary is traction-free.

Since the problem is axisymmetric, we begin by {\em constructing} an axisymmetric stress $\bm{\sigma}_p$ that is in equilibrium with this pressure, i.e., $\bm{\sigma}_p$ satisfies
\begin{equation}
\begin{aligned}
\sigma_{rrp}' +\frac{\sigma_{rrp}-\sigma_{\theta\theta p}}{r}=0~~~~~~ &\text{and} ~~~~~~\sigma_{r\theta p}=0 ~~~ \text{in} ~~~ \Omega, \\
\sigma_{rrp}\big|_{r_a}=-1 ~~~ ~~~&\text{and} ~~~~~~ \sigma_{rrp}\big|_{r_b}= 0,\\
\end{aligned}
\label{pressBVPp}
\end{equation}
where prime denotes derivative with respect to $r$. To obtain one such $\bm{\sigma}_p$, we assume the Airy stress function $\psi_p$ corresponding to $\bm{\sigma}_p$ to be
\begin{equation}
\psi_p=c_1 r+c_2 r^2,
\label{airyp}
\end{equation}
where $c_1$ and $c_2$ are free parameters. The resulting stress components are \cite{barber}
\begin{equation}
\sigma_{rrp}=\frac{\psi_p'}{r}=\frac{c_1}{r}+2c_2, ~~~~~ \sigma_{\theta \theta p}=\psi_p''=2c_2 ~~~~~ \text{and} ~~~~~ \sigma_{r\theta p}=0.
\label{stressp}
\end{equation}
We solve for the free parameters using the two boundary conditions in Eqs.~\ref{pressBVPp} and obtain
\begin{equation}
c_1=-\frac{r_a r_b }{r_b-r_a},~~~ c_2=\frac{r_a}{2 \left( r_b-r_a\right)}.
\end{equation}
These constants are substituted in Eq.~\ref{airyp}, which in turn is substituted in Eq.~\ref{stressp} to obtain $\bm{\sigma}_p$. It can be checked that $\bm{\sigma}_p$ does not satisfy strain compatibility, which is why it differs from the true stress $\bm{\sigma}$. We use the planar trace principle and substitute $\bm{\sigma}_p$ in Eq.~\ref{homo} to find, for a given $N$, an approximate stress $\bm{\sigma}^N$. Guided by the axisymmetry of the problem, we only use axisymmetric basis functions $\bm{\phi}_i$ in the expansion.
\begin{figure}[t!]
    \begin{subfigure}
	\centering
	\includegraphics [width=0.5\linewidth]{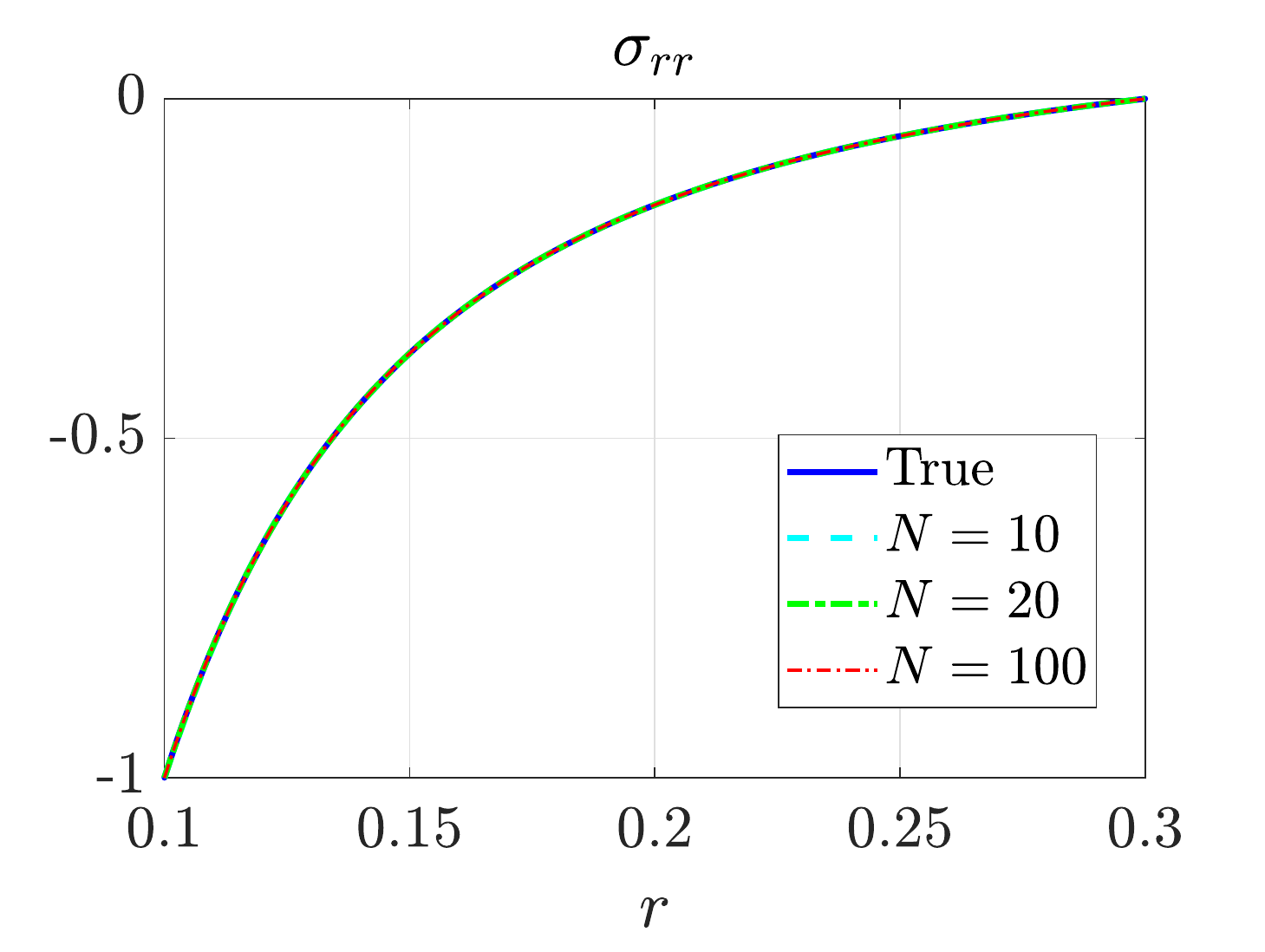}
    \end{subfigure}
    \begin{subfigure}
	\centering
	\includegraphics [width=0.5\linewidth]{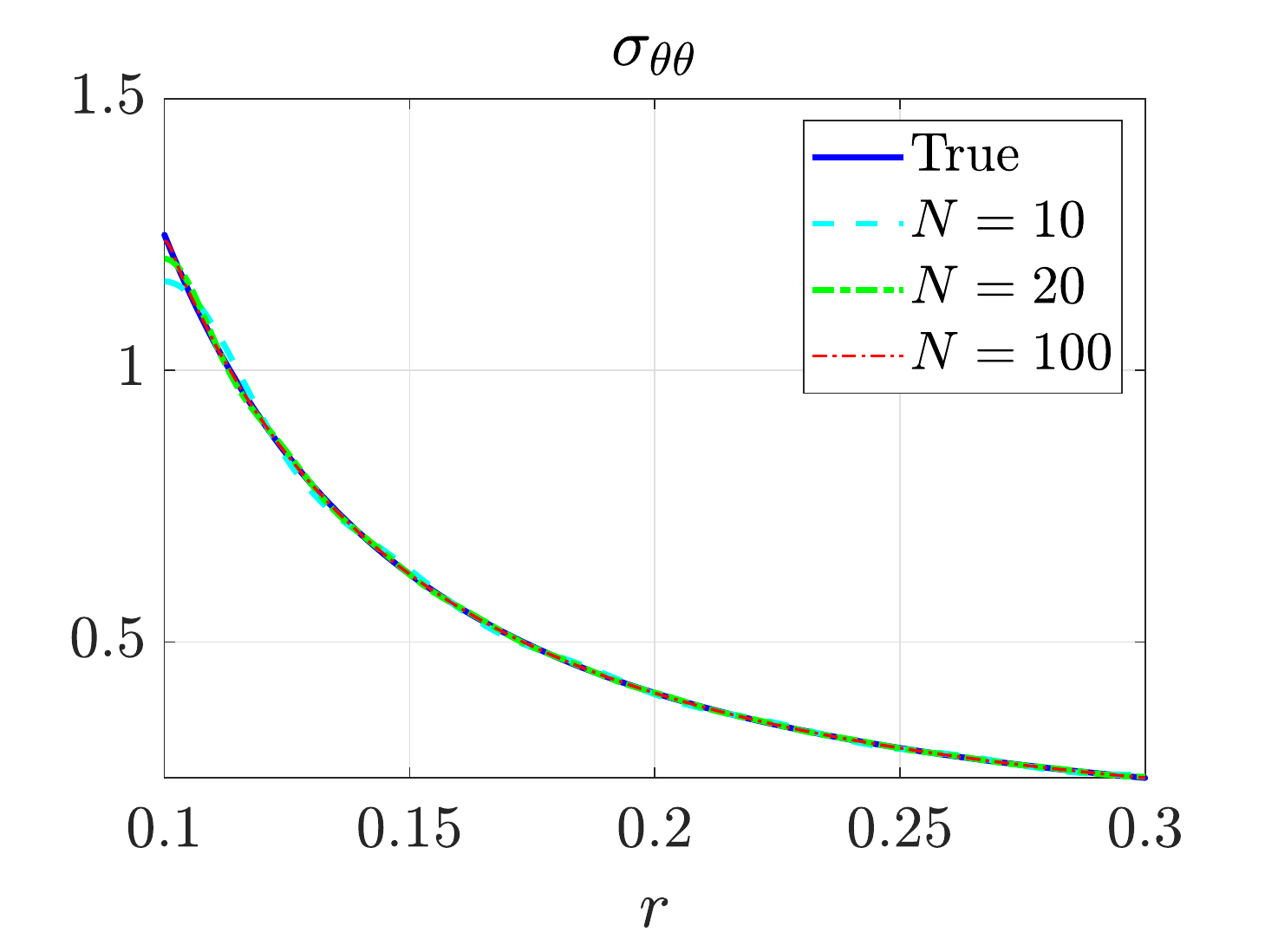}
    \end{subfigure}
\caption{Radial variation of the true and the approximate stresses for the pressurized annulus in Example 1, using 10, 20 and 100 basis functions $\bm{\phi}_i$. Note that $\sigma_{r\theta}=\sigma_{r\theta}^N=0$.}
\label{SEx1}
\end{figure}

We plot the radial variation of the approximations $\bm{\sigma}^N$, along with the true stress $\bm{\sigma}$ (obtained by direct numerical integration of the governing equations; details omitted), in Figure \ref{SEx1} for $N=10, 20$ and 100. We see that the approximations are good, even for $N=10$.
 
To quantify the accuracy of the approximation, we define the approximation error $E_N$ as the strain energy norm of  $\bm{\sigma}-\bm{\sigma}^N$ divided by the strain energy norm of $\bm{\sigma}$, i.e.,
$$E_N=\frac{\norm{\bm{\sigma}-\bm{\sigma}^N}_{\mathcal{E}}}{\norm{\bm{\sigma}}_{\mathcal{E}}}=\left(\frac{\int_{\Omega} \bm{C}^{-1}(\bm{\sigma}-\bm{\sigma}^N)\cdot (\bm{\sigma}-\bm{\sigma}^N)\, dA}{\int_{\Omega} \bm{C}^{-1} \bm{\sigma}\cdot \bm{\sigma}\, dA}\right)^{0.5}.$$
$E_N$ is plotted in Figure \ref{EEx1} in blue, both on linear (left) and log-log (right) scales. Note that the point on the extreme left ($N=0$) in the left panel corresponds to $\bm{\sigma}_p$; a sharp decay in $E_N$ is seen even with the incorporation of just one basis function ($N=1$). The dotted black line near $N=500$ in the log-log plot has a slope of $-1.5$, i.e., it decays like $1/N^{1.5}$ for large $N$, and we find that $E_N$ also decays at approximately the same rate.

In the left panel of Figure \ref{EEx1}, we have also plotted the strain energies $\mathcal{E}_N$ corresponding to the approximate stresses in orange. The dashed horizontal orange line indicates the true strain energy $\mathcal{E}$. It is found that $\mathcal{E}_N$ is within 0.01\% of $\mathcal{E}$ with just $N=10$ basis functions.
\begin{figure}[t!]
    \begin{subfigure}
	\centering
	\includegraphics [width=0.5\linewidth]{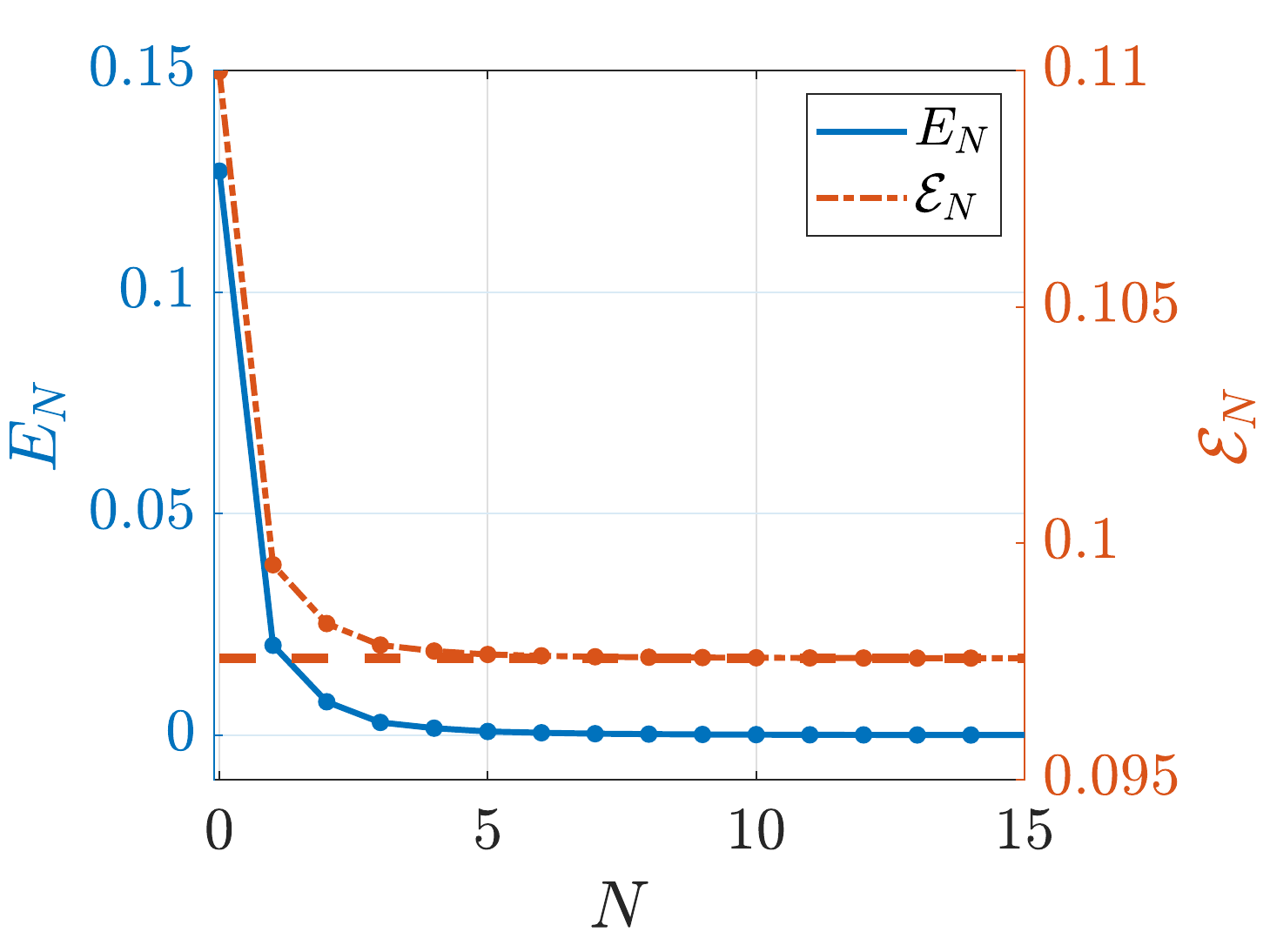}
    \end{subfigure}
    \begin{subfigure}
	\centering
	\includegraphics [width=0.5\linewidth]{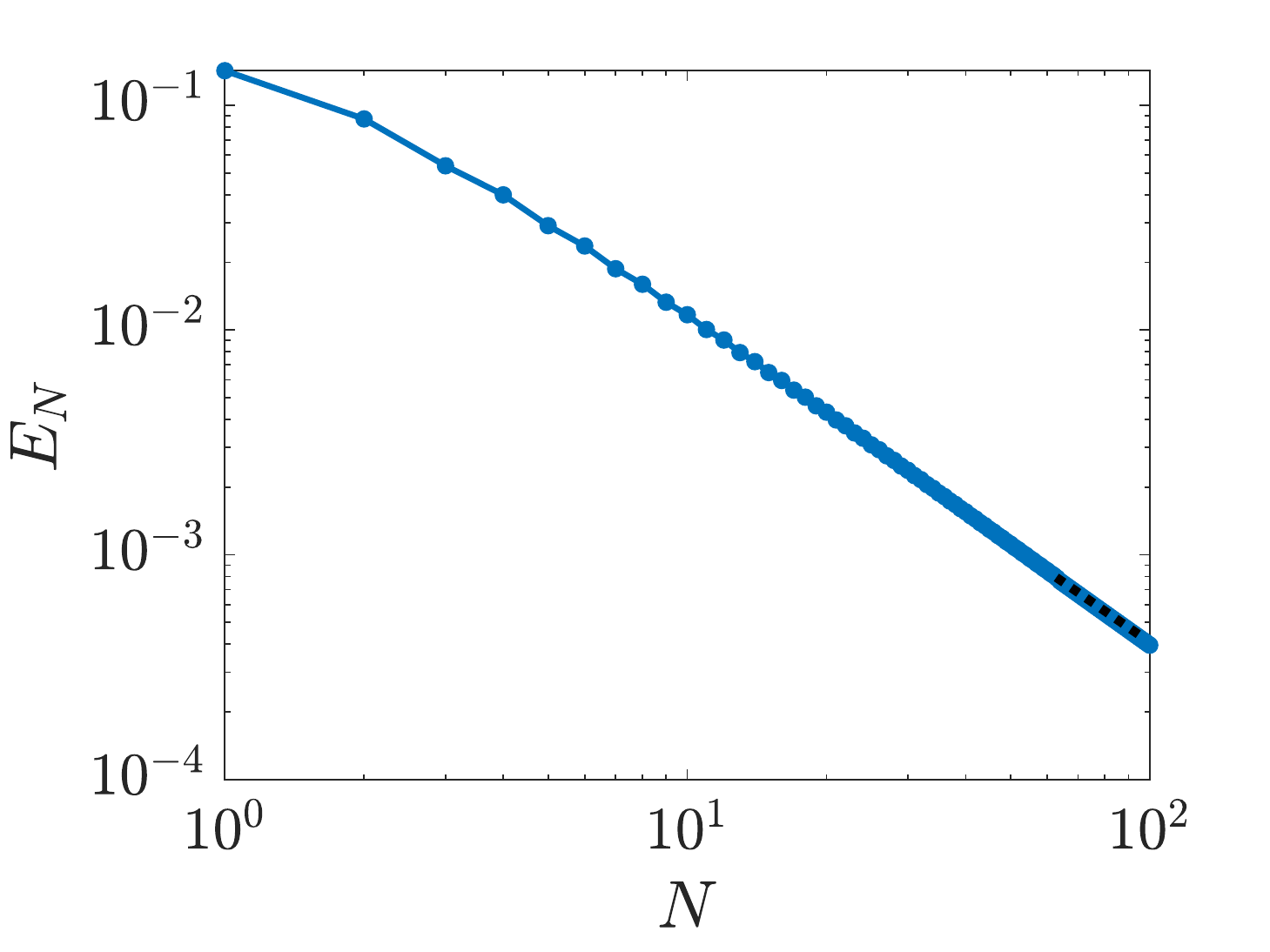}
    \end{subfigure}
\caption{Left: approximation error $E_N$ (in blue) and strain energy corresponding to approximate stress $\mathcal{E}_N$ (in orange) versus $N$ for Example 1; the dashed horizontal orange line corresponds to the strain energy of the true stress. Right: $E_N$ on a log-log scale.}
\label{EEx1}
\end{figure}

In Figure \ref{split}, we plot the radial variation of all three of $\bm{\sigma}$, $\bm{\sigma}_p$ and $\bm{\sigma}_h^N$ (for $N=100$) to depict graphically the splitting $\bm{\sigma}=\bm{\sigma}_p+\bm{\sigma}_h$. To indicate that the computed homogeneous stress is approximate, we have used the $\approx$ sign.
\begin{figure}[t!] 
	\centering
	\includegraphics [width=1.05\linewidth]{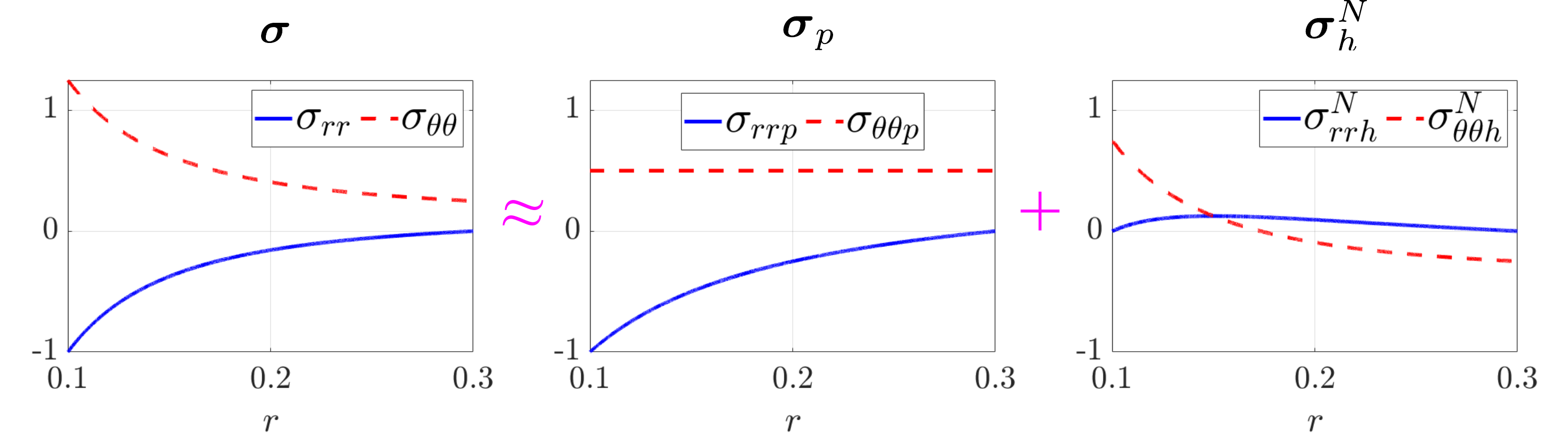}
\caption{Graphical depiction of the splitting $\bm{\sigma}=\bm{\sigma}_p+\bm{\sigma}_h$ for Example 1 (with $N=100$).}
\label{split}
\end{figure}

\subsubsection{Example 2: Homogeneous block subjected to discontinuous pressure} \label{nup}
We now consider the problem presented in the introduction (Figure \ref{Demo}). The square block shown in Figure \ref{Demo} is subjected to pressure $p=1$ on the middle halves of the top and bottom edges.

\begin{figure}[t!]
	\centering
	\includegraphics [width=1\linewidth]{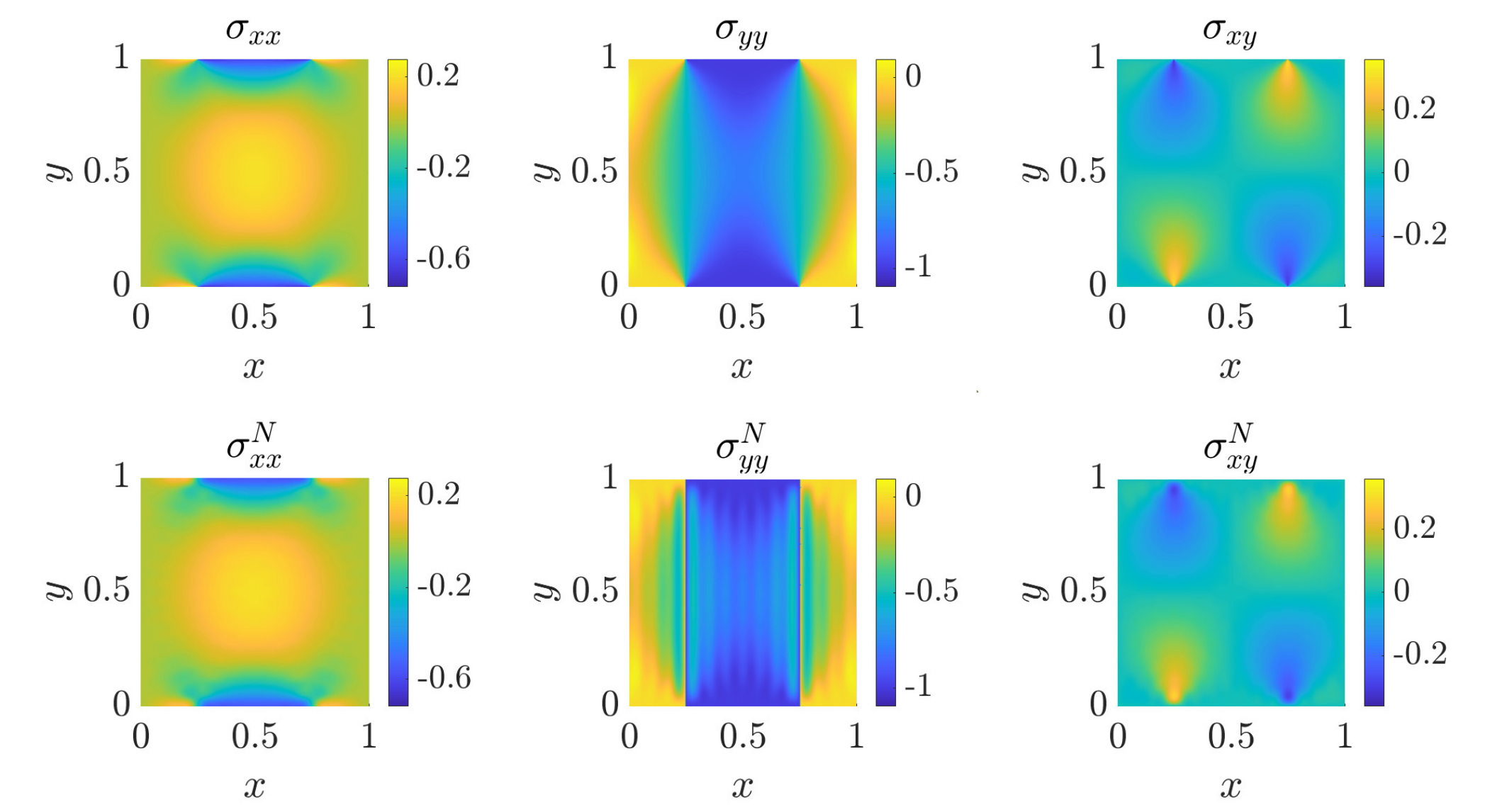}
\caption{True (top row) and approximate (bottom row) stresses in the square block subjected to discontinuous pressure in Example 2 (with $N=500$).}
\label{DTracS}
\end{figure}
\begin{figure}[t!]
    \begin{subfigure}
	\centering
	\includegraphics [width=0.32\linewidth]{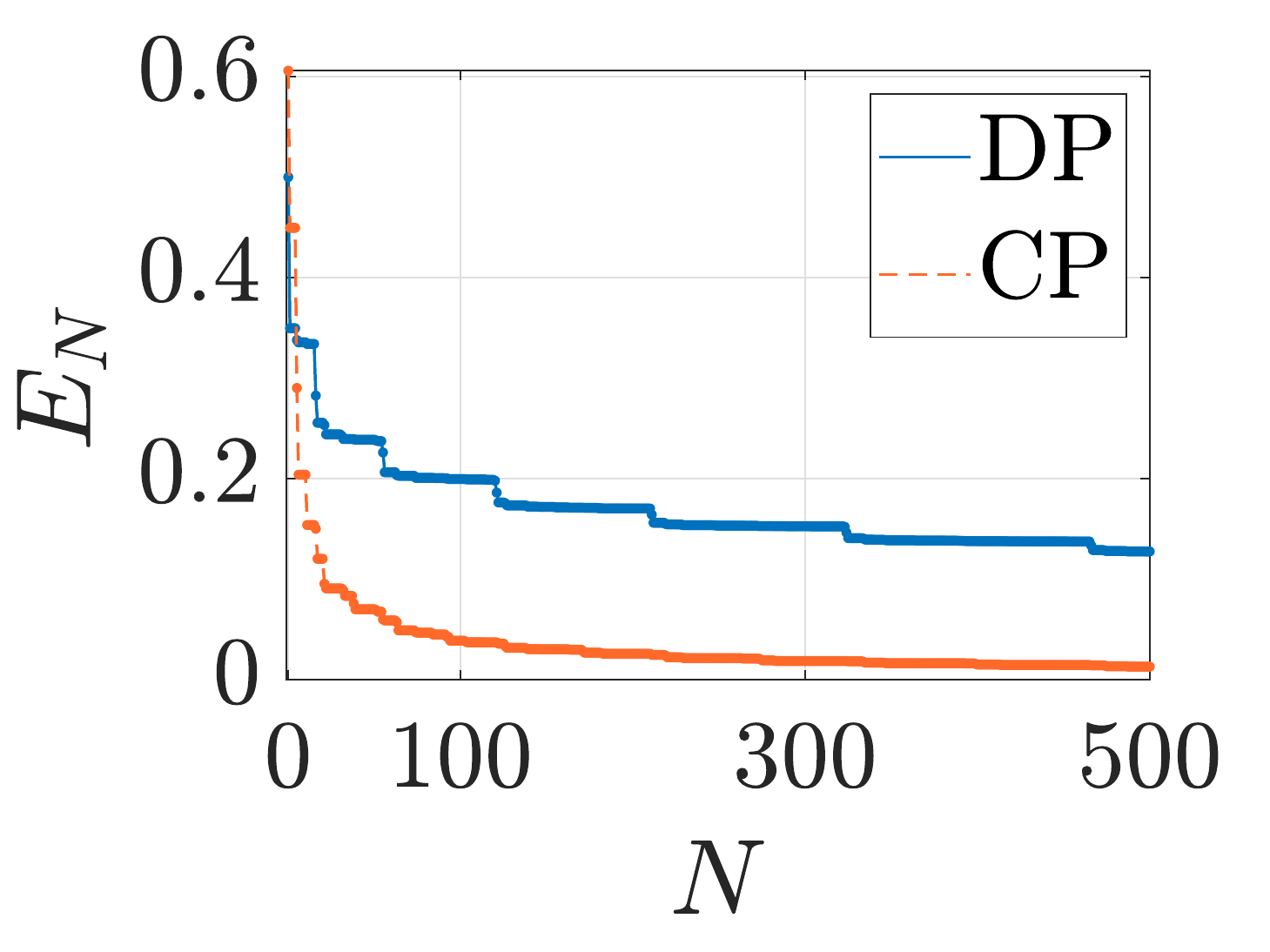}
    \end{subfigure}
    \begin{subfigure}
	\centering
	\includegraphics [width=0.32\linewidth]{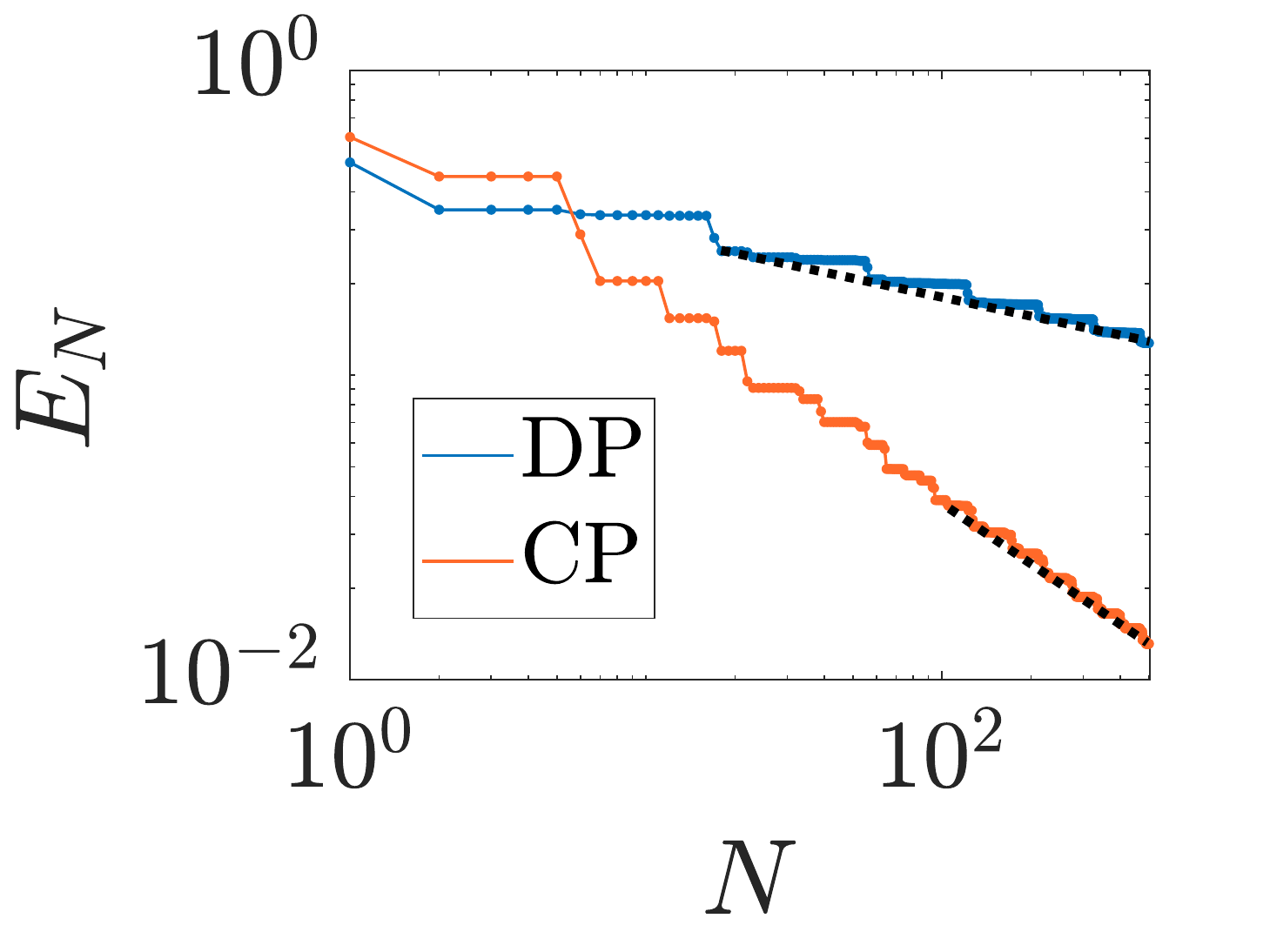}
    \end{subfigure}
    \begin{subfigure}
	\centering
	\includegraphics [width=0.32\linewidth]{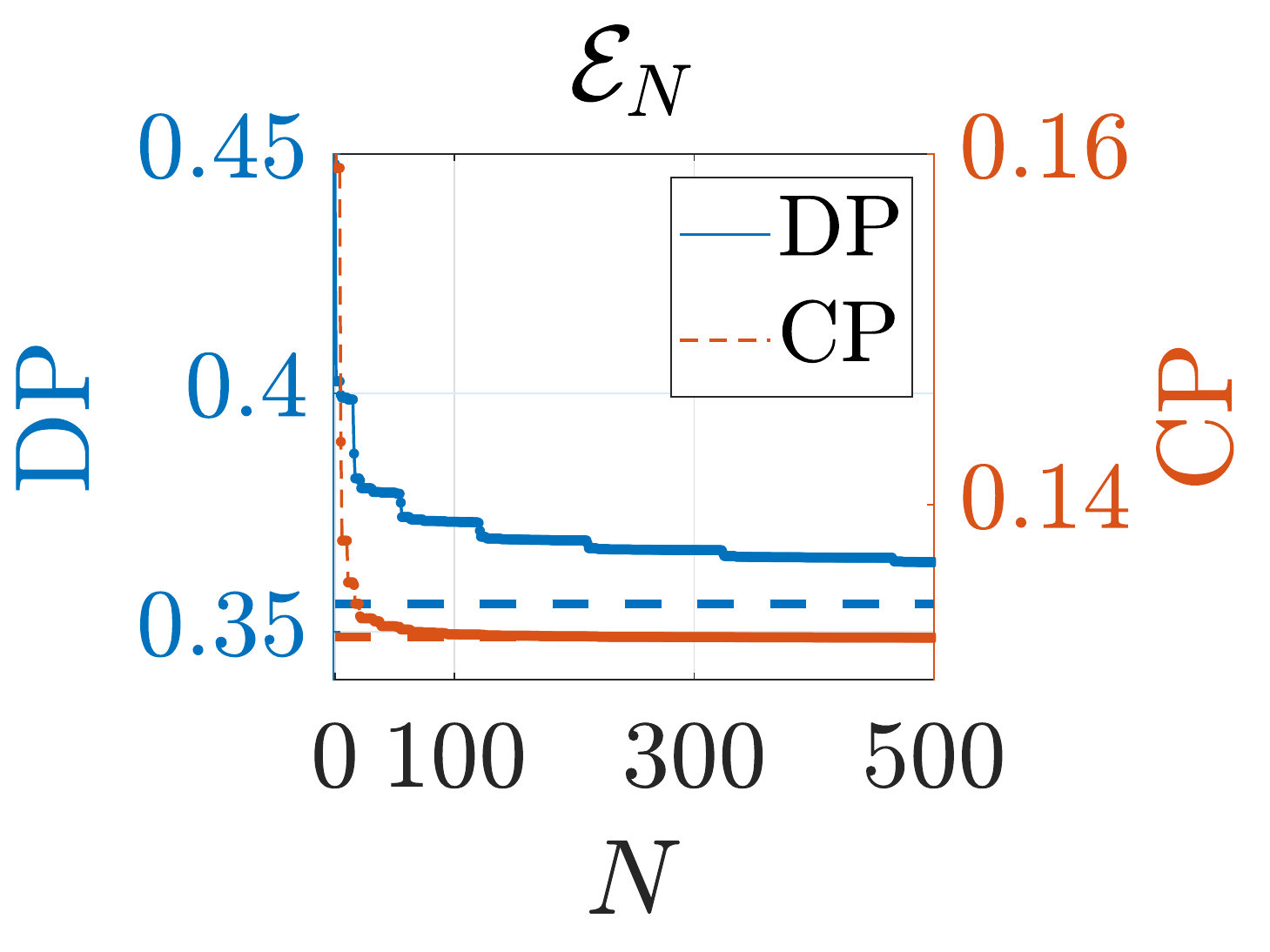}
    \end{subfigure}
\caption{In the figure legends and labels, `DP' refers to the discontinuous pressure case and `CP' refers to the continuous pressure case of Example 2. Left panel: Approximation error $E_N$ versus $N$ on a linear scale; blue: discontinuous pressure, orange: continuous pressure. Middle panel: $E_N$ on a log-log scale. Right panel: Strain energy $\mathcal{E}_N$ corresponding to the approximate stresses; the dashed horizontal lines indicate true strain energies.}
\label{EDTrac}
\end{figure}

A candidate $\bm{\sigma}_p$ was given in Eq.~\ref{casesintro}, and is reproduced below for the ease of reading:
\begin{equation}
\displaystyle
\sigma_{xxp}=0, ~~~ \sigma_{xyp}=0, ~~~\\
\sigma_{yyp}=
\begin{cases}
-p, ~~~~ 1/4\leq x \leq 3/4,\\
0, ~~~~~~\, \text{otherwise.}
\end{cases}
\nonumber
\end{equation}
The strain corresponding to $\bm{\sigma}_p$ is discontinuous and thus incompatible. Therefore, $\bm{\sigma}_p$ is different from the true stress $\bm{\sigma}$. We substitute the above $\bm{\sigma}_p$ in Eq.~\ref{homo} corresponding to the planar trace principle to find the approximations $\bm{\sigma}^N$.

The contours on the undeformed geometry of the true stress (computed using the finite element software Abaqus; details omitted) and the approximate stress (with $N=500$) are shown in Figure \ref{DTracS}. The match is reasonably good, except perhaps for the $yy$ component of the stress, wherein Gibbs oscillations \cite{hewitt} are observed (bottom-middle panel of Figure \ref{DTracS}).

The approximation error $E_N$ is plotted against $N$ in blue on linear and log-log scales in the left and middle panels of Figure \ref{EDTrac}, respectively. From the log-log plot, $E_N$ seems to be decaying but not at a constant rate; the black dotted line fitted to the decaying regime has a slope of $-0.22$. This slow convergence is also mirrored in the right panel of Figure \ref{EDTrac}, where we plot in blue the strain energies $\mathcal{E}_N$ corresponding to the approximate stresses.

The reason for the slow convergence and Gibbs oscillations is that while the stress basis functions $\bm{\phi}_i$ are smooth \cite{tiwari}, our candidate $\bm{\sigma}_p$ is discontinuous, thereby making the limiting $\bm{\sigma}_h^N$ also discontinuous (note that the true stress $\bm{\sigma}$ is continuous). We see this in Figure \ref{splitSq} in the graphical depiction of the splitting $\bm{\sigma}=\bm{\sigma}_h+\bm{\sigma}_p$ where we observe that the approximate $\sigma_{yyh}$ (for $N=500$) needs to approximate a discontinuous function.
\begin{figure}[t!]
	\centering
	\includegraphics [width=1\linewidth]{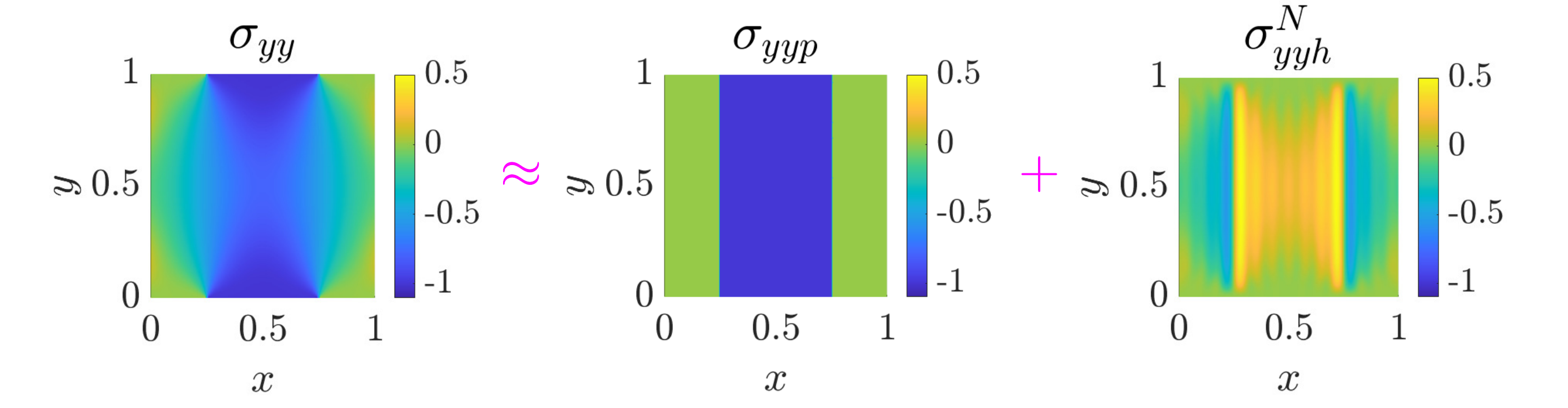}
\caption{Graphical depiction of the splitting $\bm{\sigma}=\bm{\sigma}_p+\bm{\sigma}_h$ through the $yy$ stress components for the square block subjected to discontinuous pressure (Example 2). It is seen that $\sigma_{yyh}^N$ seems to be approximating a discontinuous function.}
\label{splitSq}
\end{figure}

If the loading has greater regularity, faster convergence is observed. We demonstrate this by repeating the above problem with a different pressure $\hat{p}$, which is continuous and has a continuous slope, keeping the net force on each edge the same as before. The expression for $\hat{p}$ is
\begin{equation}
\begin{aligned}
\displaystyle
\hat{p}&=
\begin{cases}
256\left(x- \frac{1}{4}\right)^2\left(x- \frac{3}{4}\right)^2, ~~~~  \frac{1}{4} \leq x \leq  \frac{3}{4},\\
0, ~~~~~~\, \text{otherwise.}
\end{cases}
\end{aligned}
\end{equation}

Accordingly, we take
\begin{equation}
\displaystyle
\sigma_{xxp}=0, ~~~ \sigma_{xyp}=0, ~~~\\
\sigma_{yyp}=
\begin{cases}
-256\left(x- \frac{1}{4}\right)^2\left(x- \frac{3}{4}\right)^2, ~~~~ \frac{1}{4}\leq x \leq \frac{3}{4},\\
0, ~~~~~~\, \text{otherwise.}
\end{cases}
\end{equation}

$E_N$ and $\mathcal{E}_N$ for this loading are plotted in orange in Figure \ref{EDTrac}. In the log-log plot, $E_N$ decays at a faster average rate: the black dotted line fitted to the decaying regime has a slope of $-0.72$. From the $\mathcal{E}_N$ plot in the right panel of Figure \ref{EDTrac} (orange curve), we find that $\mathcal{E}_N$ is within $1\%$ of the true strain energy with just $N=20$ basis functions.

That the convergence is faster for the continuous loading is also borne out by Figure \ref{DCNProg}, wherein we have shown the color plots of the $yy$ component of $\bm{\sigma}$, along with those of the approximations $\bm{\sigma}^N$ with $N=1,10$ and 100. We see that the approximations become better as $N$ increases, and the one for $N=100$ is visually indistinguishable from the true stress. Also, Gibbs oscillations have disappeared since the limiting $\bm{\sigma}_{h}$ is continuous. The latter is also indicated by Figure \ref{syyDC}, which shows the $yy$ component of $\bm{\sigma}^N_h$ (with $N=500$), along with those of $\bm{\sigma}$ and $\bm{\sigma}_p$.
\begin{figure}[t!]
	\centering
	\includegraphics [width=1.1\linewidth]{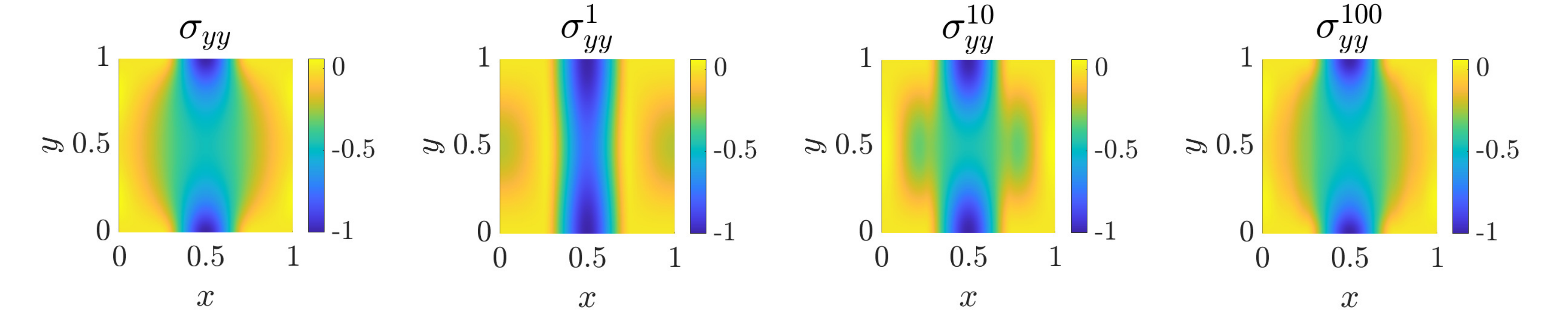}
\caption{True (left-most) and approximate $yy$ components of the stress, with $N=1,10$ and 100, for the case of continuous pressure in Example 2.}
\label{DCNProg}
\end{figure}
\begin{figure}[t!]
	\centering
	\includegraphics [width=1.05\linewidth]{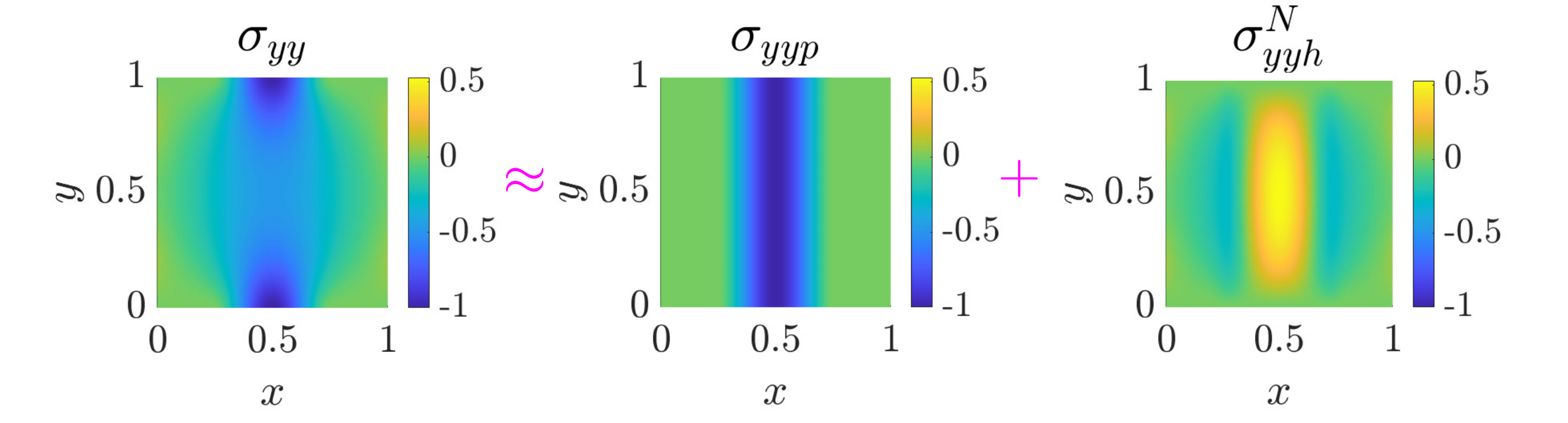}
\caption{$yy$ stress components for the case of continuous pressure in Example 2 (with $N=500$).}
\label{syyDC}
\end{figure}

\subsubsection{Example 3: Irregularly shaped body with uniformly pressurized circular hole} \label{irrHoleSec}
We now consider an irregularly shaped body having a circular hole of radius $r_a$ centered at the origin $O$, as shown in Figure \ref{RIHSketch}. The hole is subjected to uniform pressure $p=1$. 
\begin{figure}[t!]
\centering
\includegraphics[width=0.5\textwidth]{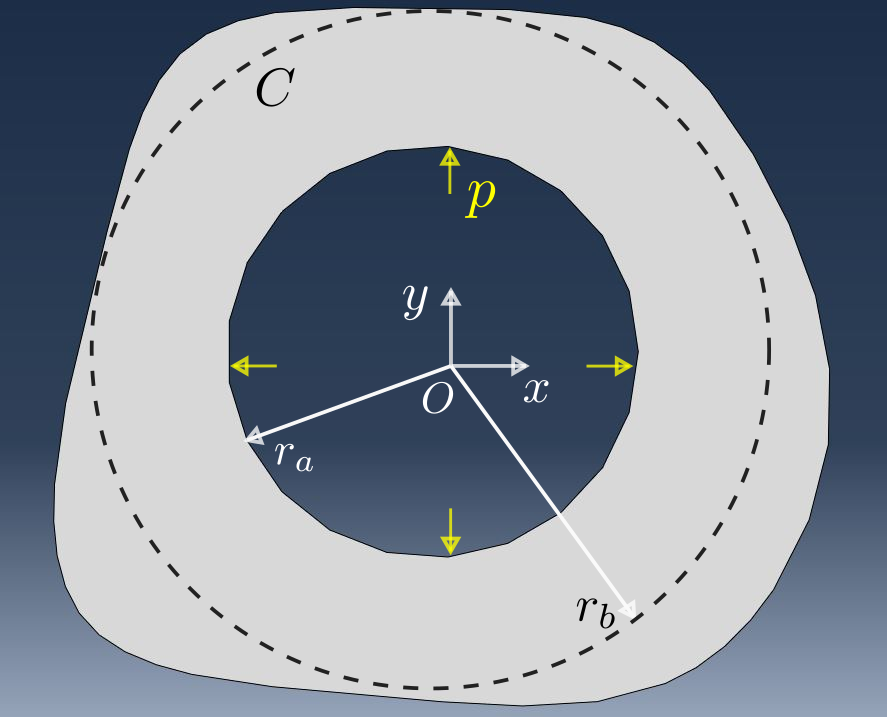}
\caption{An irregularly shaped body with a uniformly pressurized circular hole (Example 3).}
\label{RIHSketch}
\end{figure}

To find a sequence of approximate stresses, we first construct a $\bm{\sigma}_p$ in equilibrium with the loading. To that end, we consider approximately the biggest circle $C$ centered at $O$ completely inside the domain of the body (Figure \ref{RIHSketch}). Denoting the radius of $C$ as $r_b$, we construct a continuous axisymmetric candidate $\bm{\sigma}_p$ that is non-zero in $r_a \le r \le r_b$ and zero for $r>r_b$, i.e., $\bm{\sigma}_p$ satisfies
\begin{equation}
\begin{aligned}
\sigma_{rrp}'+ \frac{\sigma_{rrp}-\sigma_{\theta\theta p}}{r} &= 0,~~~r_a \le r \le r_b,\\
\sigma_{rrp}=\sigma_{\theta\theta p}&=0, ~~~ r> r_b, \\
\sigma_{rrp}&=-1 ~~~ \text{at}~~~ r=r_a,\\
\sigma_{\theta \theta p}|_{r_b^-} &= \sigma_{\theta \theta p}|_{r_b^+},
\end{aligned}
\label{spIH}
\end{equation}
where the last equation enforces continuity of $\sigma_{\theta \theta p}$, and consequently $\bm{\sigma}_p$, at the circle $C$. Note that we can also use the $\bm{\sigma}_p$ of Example 1 for $r\leq r_b$. However, due to the second of Eqs.~\ref{spIH}, it will be discontinuous at $C$, resulting in slower convergence.

To find a candidate $\bm{\sigma}_p$ satisfying Eqs.~\ref{spIH}, we assume the following Airy stress function for the region inside $C$:
$$ \psi_p =c_1 r + c_2 r^2 + c_3 r^3.$$
Substituting $\psi_p$ in Eq.~\ref{stressp}, and using the boundary conditions in Eqs.~\ref{spIH} to eliminate the free parameters, we finally obtain
\begin{equation*}
\begin{aligned}
\sigma_{rrp}=-\frac{\left(r_b-r\right)^2}{\left(r_b-r_a\right)^2}, ~~~~~ \sigma_{\theta \theta p}=\frac{2\left(r_b-r\right)}{\left(r_b-r_a\right)^2}, ~~~~~ \sigma_{r\theta p}=0, ~~~~~~ &r_a\leq r \leq r_b,\\
\sigma_{rrp}=\sigma_{\theta \theta p}=\sigma_{r\theta p}=0, ~~~~~~ &r> r_b.
\end{aligned}
\end{equation*}
Figure \ref{sigmapSk} shows $\bm{\sigma}_p$; we see that $\sigma_{rrp}$ is zero in the yellow region (left panel) and $\sigma_{\theta\theta p}$ is zero in the blue region (right panel). These regions lie outside $C$.
\begin{figure}[t!] 
	\centering
	\includegraphics [width=1\linewidth]{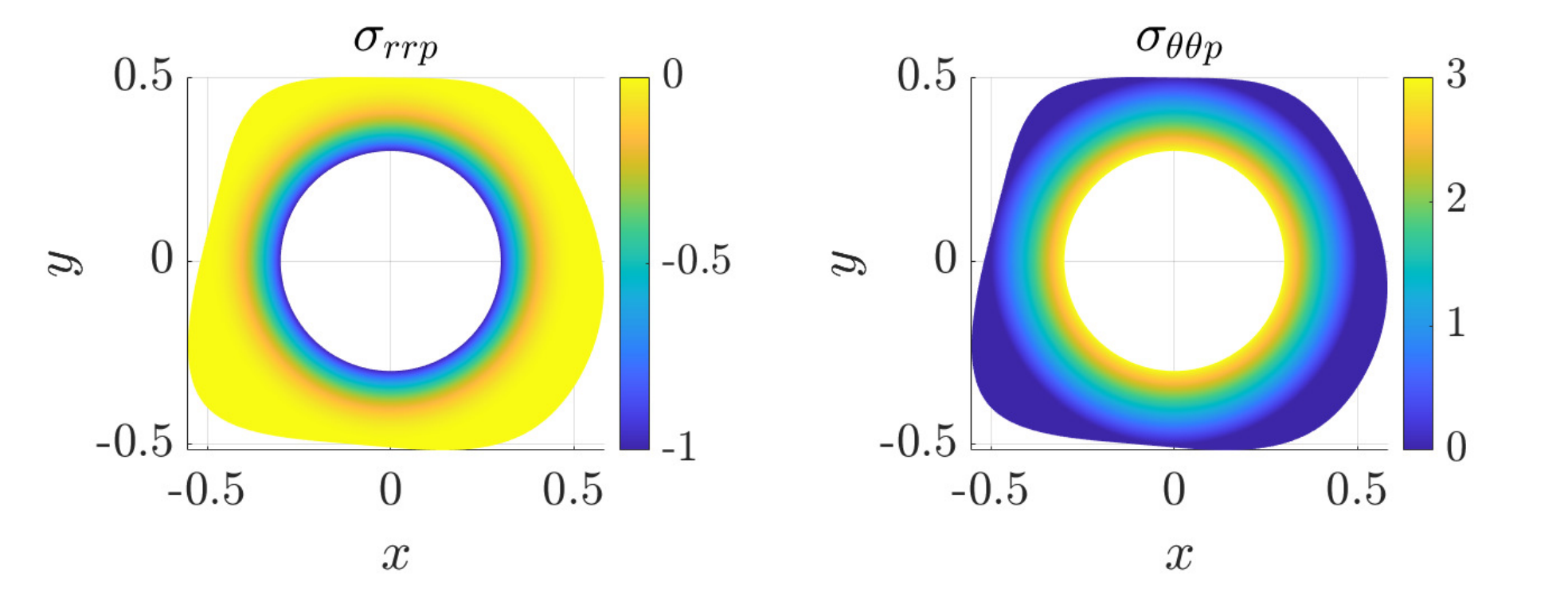}
\caption{A candidate $\bm{\sigma}_p$ for Example 3. $\sigma_{r\theta p}=0$ everywhere.}
\label{sigmapSk}
\end{figure}

\begin{figure}[t!]
	\centering
	\includegraphics [width=1\linewidth]{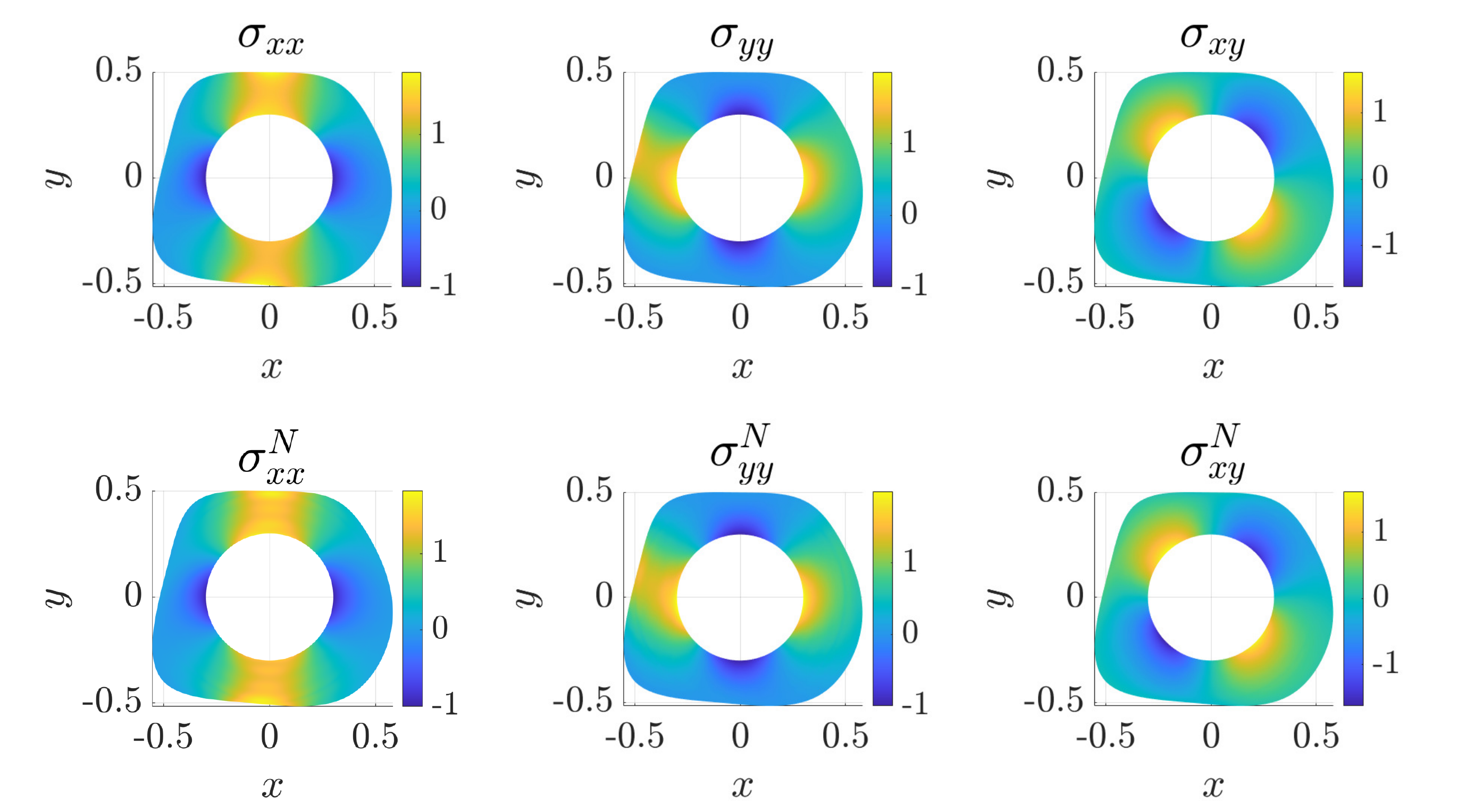}
\caption{True (top row) and approximate (bottom row) stresses in Example 3 (with $N=500$).}
\label{toolf}
\end{figure}

We now use the planar trace principle and substitute $\bm{\sigma}_p$ in Eq.~\ref{homo}. For each given $N$, we obtain an approximate stress $\bm{\sigma}^N$. The true stress (computed using Abaqus) and the approximate stress (with $N=500$) are plotted in Figure \ref{toolf}. The match is good. The approximation error $E_N$ is plotted in Figure \ref{ERIH}. We see in the log-log plot in the right panel that $E_N$ decays but not at a constant rate, and we obtain an average approximate decay rate by computing the slope of the black dotted line. We find that $E_N$ decays like $1/N^{0.75}$ for large $N$. In the left panel of Figure \ref{ERIH}, we have also plotted $\mathcal{E}_N$ in orange. A dashed horizontal orange line indicates the true strain energy $\mathcal{E}$; $\mathcal{E}_N$ is within 0.5\% of $\mathcal{E}$ with just five basis functions.
\begin{figure}[t!]
    \begin{subfigure}
	\centering
	\includegraphics [width=0.5\linewidth]{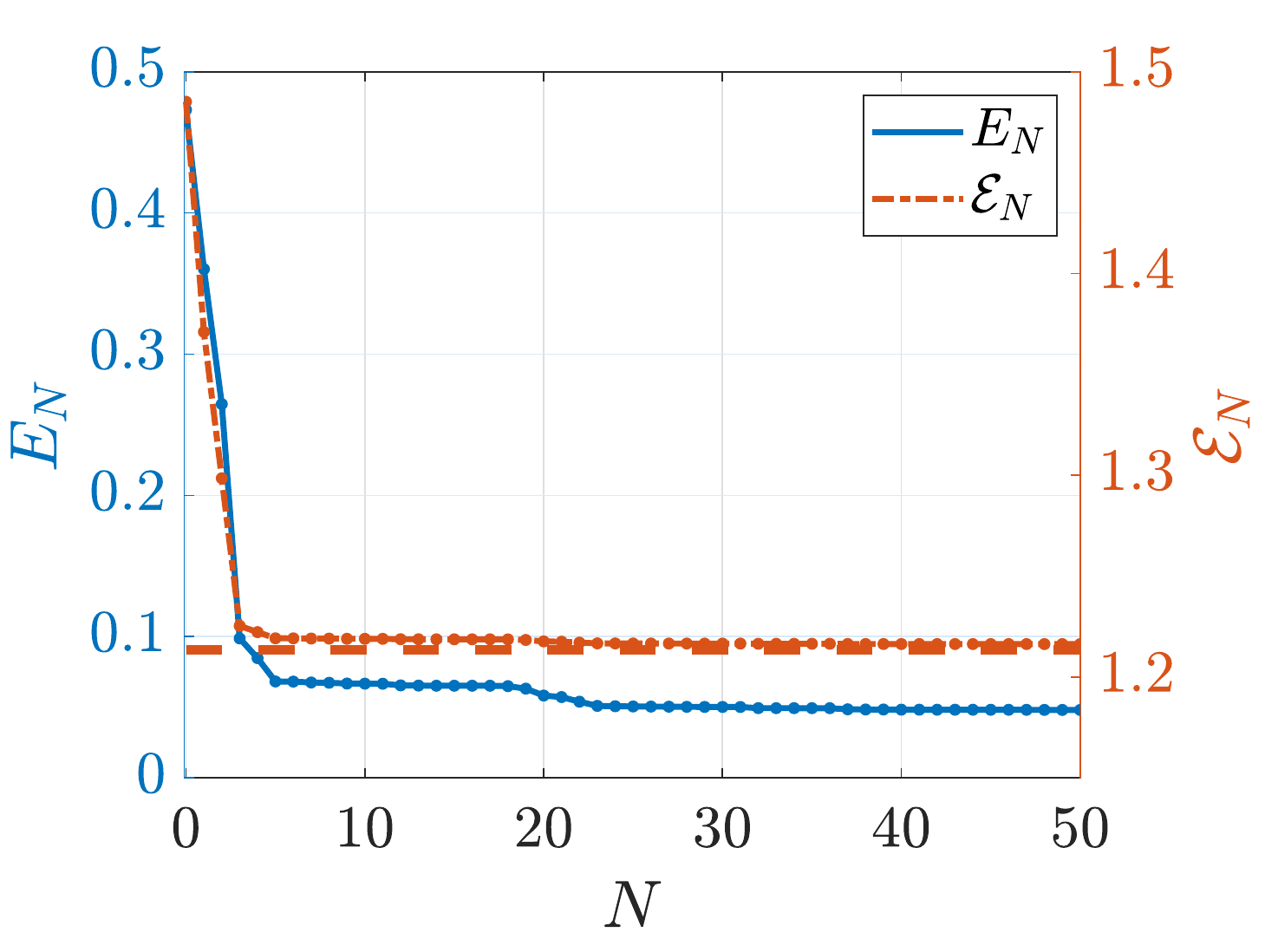}
    \end{subfigure}
    \begin{subfigure}
	\centering
	\includegraphics [width=0.5\linewidth]{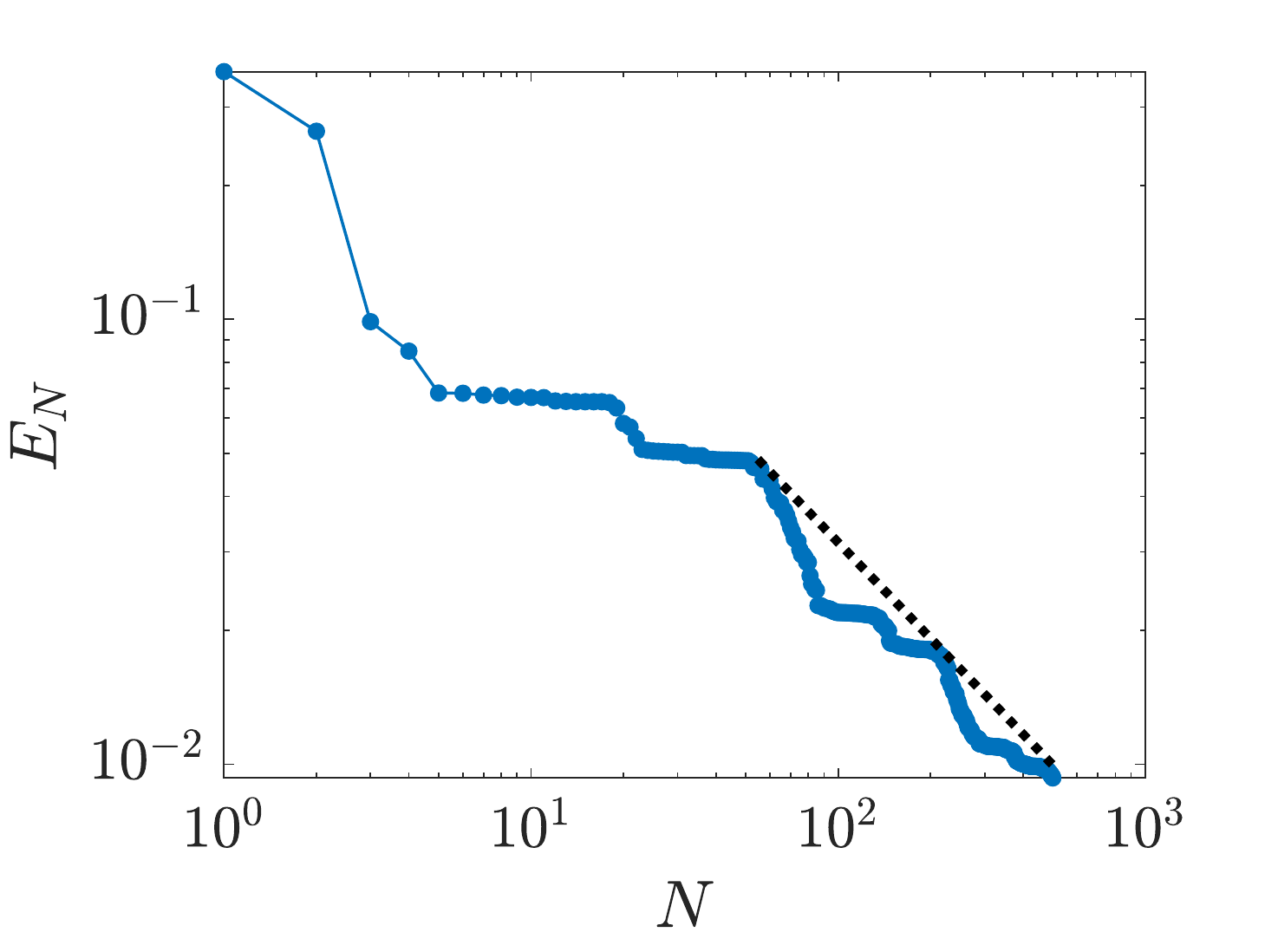}
    \end{subfigure}
\caption{Left: approximation error $E_N$ (in blue) and strain energy corresponding to approximate stress $\mathcal{E}_N$ (in orange) versus $N$ for Example 3; the dashed horizontal orange line corresponds to the strain energy of the true stress. Right: $E_N$ on a log-log scale.}
\label{ERIH}
\end{figure}

\subsubsection{Example 4: Square block resting on a table under gravity}
In the examples considered so far, we have ignored body forces. This example shows how we can incorporate a body force $\bm{b}$ in our formulation based on the planar trace principle. If there is a potential function $V$ such that $\bm{b}=-\nabla V$, then it can be shown along the lines of the proof presented in Section \ref{HomIso} that the true stress minimizes the functional 
\begin{equation}
\mathcal{T}_b=\int_{\Omega} \left(\bar{\sigma}-\frac{V}{1-\nu}\right)^2\,dA
\end{equation}
 over the set 
\begin{equation*}
\mathcal{Q}_b=\left\{\bm{\sigma} \big| ~ \bm{\sigma}\in \,\text{Sym},~\int_{\Omega} \bm{\sigma}\cdot \bm{\sigma}\,dA<\infty,~\text{div}\,\bm{\sigma}+\bm{b}=\bm{0}, ~ \left.\bm{\sigma n}\right|_{\partial \Omega}=\bm{\tau} \right\}.
\end{equation*}

Further, since the basis functions $\bm{\phi}_i$ are traction-free {\em and} have zero divergence, to find approximate stresses using our method, we must choose a $\bm{\sigma}_p$ that is in equilibrium with the given traction {\em and} body force. Mathematically,
\begin{equation*}
\begin{aligned}
\text{div}\, \bm{\sigma}_p +\bm{b}&= \bm{0} ~~~ \text{in} ~~~ \Omega, \\
\bm{\sigma}_p \bm{n} &= \bm{\tau} ~~~ \text{on} ~~~ \partial\Omega.
\end{aligned}
\end{equation*}

The above variational principle then gives an approximate stress
\begin{equation}
\bm{\sigma}^N=\bm{\sigma}_p+\sum_{j=1}^{N} \left\{- \int_{\Omega} \left(\bar{\sigma}_p-\frac{V}{1-\nu}\right) \bar{\phi}_j \, dA\right\} \bm{\phi}_j.
\label{homob}
\end{equation}
The proof of this result is presented in Appendix \ref{bf}.

For demonstration, we consider a square block resting on a rigid friction-less table in a constant gravitational field $-g\bm{e}_2$. The block is made of two different homogeneous materials with the same elastic constants $Y$ and $\nu$, but different mass densities $\rho_1$ and $\rho_2$, as shown in the left panel of Figure \ref{HB}. Note that the planar trace principle continues to hold because the elastic properties are homogeneous and isotropic.

The free-body diagram of the block is sketched in the right panel of Figure \ref{HB}. We assume that the normal force applied by the table is uniformly distributed over the lower edge of the block. The same boundary condition on the lower edge is used to compute the stress problem in Abaqus.
\begin{figure}[t!]
\begin{subfigure}
\centering
\includegraphics[width=0.5\textwidth]{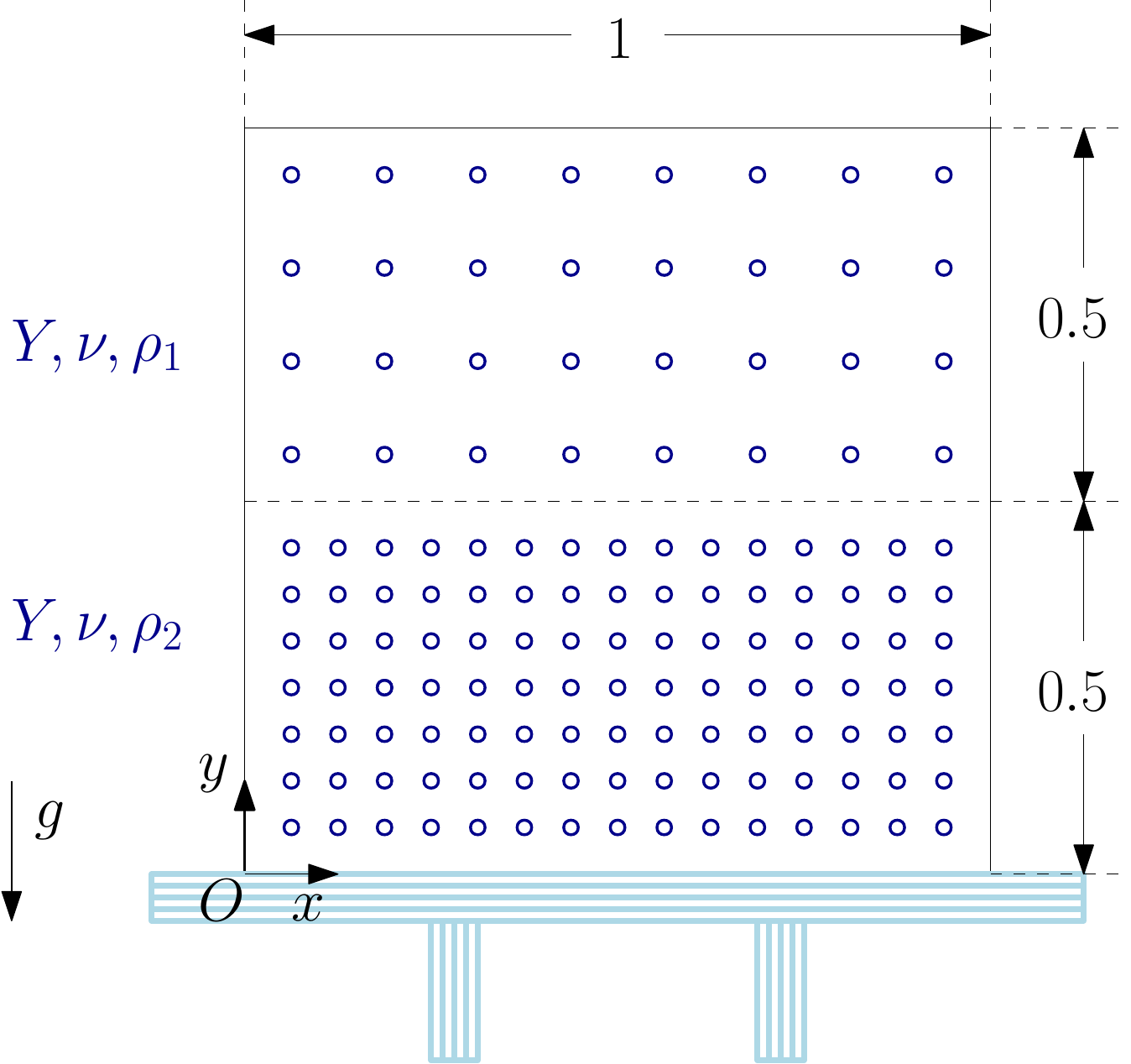}
\end{subfigure}
\begin{subfigure}
\centering
\hspace{1cm}\includegraphics[width=0.45\textwidth]{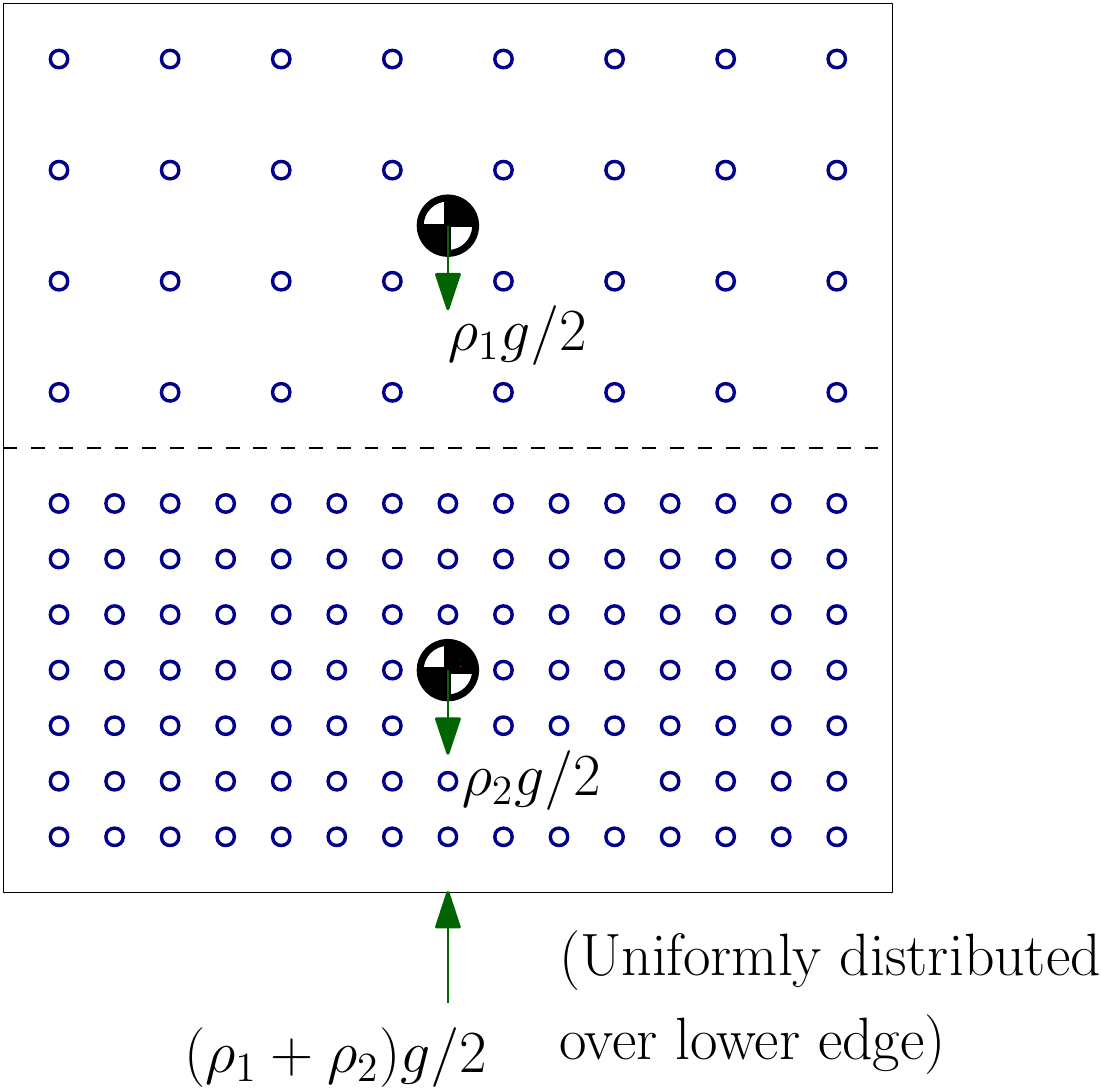}
\end{subfigure}
\caption{Left: schematic of the block resting on a rigid friction-less table in Example 4; right: free-body diagram of the block.}
\label{HB}
\end{figure}

The potential function $V$ corresponding to the gravitational force is
\begin{equation*}
V=-\int_{y=0}^1 b_y \, dy = \int_{y=0}^{1}\rho(y)g\,dy,
\end{equation*}
where $\rho(y)$ is $\rho_1$ for $0.5\leq y\leq 1$ and $\rho_2$ for $0\leq y< 0.5$; and $b_y$ is the $y$ component of $\bm{b}$. Substituting this $\rho(y)$ in the above equation, we find that
\begin{equation}
V=
\begin{cases}
\begin{aligned}
\left(\rho_2-\rho_1\right)g/2+ \rho_1 g y, ~~~~  & 0.5\leq y \leq  1,\\
\rho_2 g y, ~~~~  & 0\leq y <  0.5.
\end{aligned}
\end{cases}
\label{V}
\end{equation}

Let us now find a candidate $\bm{\sigma}_p$. Since $\bm{\sigma}_p$ is in equilibrium with the combined loading of contact force and gravity, it satisfies
\begin{equation}
\begin{aligned}
\text{div}\,\bm{\sigma}_p&=\rho(y)g\,\bm{e}_2=
\begin{cases}
\rho_1 g \bm{e}_2, ~~~~  0.5\leq y \leq  1,\\
\rho_2 g \bm{e}_2, ~~~~  0\leq y <  0.5,\\
\end{cases}\\
\sigma_{yyp}&=-g\, \int\limits_{y=0}^1 \rho(y)\, dy=-\frac{\left(\rho_1+\rho_2\right) g}{2} ~~~ \text{at}~~~ y=0.
\end{aligned}
\label{heavyBlockBVPsigmap}
\end{equation}
A straightforward candidate $\bm{\sigma}_p$ is
\begin{equation}
\sigma_{xxp}=\sigma_{xyp}=0, ~~~\sigma_{yyp}=
\begin{cases}
g\left(\rho_2 y -\frac{\rho_1+\rho_2}{2}\right),  ~~~~~~  0\leq y \leq  0.5,\\
g\rho_1 \left(y-1\right),  ~~~~~~  0.5< y \leq 1.
\end{cases}
\label{sigmapHeavyBlock}
\end{equation}
We choose $g=1, \rho_1=1$ and $\rho_2=3$, and substitute the above $\bm{\sigma}_p$ along with $V$ of Eq.~\ref{V} in Eq.~\ref{homob}. This way, we compute, for any given $N$, an approximate stress $\bm{\sigma}^N$.
\begin{figure}[t!] 
	\centering
	\includegraphics [width=1\linewidth]{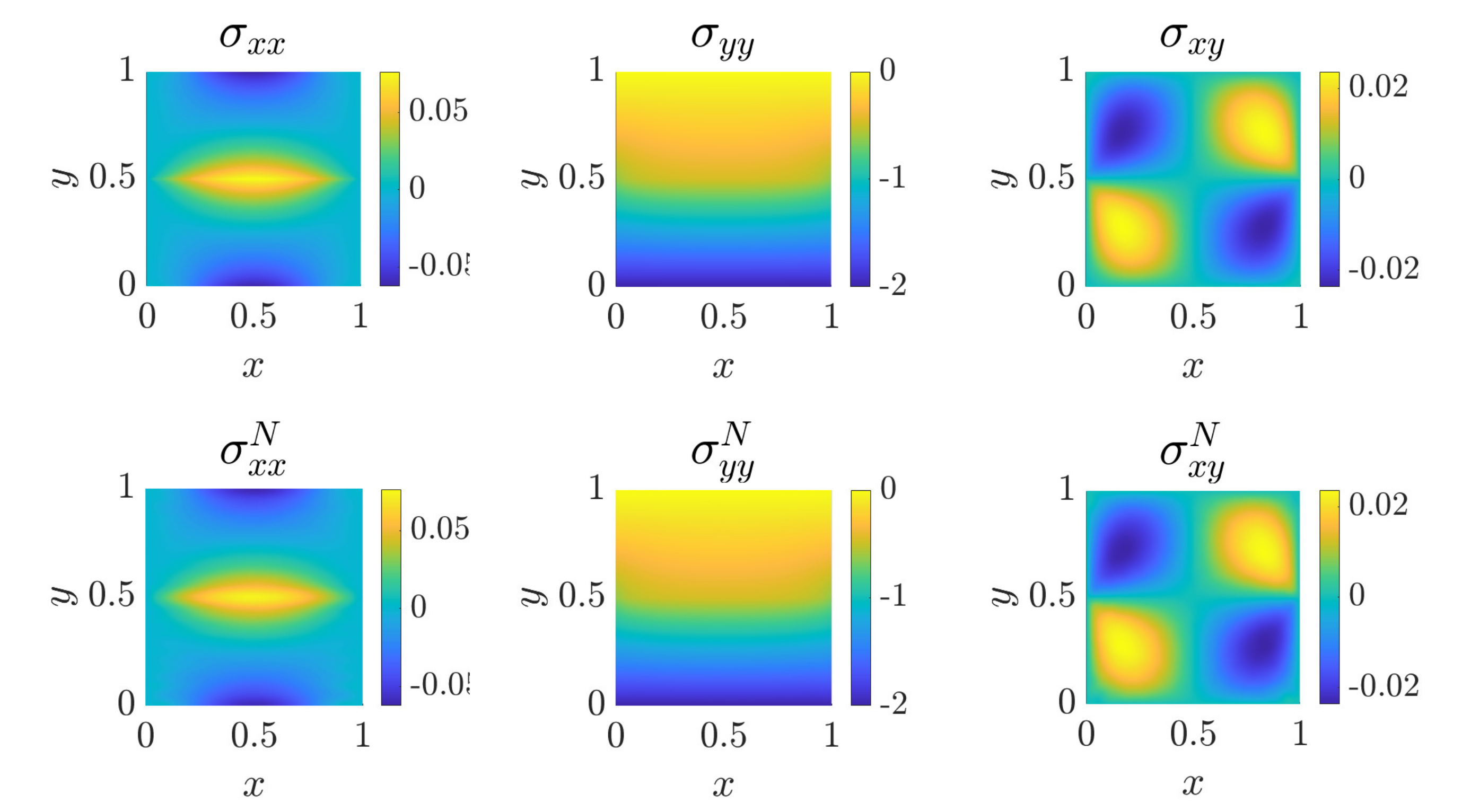}
\caption{True (top row) and approximate (bottom row) stresses in the square block resting on a table in Example 4 (with $N=500$).}
\label{HBS}
\end{figure}
We plot the true stress (obtained using Abaqus) and the approximate stress (with $N=500$) in Figure \ref{HBS}. The match is very good.

The approximation error $E_N$ is plotted in blue in Figure \ref{HBDenVar}, on a linear scale in the left panel and a log-log scale in the right panel. In the log-log plot, we fit the shown black dotted line to $E_N$ in the decaying regime, and its slope is computed to be $-0.58$. We also see in the left panel of Figure \ref{HBDenVar} that the approximate strain energies $\mathcal{E}_N$ (orange curve) converge rapidly to the true strain energy $\mathcal{E}$ indicated by the dashed horizontal orange line.

From the left of Figure \ref{HBDenVar}, we note that even for $N=0$, $E_N$ is only about 0.04. Similarly, we see from the $\mathcal{E}_N$ plot that the strain energy of $\bm{\sigma}_p$ (corresponding to $N=0$ in the figure) matches that of the true stress till the second decimal point. It indicates that $\bm{\sigma}_p$ is already close to the true stress $\bm{\sigma}$, and $\bm{\sigma}_h$ merely fine-tunes the approximation. This is seen clearly in Figure \ref{HBsplit}, where we have plotted the $yy$ components of $\bm{\sigma}$, $\bm{\sigma}_p$ and $\bm{\sigma}_h^N$ (for $N=500$). We see that $\sigma_{yyh}^N$ is imperceptible in comparison to $\sigma_{yy}$ and $\sigma_{yyp}$. In the inset on the top-right of the figure, we have plotted $\sigma_{yyh}^N$ with a re-scaled color bar to see its spatial variation clearly.

\begin{figure}[t!]
    \begin{subfigure}
	\centering
	\includegraphics [width=0.5\linewidth]{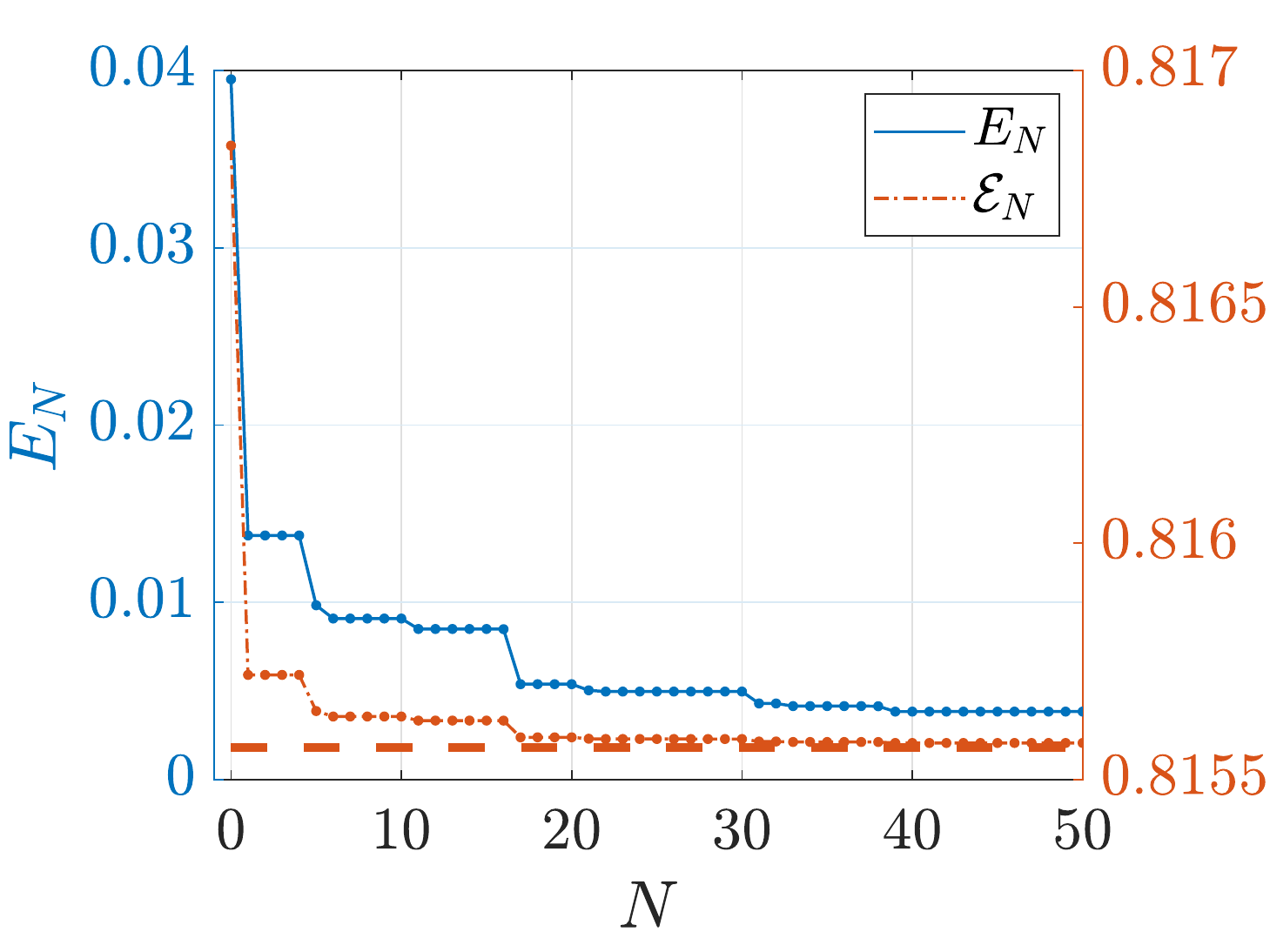}
    \end{subfigure}
    \begin{subfigure}
	\centering
	\includegraphics [width=0.5\linewidth]{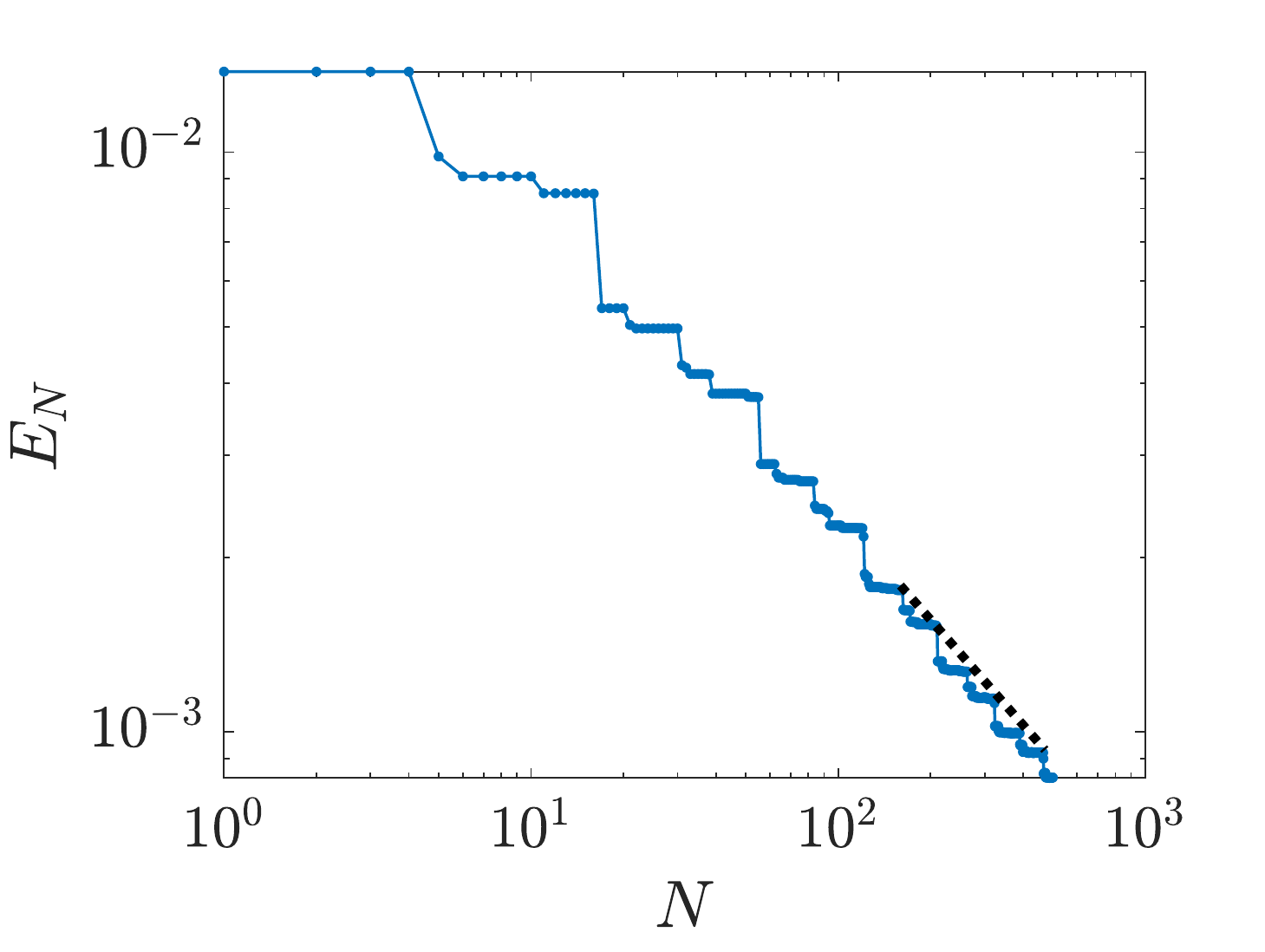}
    \end{subfigure}
\caption{Left: approximation error $E_N$ (in blue) and strain energy corresponding to approximate stress $\mathcal{E}_N$ (in orange) versus $N$ for Example 4; the dashed horizontal orange line corresponds to the strain energy of the true stress. Right: $E_N$ on a log-log scale.}
\label{HBDenVar}
\end{figure}
\begin{figure}[t!]
	\centering
	\includegraphics [width=1.05\linewidth]{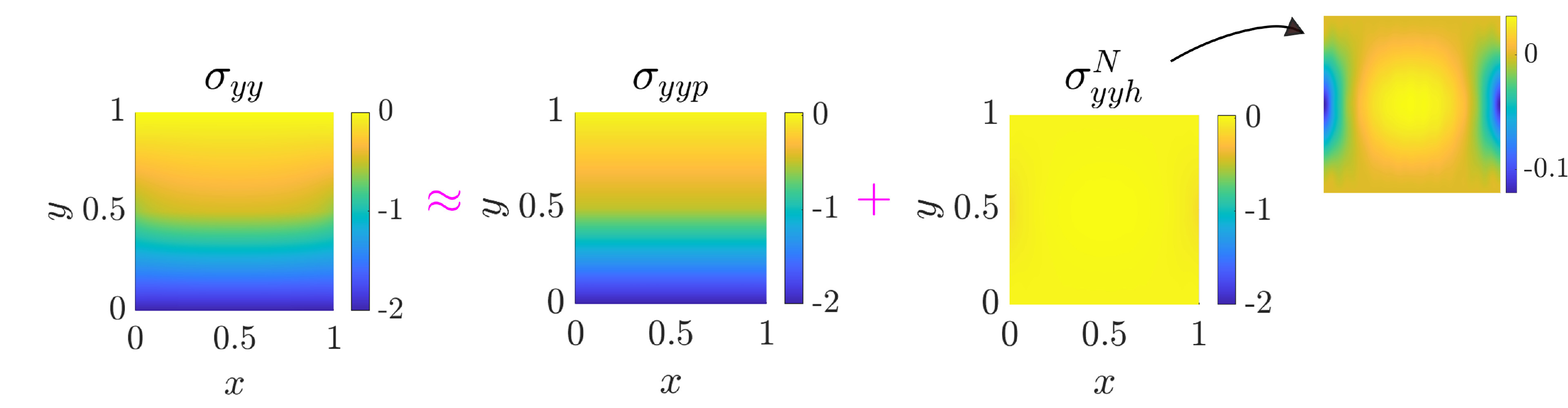}
\caption{$yy$ stress components of $\bm{\sigma}$, $\bm{\sigma}_p$ and $\bm{\sigma}_h^N$ (with $N=500$) for Example 4. Inset: $\bm{\sigma}_h^N$ with a re-scaled color bar.}
\label{HBsplit}
\end{figure}

A minor point is that, unlike the other stress problems studied, in this problem the traction on the bottom edge was not specified and was conveniently assumed to match the $\bm{\sigma}_p$ used.

\subsubsection{Example 5: Annulus with non-zero net force on the hole}
We showed in Section \ref{HomIso} that the planar trace principle does not hold if a multiply-connected body is subjected to a traction with non-zero net force on at least one internal hole. We also observed in Section \ref{kala} that the strain energy principle holds even for such cases. In this example and the following example (Example 6), we present a numerical demonstration of these observations. In the present example, we consider an annulus subjected to a traction with a chosen azimuthal wavenumber, and the problem becomes essentially 1D. In Example 6, we consider a non-axisymmetric multiply-connected body.

We consider a homogeneous isotropic elastic annulus centered at the origin with inner and outer radii $r_a=0.1$ and $r_b=0.3$, respectively. It is subjected to a traction
\begin{equation*}
\displaystyle
\bm{\tau}=
\begin{cases}
\displaystyle
-\cos{\theta} ~\bm{e}_r ~~~ \text{at} ~~~ r=r_a,\\
\left(\cos{\theta}\right)/3 ~\bm{e}_r ~~~ \text{at} ~~~ r=r_b.\\
\end{cases}
\end{equation*}
A schematic depicting this loading is shown in the left panel of Figure \ref{FnzAn}. It is clear from the figure that the net force on the hole has a non-zero horizontal component. 
\begin{figure}[t!]
    \begin{subfigure}
	\centering
	\includegraphics [width=0.5\linewidth]{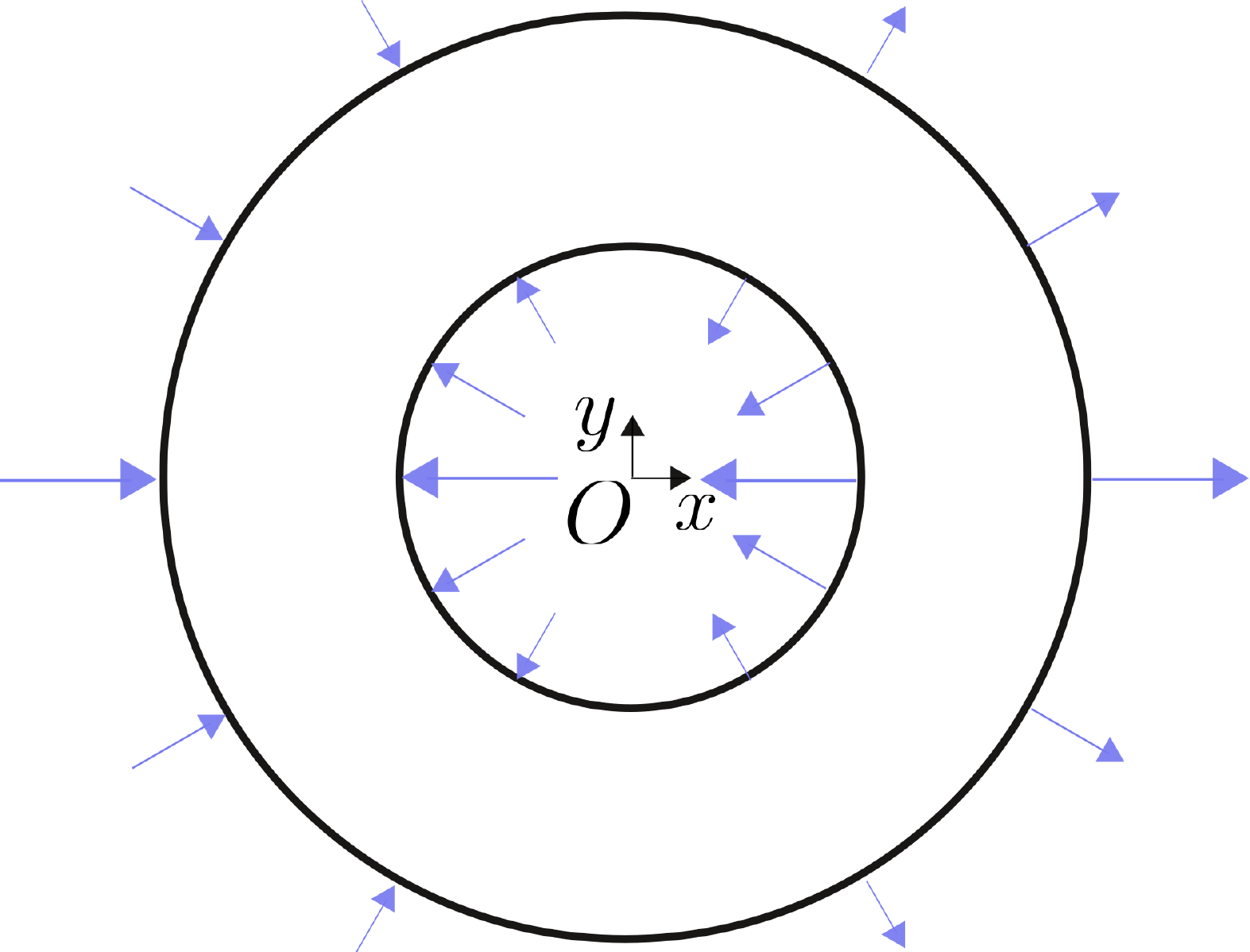}
    \end{subfigure}
    \begin{subfigure}
	\centering
	\includegraphics [width=0.5\linewidth]{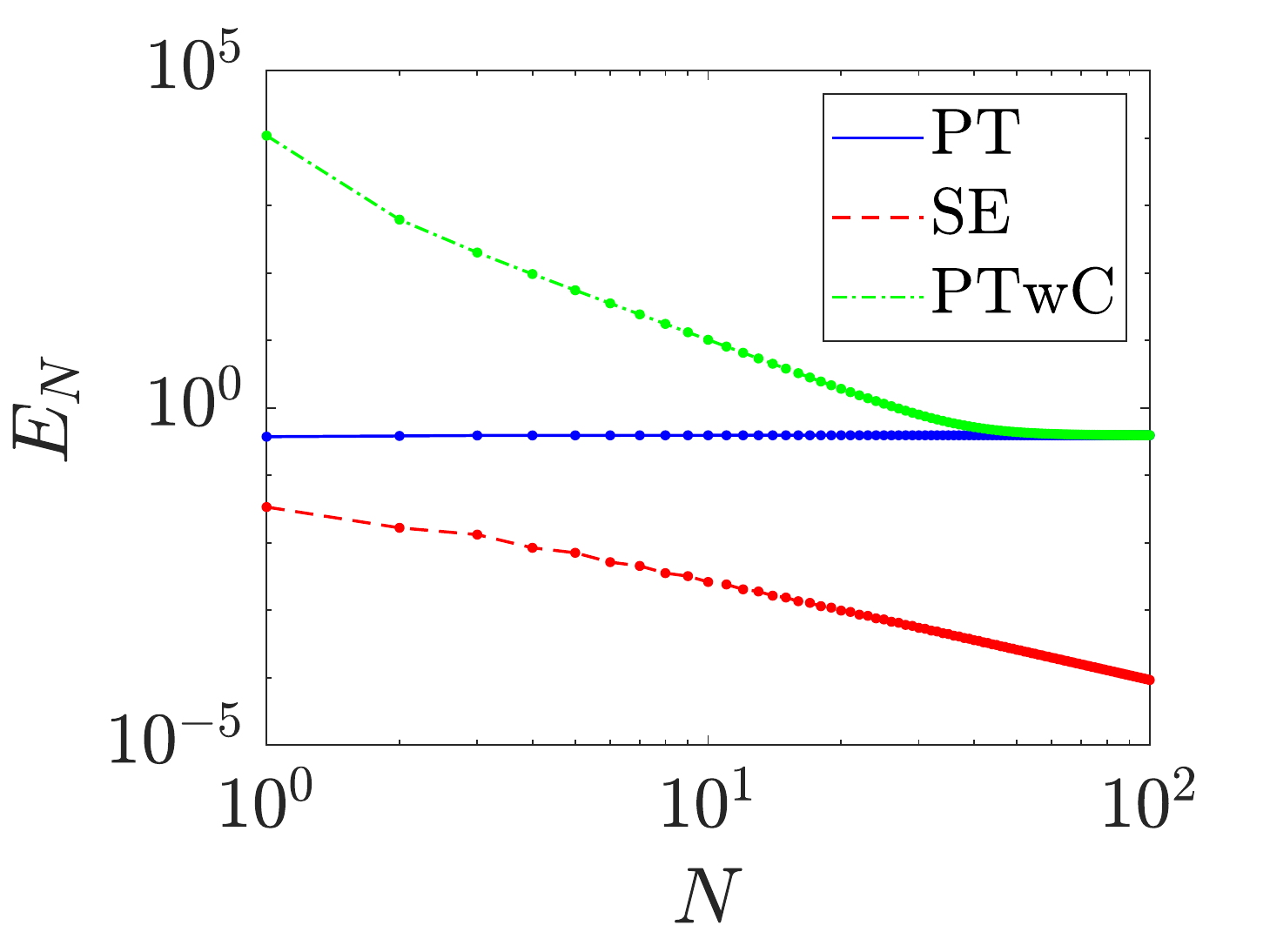}
    \end{subfigure}
\caption{Left: Schematic of an annulus subjected to traction (depicted through blue arrows) with non-zero net force on the hole in Example 5. Right: approximation error $E_N$ with the planar trace (PT), strain energy (SE) and planar trace with enforced integral constraint (PTwC) approaches.}
\label{FnzAn}
\end{figure}

In view of the given traction, we assume a stress solution of the form
$$ \sigma_{rr}=\sigma_{rr}(r)\cos{\theta}, ~~~\sigma_{\theta\theta}=\sigma_{\theta\theta}(r)\cos{\theta}, ~~~\sigma_{r\theta}=\sigma_{r\theta}(r)\sin{\theta}.$$ 
Substituting this in the equilibrium equation $\text{div}\, \bm{\sigma}=\bm{0}$ gives 
\begin{equation}
\begin{aligned}
\sigma_{rr}' + \frac{\sigma_{r\theta}+\sigma_{rr}-\sigma_{\theta\theta}}{r}&=0,\\
\sigma_{r\theta}'+\frac{2\sigma_{r\theta}-\sigma_{\theta\theta}}{r}&=0,
\end{aligned}
\label{eqap}
\end{equation}
where we have suppressed the $r$-dependence. The local compatibility equation $\Delta \bar{\sigma}=0$ gives
\begin{equation}
\sigma_{rr}'' + \sigma_{\theta\theta}'' + \frac{\sigma_{rr}'+\sigma_{\theta\theta}'}{r}-\frac{\sigma_{rr}+\sigma_{\theta\theta}}{r^2}=0.
\label{coap}
\end{equation}
Application of the C\'esaro integral condition (Eq.~\ref{loopapp}) at $r=r_a$ gives
\begin{equation}
2\epsilon_{r\theta}(r_a)+\epsilon_{rr}(r_a)-r_a \epsilon_{\theta\theta}'(r_a)=0,
\label{ceap}
\end{equation}
where
\begin{equation}
\begin{split}
\epsilon_{rr}(r_a)=\frac{1+\nu}{Y} \left\{(1-\nu)\sigma_{rr}(r_a)-\nu \sigma_{\theta\theta}(r_a)\right\},\\
\epsilon_{\theta\theta}(r_a)=\frac{1+\nu}{Y} \left\{(1-\nu)\sigma_{\theta\theta}(r_a)-\nu \sigma_{rr}(r_a)\right\},\\
\epsilon_{r\theta}(r_a)=\frac{1+\nu}{Y}\sigma_{r\theta}(r_a).
\end{split}
\nonumber
\end{equation}
Finally, the traction boundary conditions are
\begin{equation}
\begin{aligned}
\sigma_{rr}=1 ~~~ & \text{and} ~~~ \sigma_{r\theta}=0 ~~~ \text{at} ~~~ r=r_a,\\
\sigma_{rr}=1/3 ~~~ & \text{and} ~~~ \sigma_{r\theta}=0 ~~~ \text{at} ~~~ r=r_b.
\end{aligned}
\label{trap}
\end{equation}
We solve for the true stress $\bm{\sigma}$ using the Eqs.~\ref{eqap}--\ref{trap}.

To find an approximate stress using the planar trace and strain energy principles, we begin by constructing a $\bm{\sigma}_p$. Unlike Examples 1 and 3, we cannot construct a $\bm{\sigma}_p$ from an Airy stress function $\psi_p$ here. The reason is that no Airy stress function {\em exists} since the hole is subjected to a non-zero net force \cite{gurtin}. Instead, we adopt a trick from \cite{tiwari} and begin by assuming
$$ \sigma_{\theta\theta p}=c_0 + c_1 r,$$
where $c_0$ and $c_1$ are free parameters. We substitute the above into equilibrium Eqs.~\ref{eqap} and integrate the resulting equations to find $\sigma_{rrp}$ and $\sigma_{r\theta p}$ in terms of $c_0$ and $c_1$ while retaining two additional integration constants obtained {\em en route}. These four constants are chosen to satisfy the boundary conditions in Eqs.~\ref{ceap} and \ref{trap}, and we finally obtain
\begin{equation}
\begin{aligned}
\sigma_{rrp}&=\frac{r}{3}-\frac{13}{120}+\frac{0.003}{4r^2} + \frac{0.1}{r},\\
\sigma_{\theta\theta p} &=- \frac{13}{60}+r,\\
\sigma_{r\theta p}&=\frac{r}{3}-\frac{13}{120}+\frac{0.003}{4r^2}.
\end{aligned}
\label{sigmapap}
\end{equation}

Given the azimuthal variation of the traction, we consider only a subset of the basis functions $\bm{\phi}_i$ having the form\footnote{Basis functions with all other azimuthal wavenumbers are discarded. We also discard the basis functions with azimuthal wavenumber $m=1$ having the form $$ -\phi_{rri}(r) \sin{\theta}\, \bm{e}_r \otimes \bm{e}_r - \phi_{\theta\theta i}(r) \sin{\theta}\, \bm{e}_{\theta} \otimes \bm{e}_{\theta} + \phi_{r\theta i}(r) \cos{\theta}\, \bm{e}_r \otimes \bm{e}_{\theta} + \phi_{r\theta i}(r) \cos{\theta} \,\bm{e}_{\theta} \otimes \bm{e}_{r}.$$}
\begin{equation}
\begin{split}
\bm{\phi}_i = \phi_{rri}(r) \cos{\theta} \,\bm{e}_r \otimes \bm{e}_r \,+\, \phi_{\theta\theta i}(r) \cos{\theta}\, \bm{e}_{\theta} \otimes \bm{e}_{\theta} \\
+\, \phi_{r\theta i}(r) \sin{\theta}\, \bm{e}_r \otimes \bm{e}_{\theta} + \phi_{r\theta i}(r) \sin{\theta}\, \bm{e}_{\theta} \otimes \bm{e}_{r}.
\end{split}
\nonumber
\end{equation}
Substituting these basis functions and $\bm{\sigma}_p$ (Eq.~\ref{sigmapap}) into Eqs.~\ref{hetero} and \ref{homo}, we obtain approximate stresses using the strain energy and planar trace principles, respectively.

The approximation errors $E_N$ obtained using these two principles are plotted in the right panel of Figure \ref{FnzAn} in blue and red, respectively. We see that $E_N$ corresponding to the planar trace principle settles at a non-zero value, while that corresponding to the strain energy principle decays steadily with $N$. That the trace energy principle does not work is expected from our theory since the net force on the hole is non-zero. Accordingly, the global compatibility of strain is not satisfied (Theorem \ref{PTP}).

It is important to note that applying the planar trace principle with the C\'esaro integral condition (Eq.~\ref{loopapp}) on the inner boundary $r=r_a$ as an additional constraint does not give the correct solution either. The corresponding approximation error, plotted in green in the right panel of Figure \ref{FnzAn}, settles at a non-zero value. The applicability of the planar trace principle to problems dealing with multiply-connected bodies that have at least one hole subjected to non-zero net force may be clarified with future work.

\subsubsection{Example 6: A multiply-connected body with a more general loading}
In this example, we consider a homogeneous isotropic body that is multiply-connected and non-axisymmetric. It is subjected to loading of the following three types: (1) non-zero net force and zero net moment (about the center $O$; see Figure \ref{FM}) on the hole, (2) zero net force and zero net moment on the hole, and (3) zero net force and non-zero net moment on the hole. These cases are depicted in the top-right, bottom-left and bottom-right panels of Figure \ref{FM}.
\begin{figure}[t!]
 \begin{subfigure}
	\centering
	\includegraphics [width=0.5\linewidth]{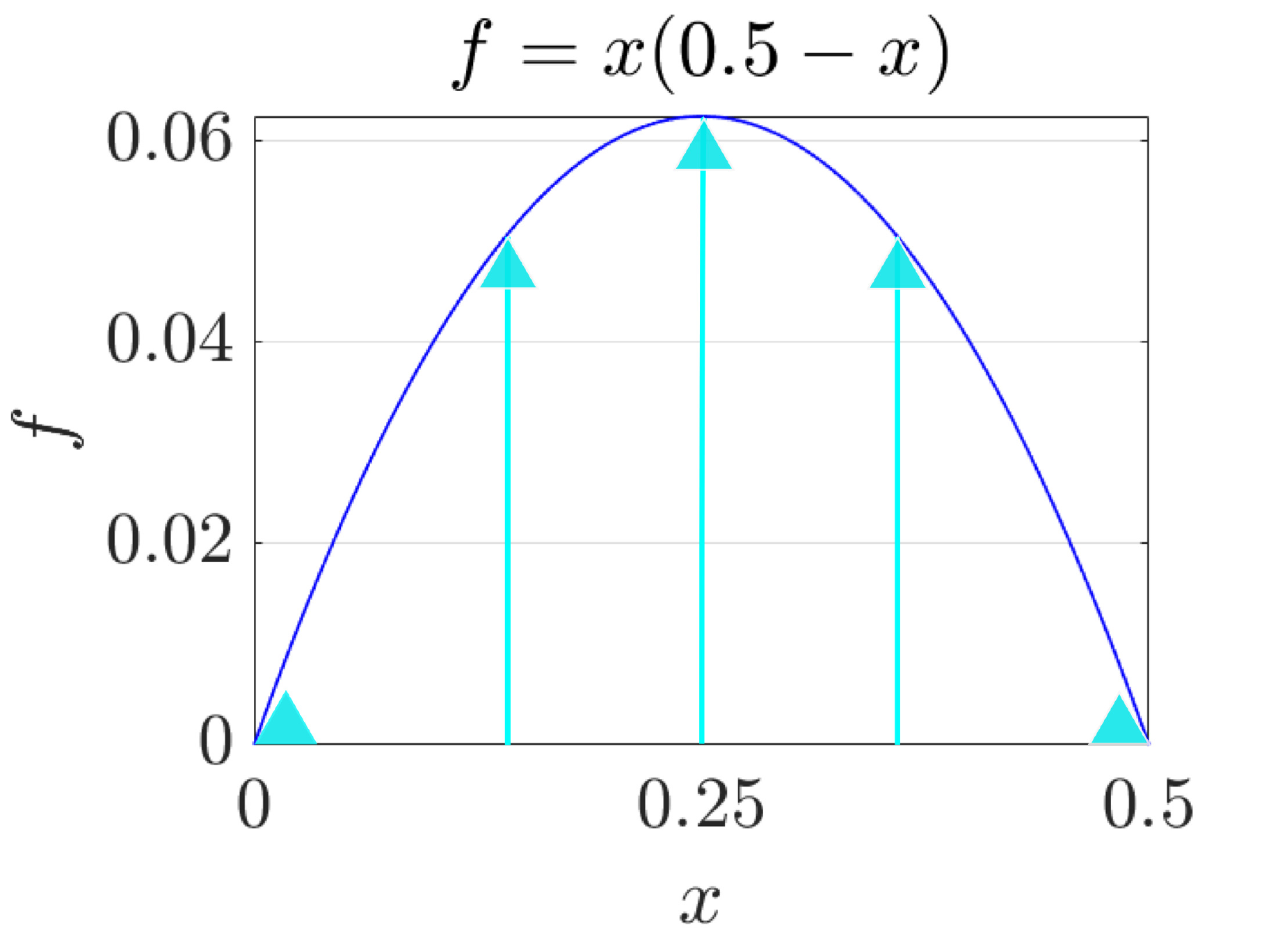}
    \end{subfigure}
    \begin{subfigure}
	\centering
	\includegraphics [width=0.5\linewidth]{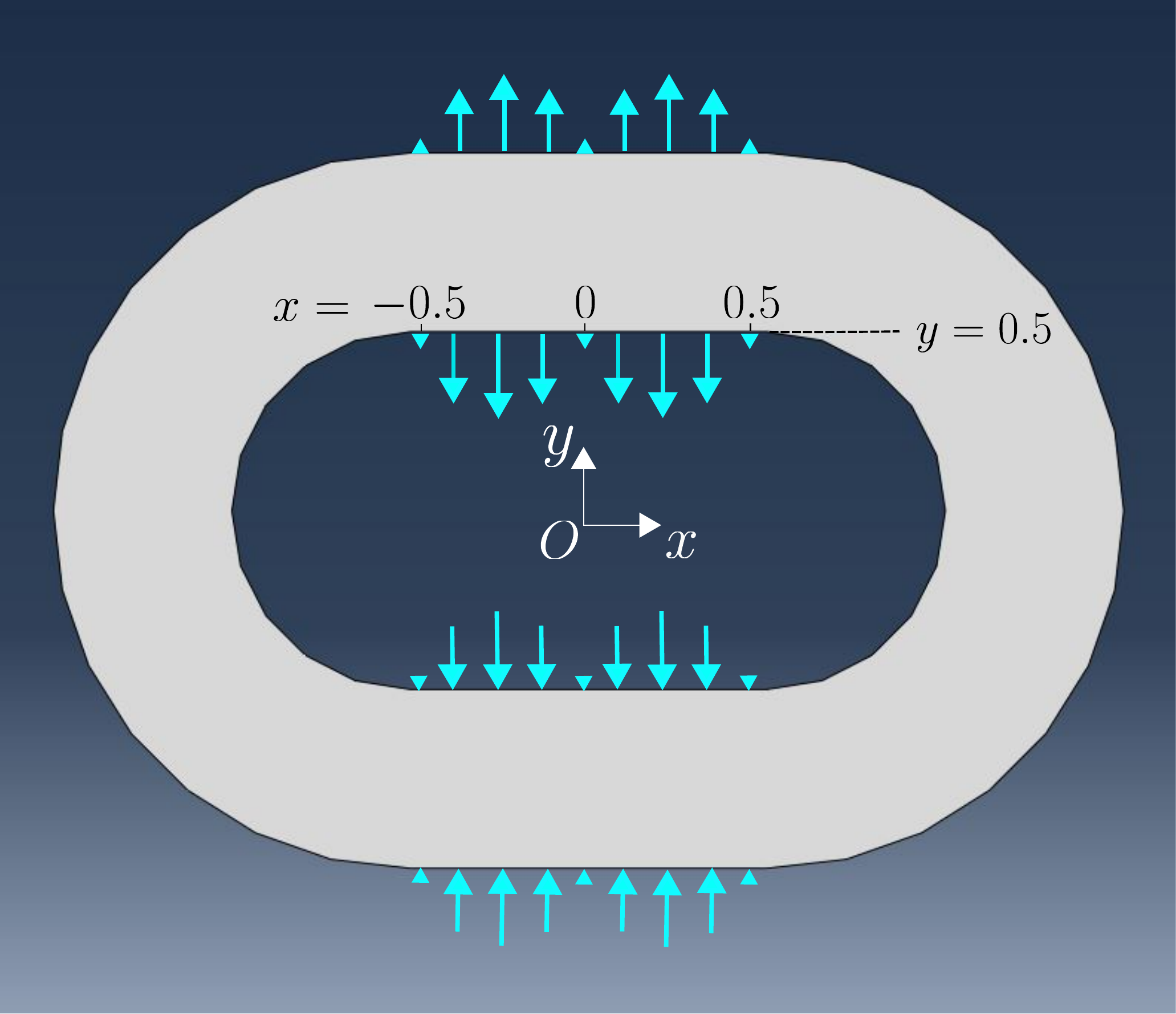}
    \end{subfigure}
    \begin{subfigure}
	\centering
	\includegraphics [width=0.5\linewidth]{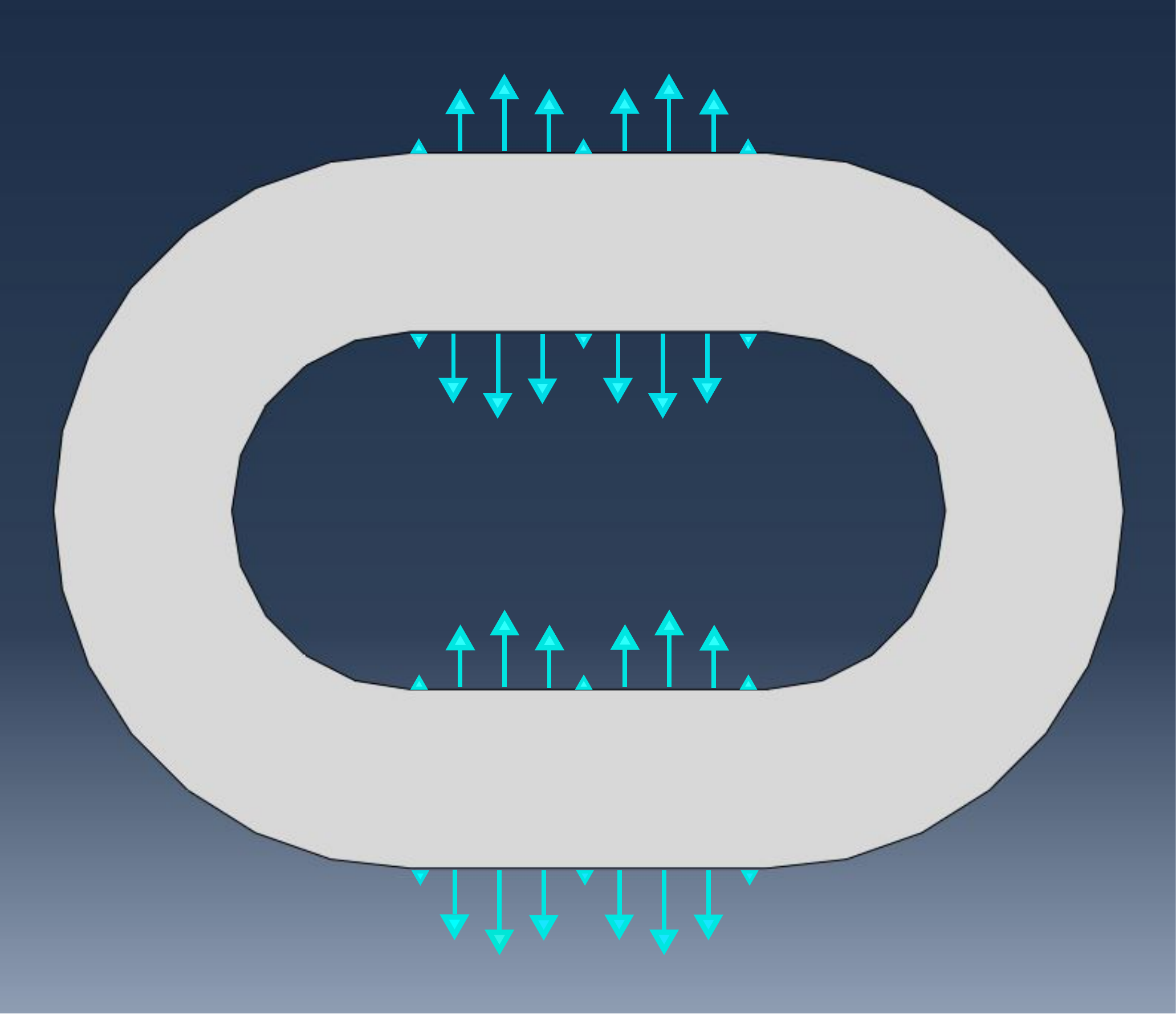}
    \end{subfigure}
    \begin{subfigure}
	\centering
	\includegraphics [width=0.5\linewidth]{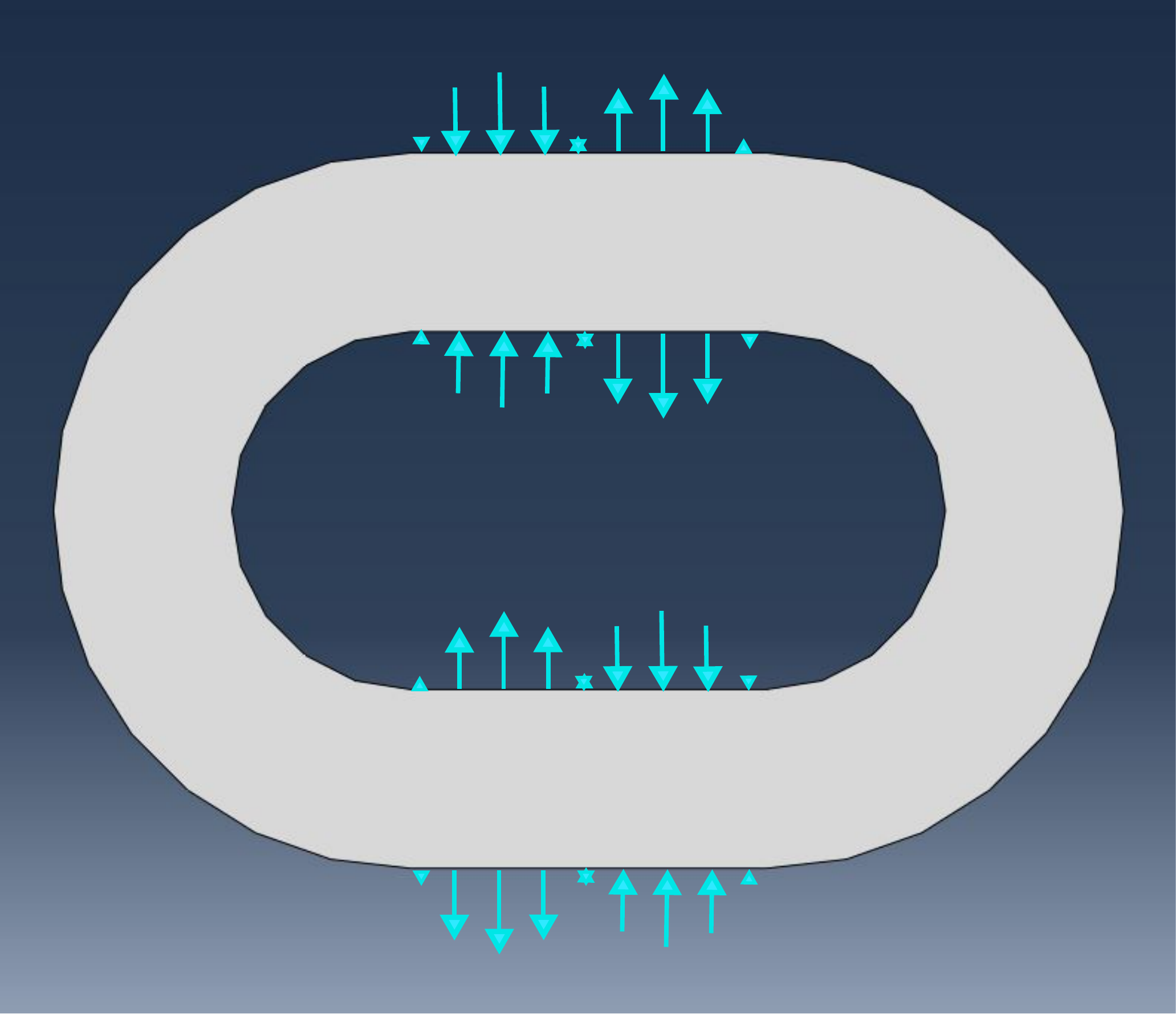}
    \end{subfigure}
\caption{Traction distributions in Example 6. Top-left: local distribution of the magnitude of the $y$ component of traction. Top-right: traction distribution corresponding to non-zero net force and zero net moment on the hole. Bottom-left: traction distribution corresponding to zero net force and zero net moment on the hole. Bottom-right: traction distribution corresponding to zero net force and non-zero net moment on the hole.}
\label{FM}
\end{figure}

In each case, we have assumed a traction distribution that is continuous and has a continuous slope. For instance, for case (1), the traction distribution $\bm{T}(x,y)$ on the upper outer boundary for the portion $0\leq x \leq 0.5$ is
\begin{equation*}
T_x=0, ~~~T_y=f=x(0.5-x).
\end{equation*}
The parabolic function $f=x(x-0.5)$ defined on $0\leq x \leq 0.5$ is shown in the top-left panel of Figure \ref{FM}. The traction distributions on the other portions of the hole surface and the outer boundary follow the same locally parabolic distribution, with their directions indicated in Figure \ref{FM}.  

We use both the strain energy principle (Eq.~\ref{hetero}) and the planar trace principle (Eq.~\ref{homo}) to find approximate stresses $\bm{\sigma}^N$. For $\bm{\sigma}_p$, we use the stress field corresponding to stress problems with the same loading but an {\em orthotropic} elastic body with arbitrarily chosen elastic constants $Y_x=5$, $Y_y=1$, $\nu_{xy}=0.2$, $G_{xy}=0.2$. 

We plot the approximation errors $E_N$ for the three problems (labeled as `1', `2' and `3', respectively) on linear and log-log scales in Figure \ref{FMEn}.  As per the theory, $E_N$ for the case with non-zero net force on the hole with the planar trace principle (curve `PT1' in the figure) should saturate at a non-zero value, and the numerics confirm that: the $E_N$ for all the cases seem to be decaying to zero except for that corresponding to PT1.
\begin{figure}[t!]
    \begin{subfigure}
	\centering
	\includegraphics [width=0.5\linewidth]{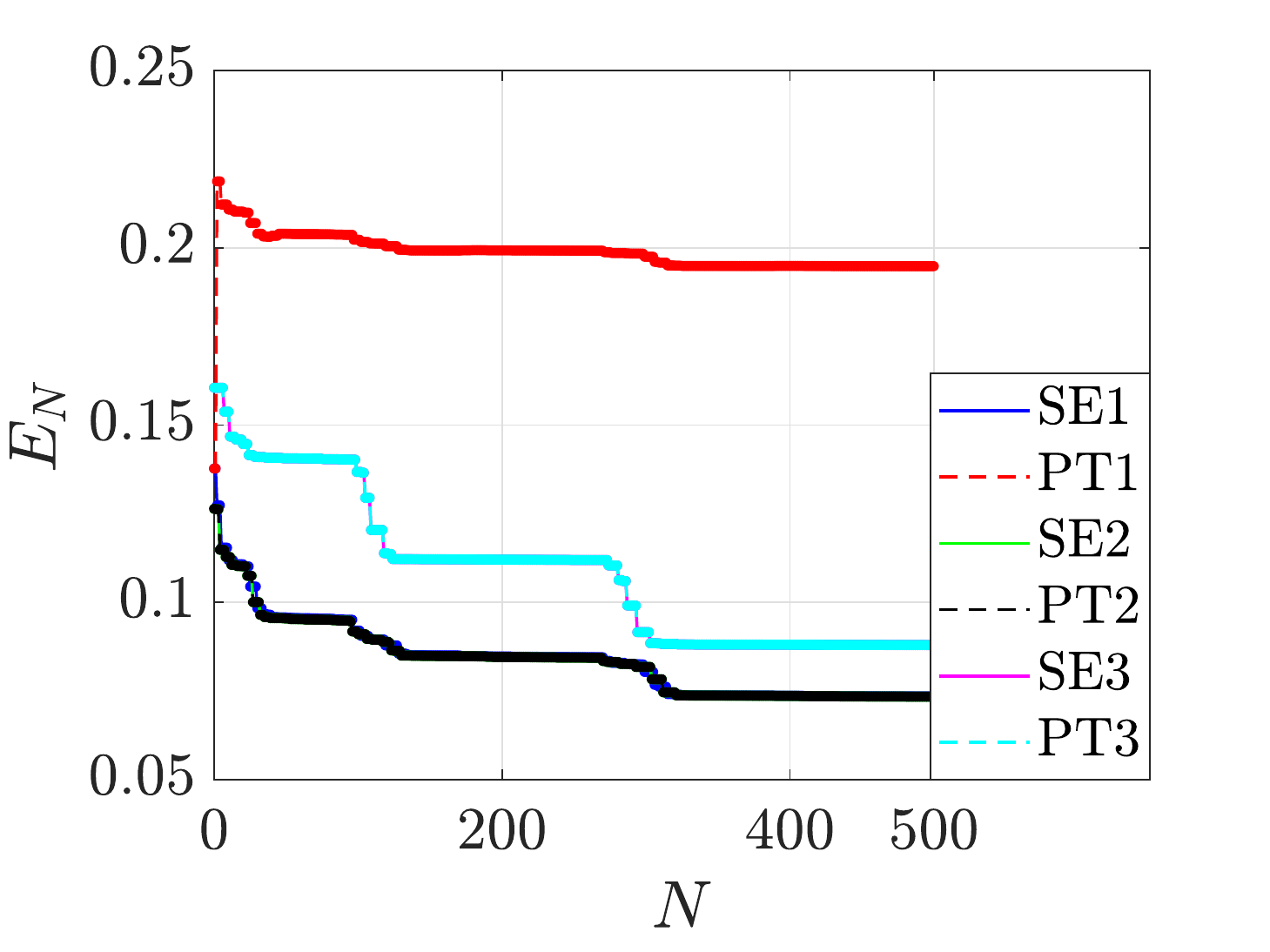}
    \end{subfigure}
    \begin{subfigure}
	\centering
	\includegraphics [width=0.5\linewidth]{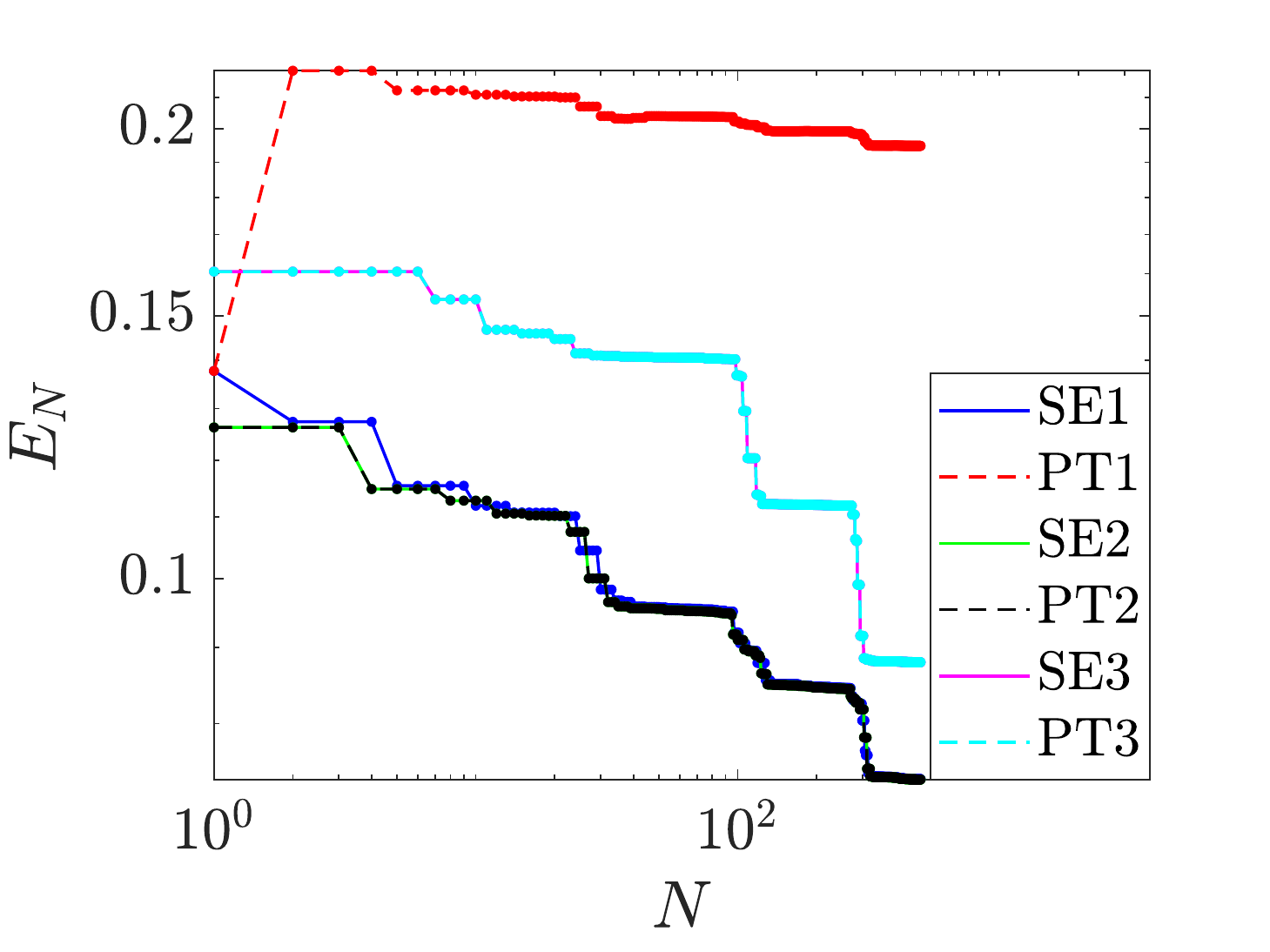}
    \end{subfigure}
\caption{$E_{N}$ in Example 6 on a linear scale (left panel) and a log-log scale (right panel) for the problems with (1) non-zero net force and zero net moment, (2) zero net force and zero net moment, and (3) zero net force and non-zero net moment on the hole, with the strain energy (SE) and planar trace (PT) principles. In the legends, `SE1' denotes $E_N$ for case 1 with the strain energy principle, and so on.}
\label{FMEn}
\end{figure}

This concludes our presentation of examples with homogeneous isotropic bodies. In the next section, we will consider general linear elastic bodies.

\subsection{General linear elastic bodies} \label{examples2}
In this section, we consider planar linear elastic bodies that need not be homogeneous and isotropic. Accordingly, we use the strain energy principle to solve stress problems. We will consider two examples: a square block with spatially varying Young's modulus and an arbitrarily shaped body obeying orthotropic elasticity containing an internally pressurized hole. Both examples correspond to the plane-strain case.

\subsubsection{Example 7: Inhomogeneous block with uniformly pressurized horizontal edges}
We consider a square block made of two different homogeneous materials with Young's moduli $Y_1=1$ and $Y_2=3$, and Poisson's ratio $\nu=0.33$, as shown in the left panel of Figure \ref{JellySteelSketch}. The block is subjected to pressure $p=1$ on the top and bottom edges. 
\begin{figure}[t!]
\begin{subfigure}
\centering
\includegraphics[width=0.4\textwidth]{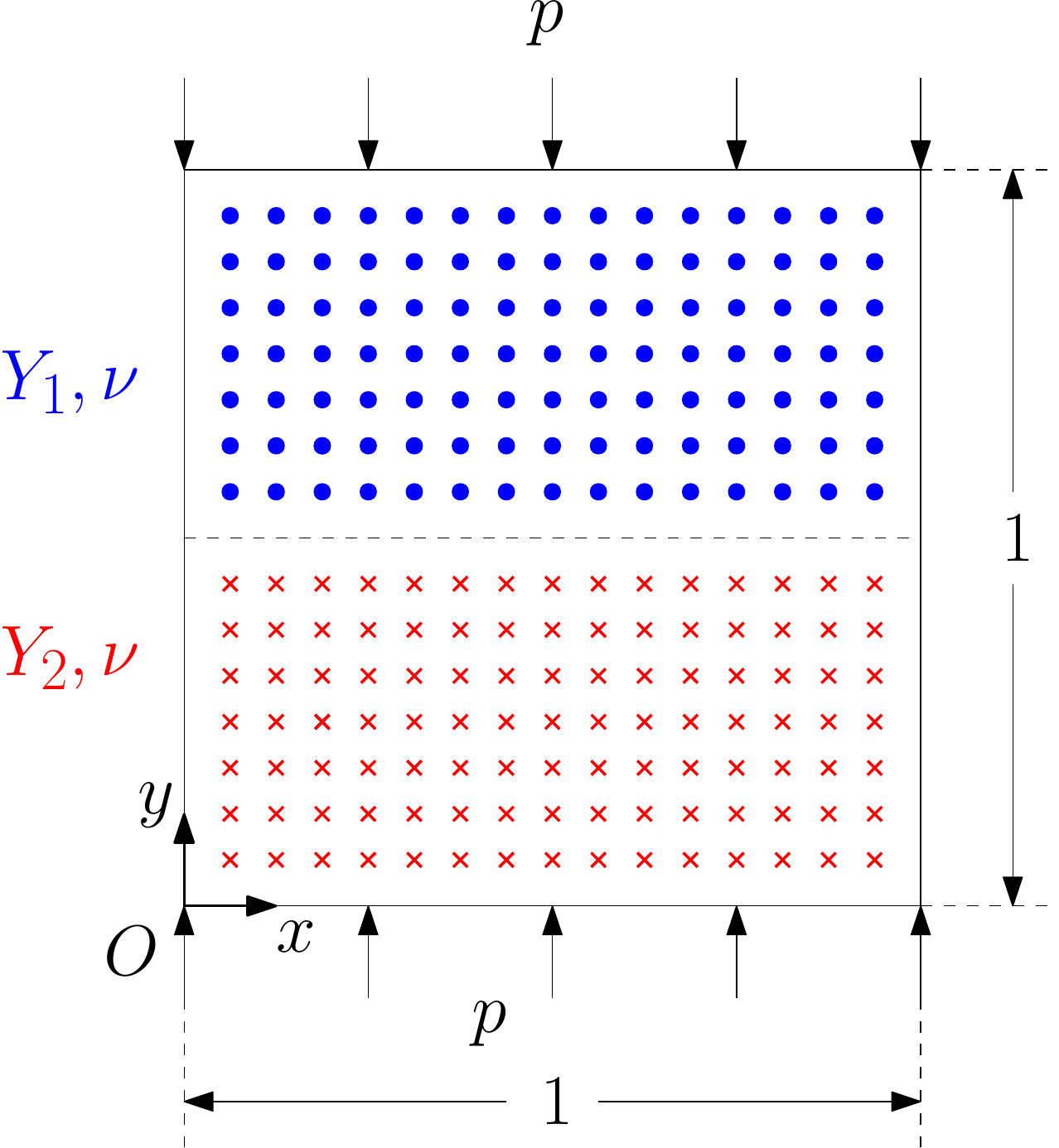}
\end{subfigure}
\begin{subfigure}
\centering
\includegraphics[width=0.52\textwidth]{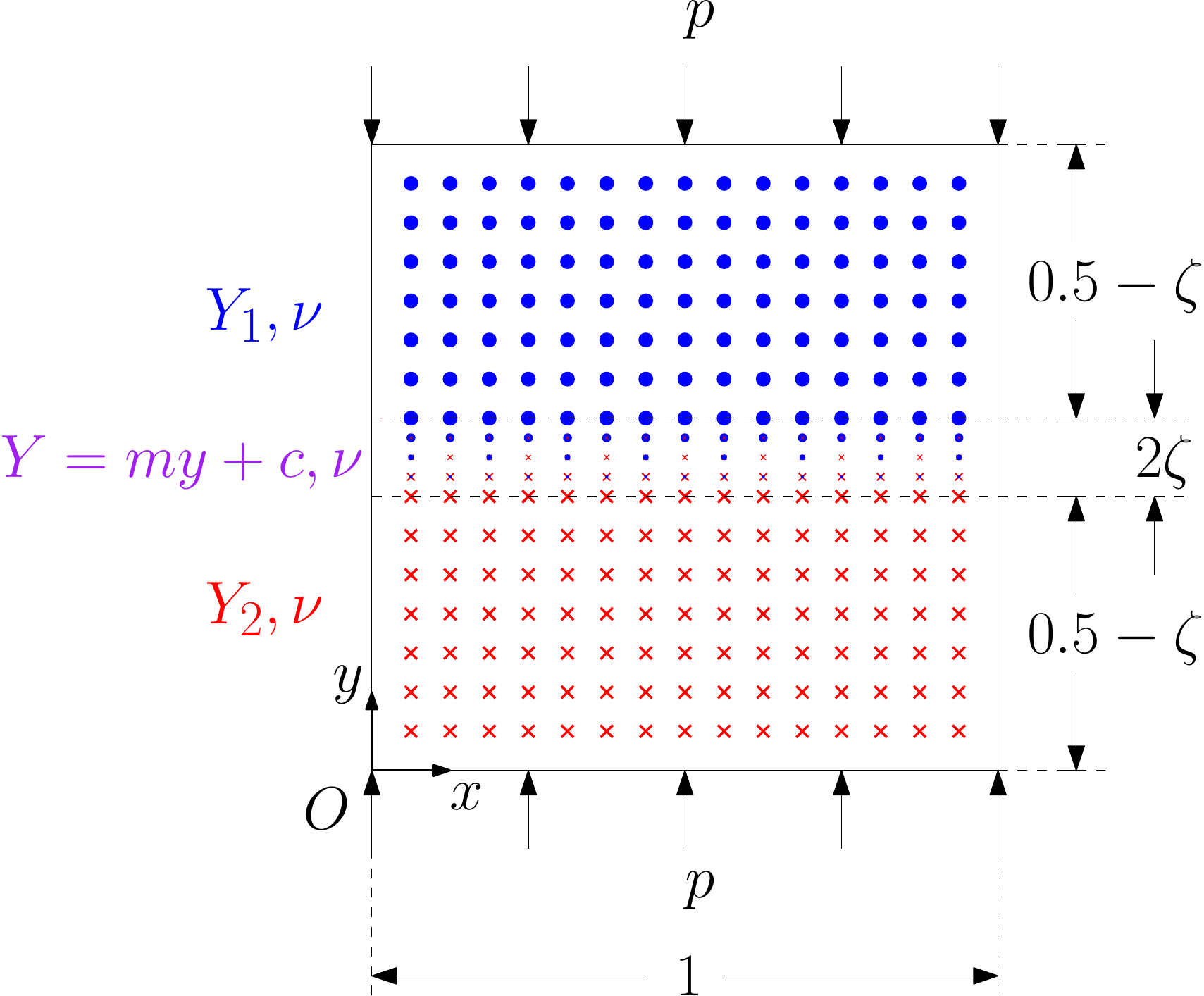}
\end{subfigure}
\caption{Inhomogeneous block subjected to uniform pressure on the top and bottom edges in Example 7. Left: discontinuous Young's modulus; right: piecewise-linear Young's modulus.}
\label{JellySteelSketch}
\end{figure}

A straightforward candidate $\bm{\sigma}_p$ in equilibrium with this loading is:
\begin{equation}
\sigma_{xxp}=0, ~~~ \sigma_{yyp}=-1 ~~~ \text{and} ~~~ \sigma_{xyp}=0 ~~~ \text{in}~~~ \Omega.
\label{sigmap_JS}
\end{equation}
Denoting the domain corresponding to the material with Young's modulus $Y_k$ as $\Omega_k$ ($k=1,2$), the strain $\bm{\varepsilon}_p$ corresponding to $\bm{\sigma}_p$ in $\Omega_k$ is 
\begin{equation*}
\begin{aligned}
\varepsilon_{xxp}&=\frac{\left(1-\nu^2\right)\sigma_{xxp}-\nu(1+\nu)\sigma_{yyp}}{Y_k}=\frac{\nu(1+\nu)}{Y_k},\\
\varepsilon_{yyp}&=\frac{\left(1-\nu^2\right)\sigma_{yyp}-\nu(1+\nu)\sigma_{xxp}}{Y_k}=\frac{-1+\nu^2}{Y_k},\\
\varepsilon_{xyp}&=0.
\end{aligned}
\end{equation*}
We see that $\bm{\varepsilon}_p$ is discontinuous at the interface and, therefore, incompatible. For this reason, the true stress $\bm{\sigma}$ is different from $\bm{\sigma}_p$.

We use the strain energy principle and find a sequence of approximate stresses by substituting the above $\bm{\sigma}_p$ (Eq.~\ref{sigmap_JS}) in Eq.~\ref{hetero}. The true stress (computed using Abaqus) and the approximate stress (with $N=500$) are plotted in Figure \ref{JS}. We observe Gibbs oscillations in the approximate stress (bottom row).

The approximation error $E_N$ is plotted in blue in the left (on a linear scale) and middle (on a log-log scale) panels of Figure \ref{EJS}. In the log-log plot, we find that $E_N$ decays slowly: the slope of the fitted black dotted line is measured to be $-0.22$. Similarly, the convergence of $\mathcal{E}_N$ to $\mathcal{E}$ is slow, as seen in the right panel of Figure \ref{EJS}. The reason for Gibbs oscillations and slow convergence is that $\bm{\sigma}_h$ is discontinuous (since $\bm{\sigma}$ is discontinuous at the interface as seen in the top row of Figure \ref{JS}, while $\bm{\sigma}_p$ is continuous).
\begin{figure}[t!]
	\centering
	\includegraphics [width=1\linewidth]{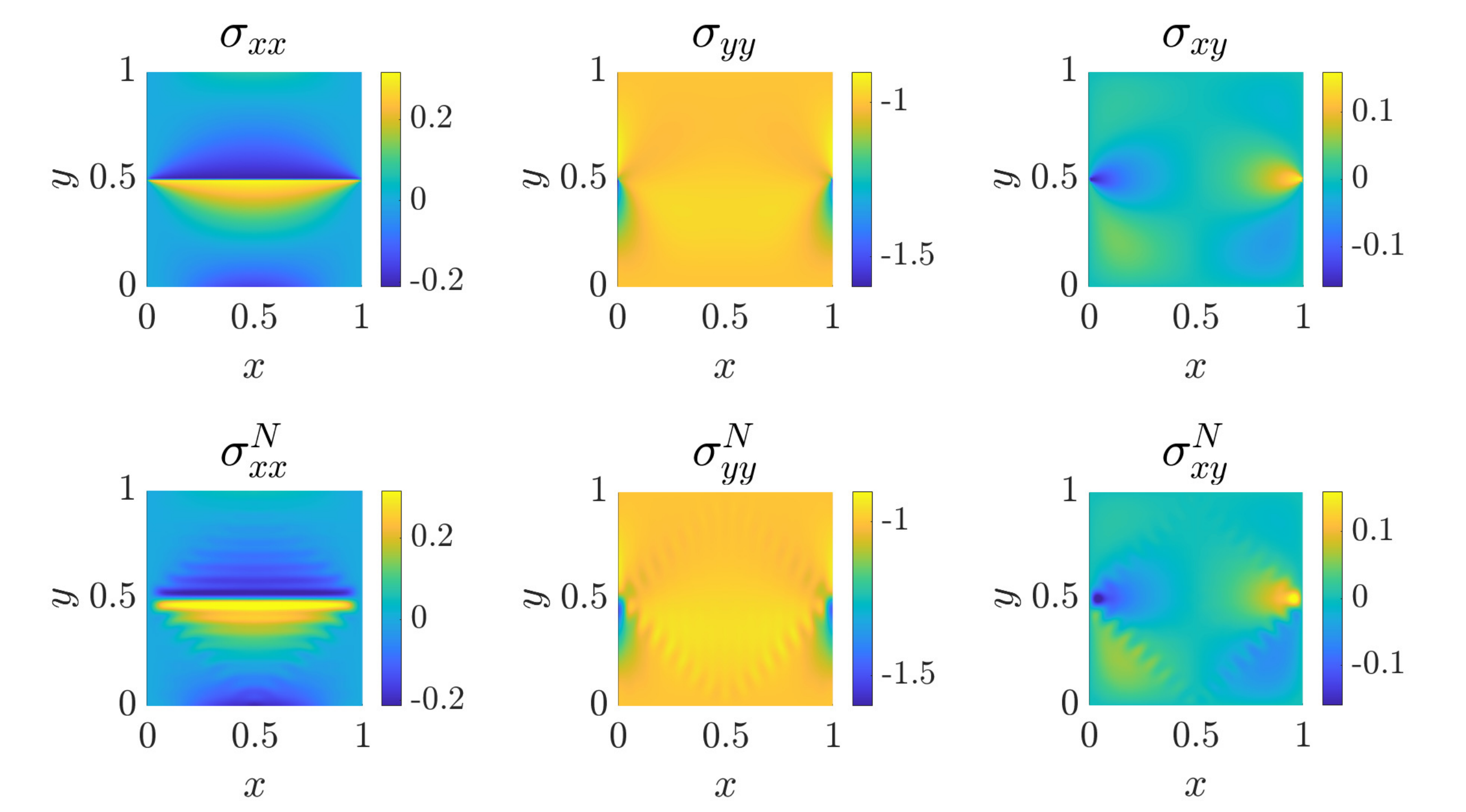}
\caption{True stress (top row) and approximate stress (bottom row) in the block with discontinuous $Y$ in Example 7 (with $N=500$).}
\label{JS}
\end{figure}

\begin{figure}[t!]
    \begin{subfigure}
	\centering
	\includegraphics [width=0.32\linewidth]{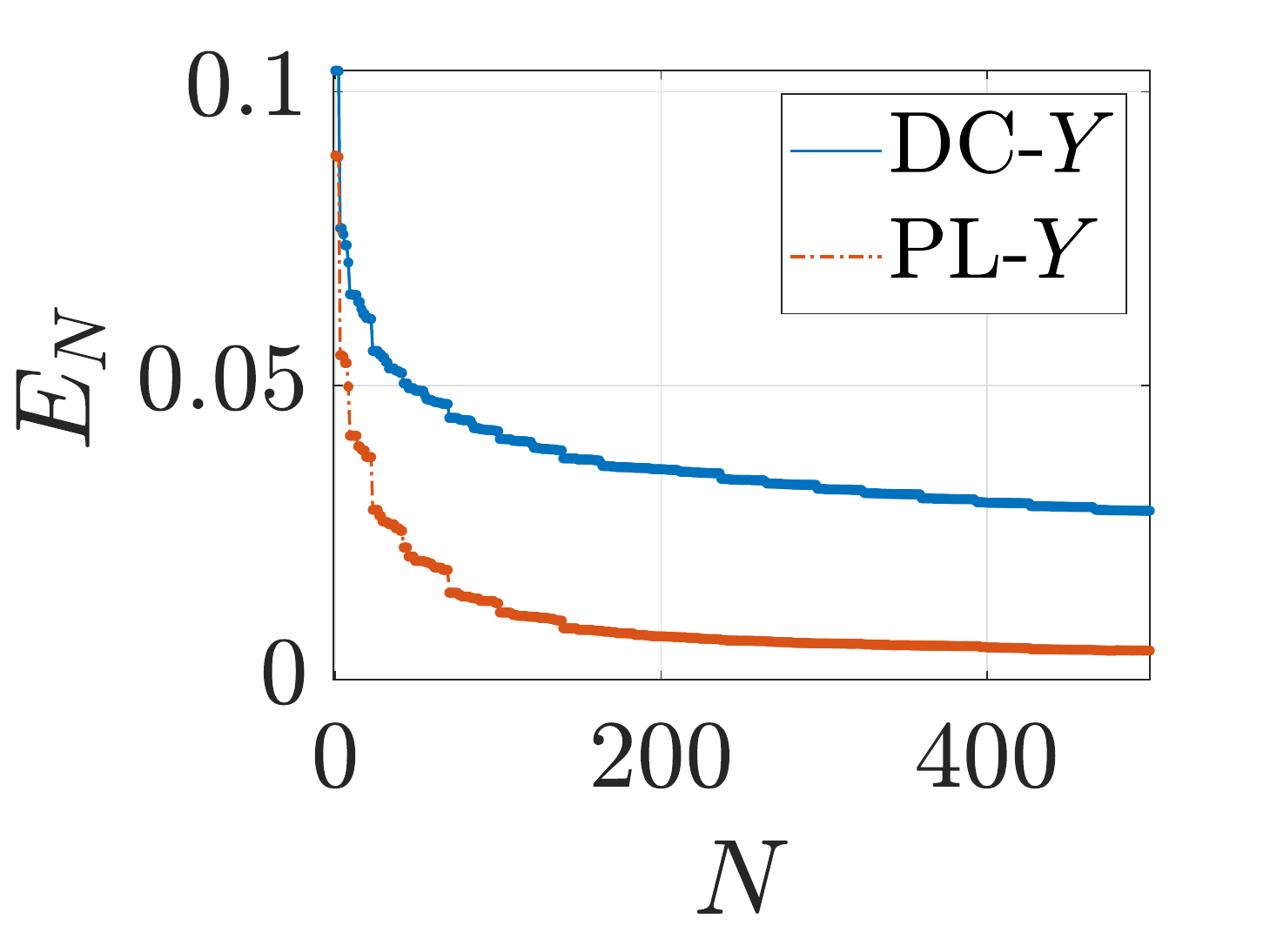}
    \end{subfigure}
    \begin{subfigure}
	\centering
	\includegraphics [width=0.32\linewidth]{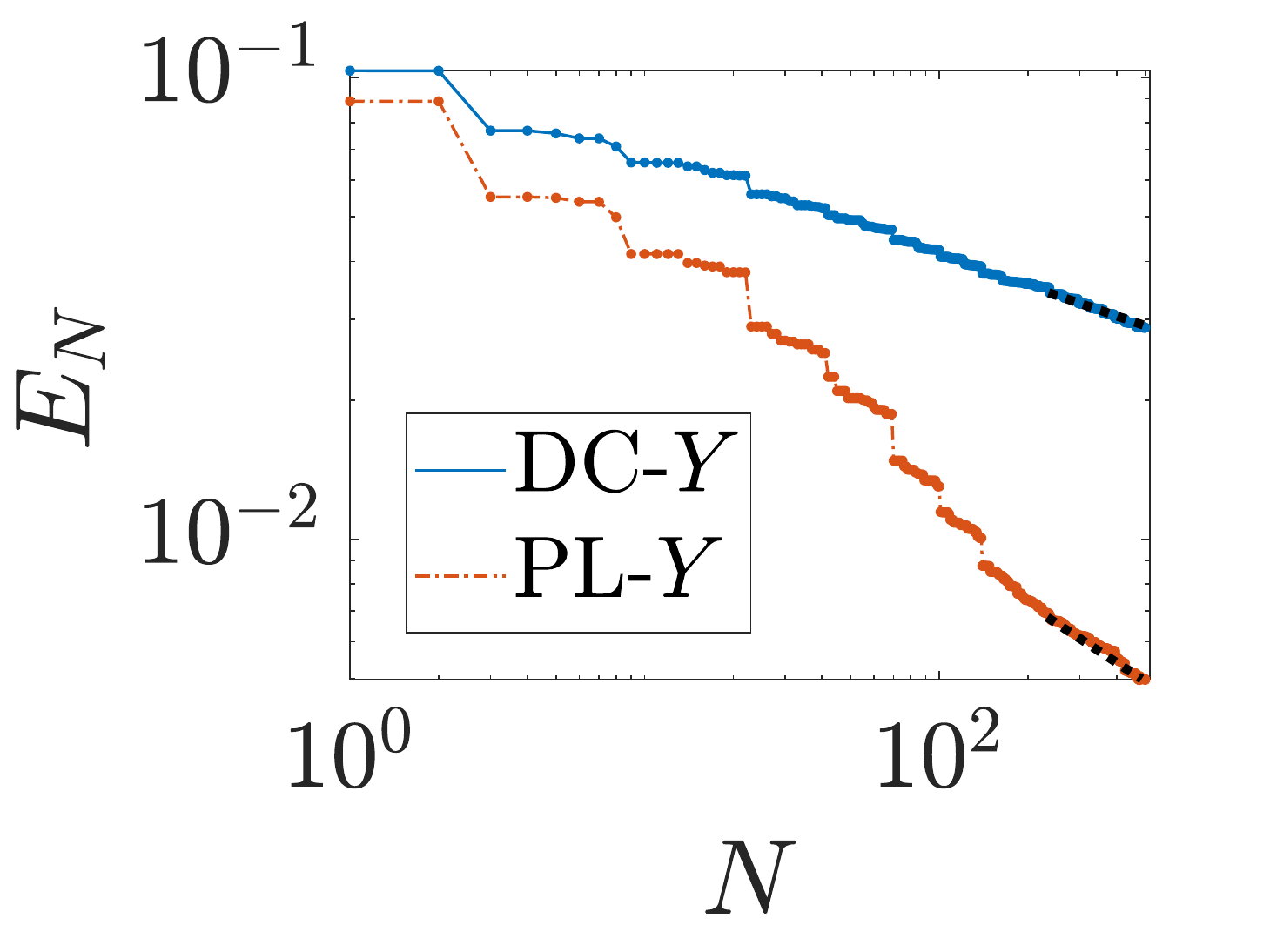}
    \end{subfigure}
    \begin{subfigure}
	\centering
	\includegraphics [width=0.32\linewidth]{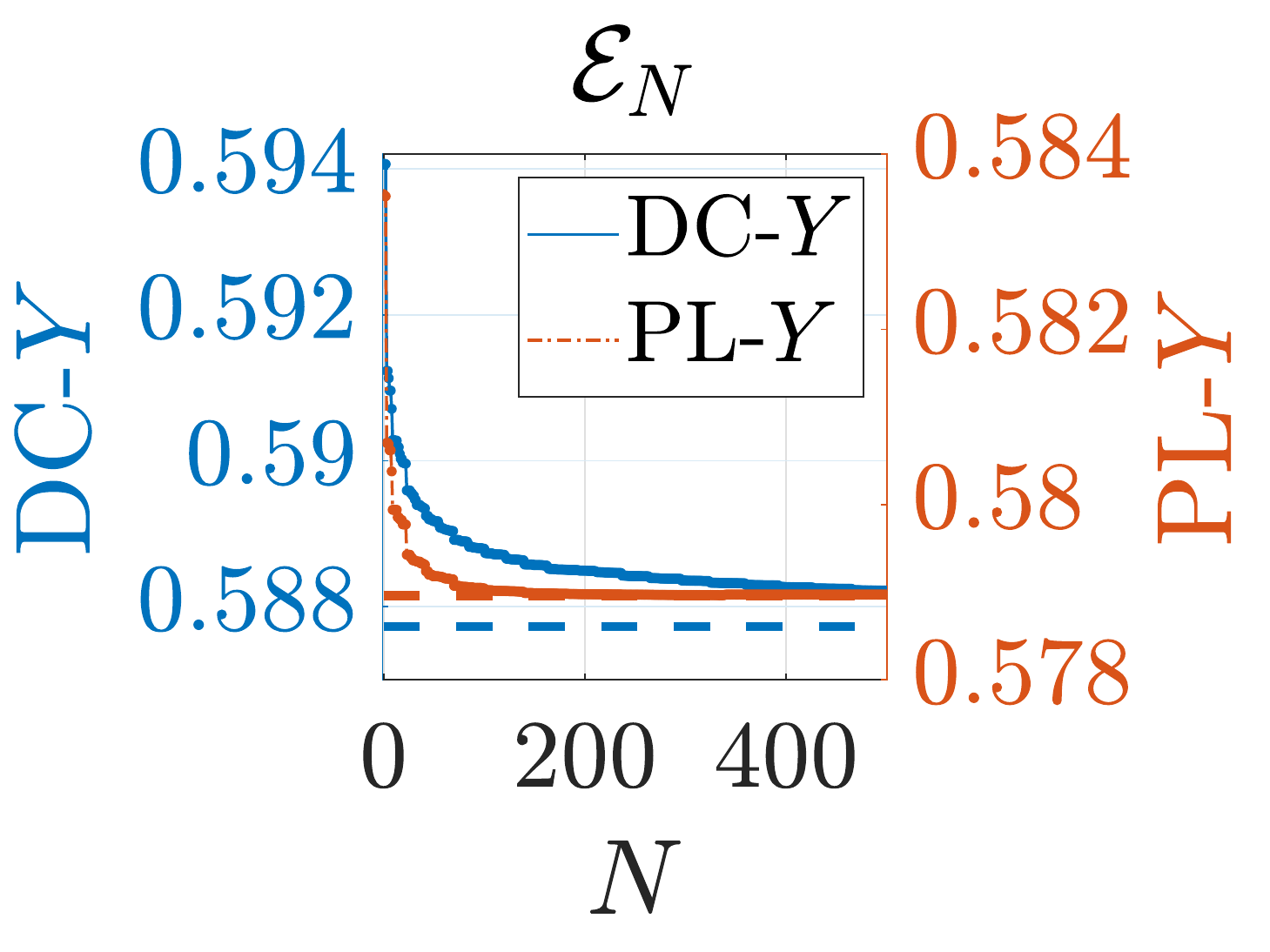}
    \end{subfigure}
\caption{In the figure legends and labels, `DC-$Y$' refers to discontinuous $Y$, and `PW-$Y$' refers to piecewise linear $Y$. Left: Approximation error $E_N$ versus $N$ on linear scale for Example 7; blue: discontinuous $Y$, orange: piecewise-linear $Y$. Middle: $E_N$ on a log-log scale. Right: Strain energy $\mathcal{E}_N$ corresponding to the approximate stresses; the dashed horizontal lines indicate true strain energies.}
\label{EJS}
\end{figure}
\begin{figure}[t!]
	\centering
	\includegraphics [width=1.1\linewidth]{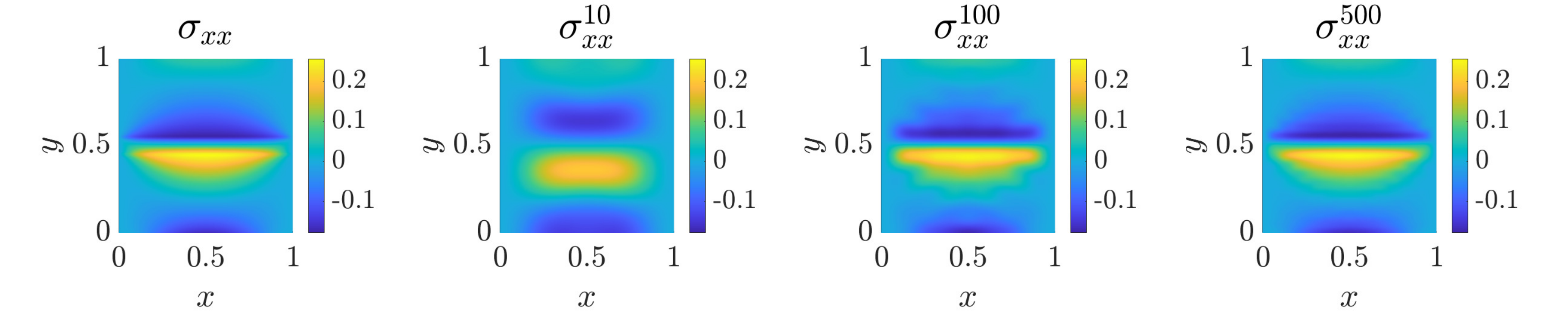}
\caption{True (left-most) and approximate $xx$ components of the stress, with $N=10,100$ and 500, for the block with piecewise-linear $Y$ in Example 7.}
\label{JSNProg}
\end{figure}

We observe faster convergence if smoother transition in material properties is considered. For instance, consider a functionally graded material in which Young's modulus changes from $Y_1$ to $Y_2$ linearly over a distance of $2\zeta$, as shown in the right panel of Figure \ref{JellySteelSketch}. Mathematically,
\begin{equation*}
\displaystyle
Y(y)=
\begin{cases}
\begin{aligned}
\displaystyle
Y_2, ~~~~ &0\leq y \leq 0.5-\zeta,\\
\frac{Y_1-Y_2}{2\zeta}\left( y-\frac{1}{2}\right)+\frac{Y_1+Y_2}{2}, ~~~~ & 0.5-\zeta< y < 0.5+\zeta,\\
Y_1, ~~~~& 0.5+\zeta\leq y \leq 1.
\end{aligned}
\end{cases}
\end{equation*}

In particular, we take $\zeta=0.05$. We plot the approximation error $E_N$ for this case in orange in Figure \ref{EJS}. We find that $E_N$ decays faster on average this time, and the slope of the fitted line in the log-log plot (black dotted line in the middle panel of Figure \ref{EJS}) is measured to be $-0.42$. We also find that $\mathcal{E}_N$, plotted in orange in the right panel of Figure \ref{EJS}, is within 0.1\% of $\mathcal{E}$ with $N=40$ basis functions.

Finally, we see in Figure \ref{JSNProg} where we plot the $xx$ components of $\bm{\sigma}$ and $\bm{\sigma}^N$ that Gibbs oscillations have reduced substantially in comparison to the discontinuous $Y$ case.

\subsubsection{Example 8: Irregularly shaped anisotropic body}
For our final example, we consider an irregularly shaped body made of a homogeneous material obeying orthotropic elasticity. The body has an arbitrarily shaped hole subjected to uniform pressure $p=1$, as shown in Figure \ref{AnisoSketch}.
\begin{figure}[t!]
\centering
\includegraphics[width=0.5\textwidth]{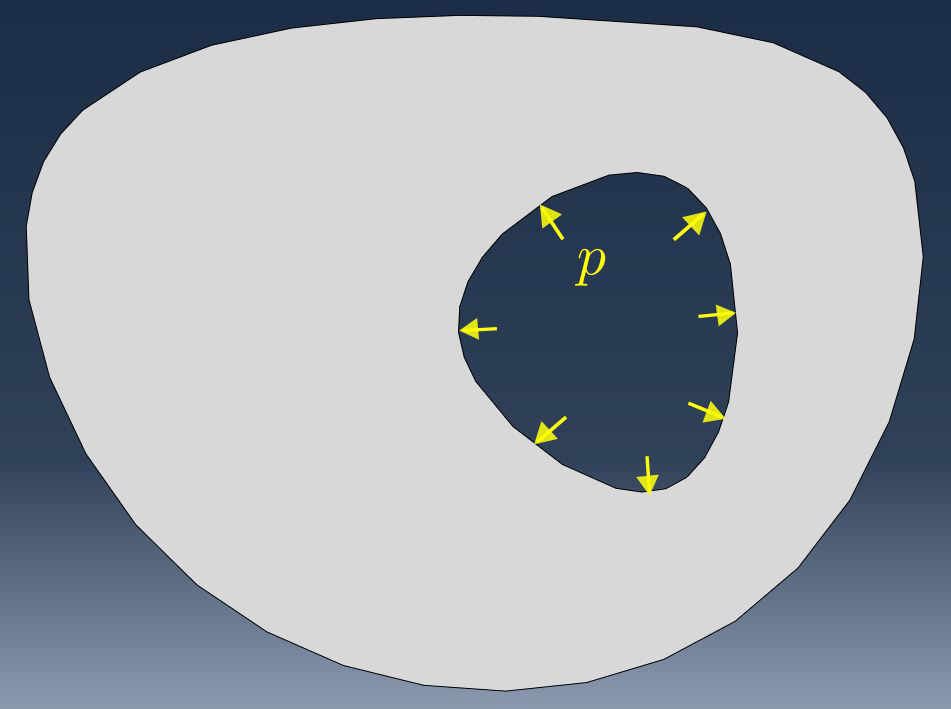}
\caption{An irregularly shaped anisotropic elastic body subjected to uniform internal pressure (Example 8).}
\label{AnisoSketch}
\end{figure}

For orthotropic plane-strain elasticity, the planar components of stress and strain are related as \cite{lekhnitskii}
\begin{equation}
\begin{split}
\epsilon_{xx}=(1-\nu_{xy}^2)\frac{\sigma_{xx}}{Y_x} - \nu_{xy}(1+\nu_{xy})\frac{\sigma_{yy}}{Y_y},\\
\epsilon_{yy}=- \nu_{xy}(1+\nu_{xy})\frac{\sigma_{xx}}{Y_x}+(1-\nu_{xy}^2)\frac{\sigma_{yy}}{Y_y}, \\
\epsilon_{xy}=\frac{\sigma_{xy}}{2G_{xy}},
\end{split}
\label{anisoC}
\end{equation}
where $Y_x$, $Y_y$, $\nu_{xy}$ and $G_{xy}$ are constants. \drop{Accordingly, the integrand in Eq.~\ref{hetero} is
\begin{equation*}
\begin{aligned}
\bm{C}^{-1}\bm{\sigma}_p \cdot \bm{\phi}_i&=\left( (1-\nu_{xy}^2)\frac{\sigma_{xxp}}{Y_x} - \nu_{xy}(1+\nu_{xy})\frac{\sigma_{yyp}}{Y_y}\right)\phi_{xxi}\\
&+\left( - \nu_{xy}(1+\nu_{xy})\frac{\sigma_{xxp}}{Y_x}+(1-\nu_{xy}^2)\frac{\sigma_{yyp}}{Y_y}\right)\phi_{yyi}+\frac{\sigma_{xyp}\phi_{xyi}}{G_{xy}}.
\end{aligned}
\end{equation*}
}
We use $Y_x=1$, $Y_y=2$, $\nu_{xy}=0.33$ and $G_{xy}=1$, and take the particular stress $\bm{\sigma}_p$ to be the stress solution corresponding to an {\em isotropic} elastic body with $Y=1$ and $\nu=0.33$ (obtained using Abaqus). We use the strain energy principle and substitute this $\bm{\sigma}_p$ in Eq.~\ref{hetero} to obtain, for each given $N$, an approximate stress $\bm{\sigma}^N$. 
\begin{figure}[t!] 
	\centering
	\includegraphics [width=1\linewidth]{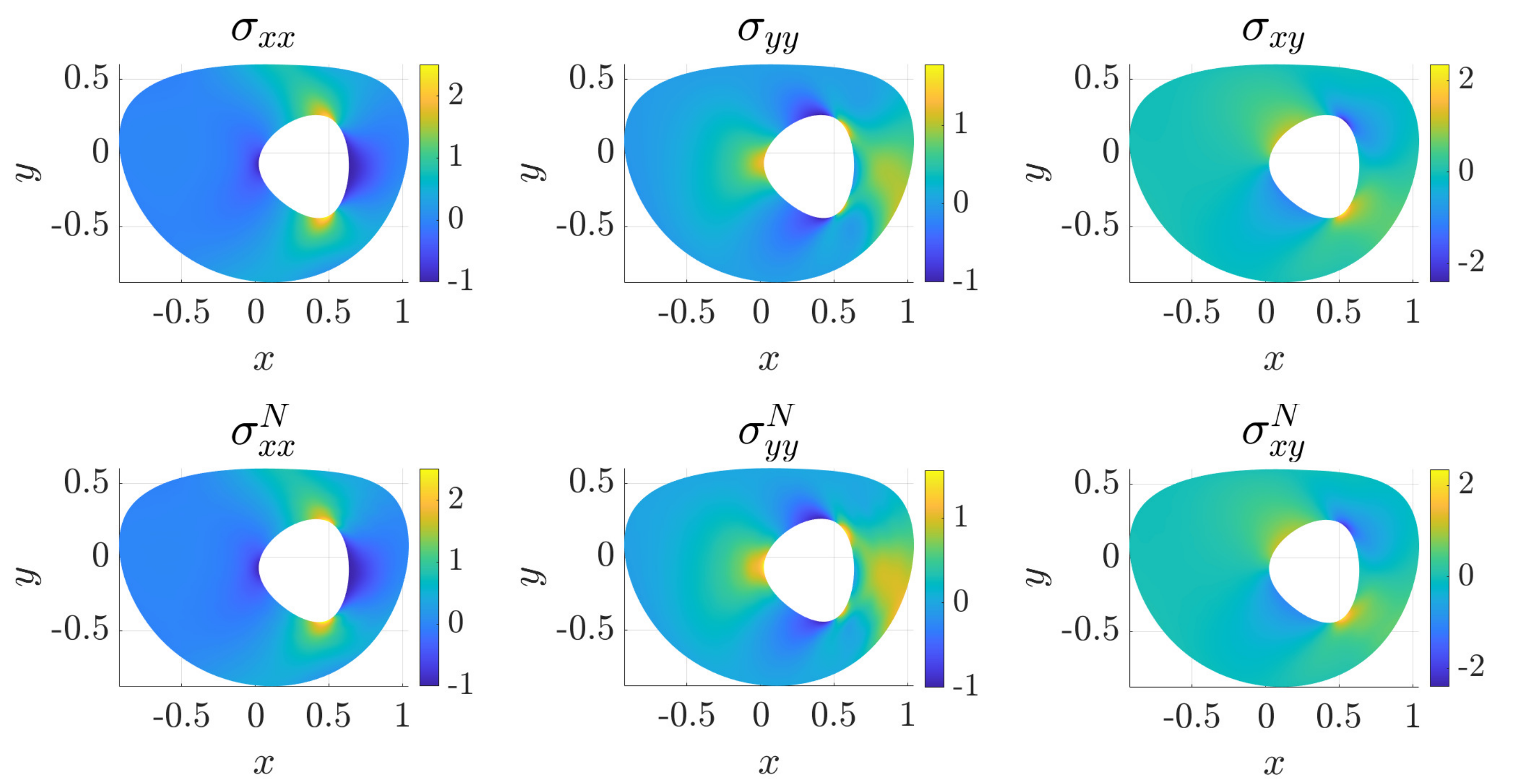}
\caption{True (top row) and approximate (bottom row) stresses in the irregularly shaped anisotropic elastic body in Example 8 (with $N=500$).}
\label{AnisoS}
\end{figure}

We plot the true stress (obtained using Abaqus) and the approximate stress (with $N=500$) in Figure \ref{AnisoS}. The match is good. $E_N$ and $\mathcal{E}_N$ are plotted in Figure \ref{EAniso}. The slope of the black dotted line fitted to $E_N$ in the log-log plot is measured to be $-0.42$.
\begin{figure}[t!]
    \begin{subfigure}
	\centering
	\includegraphics [width=0.5\linewidth]{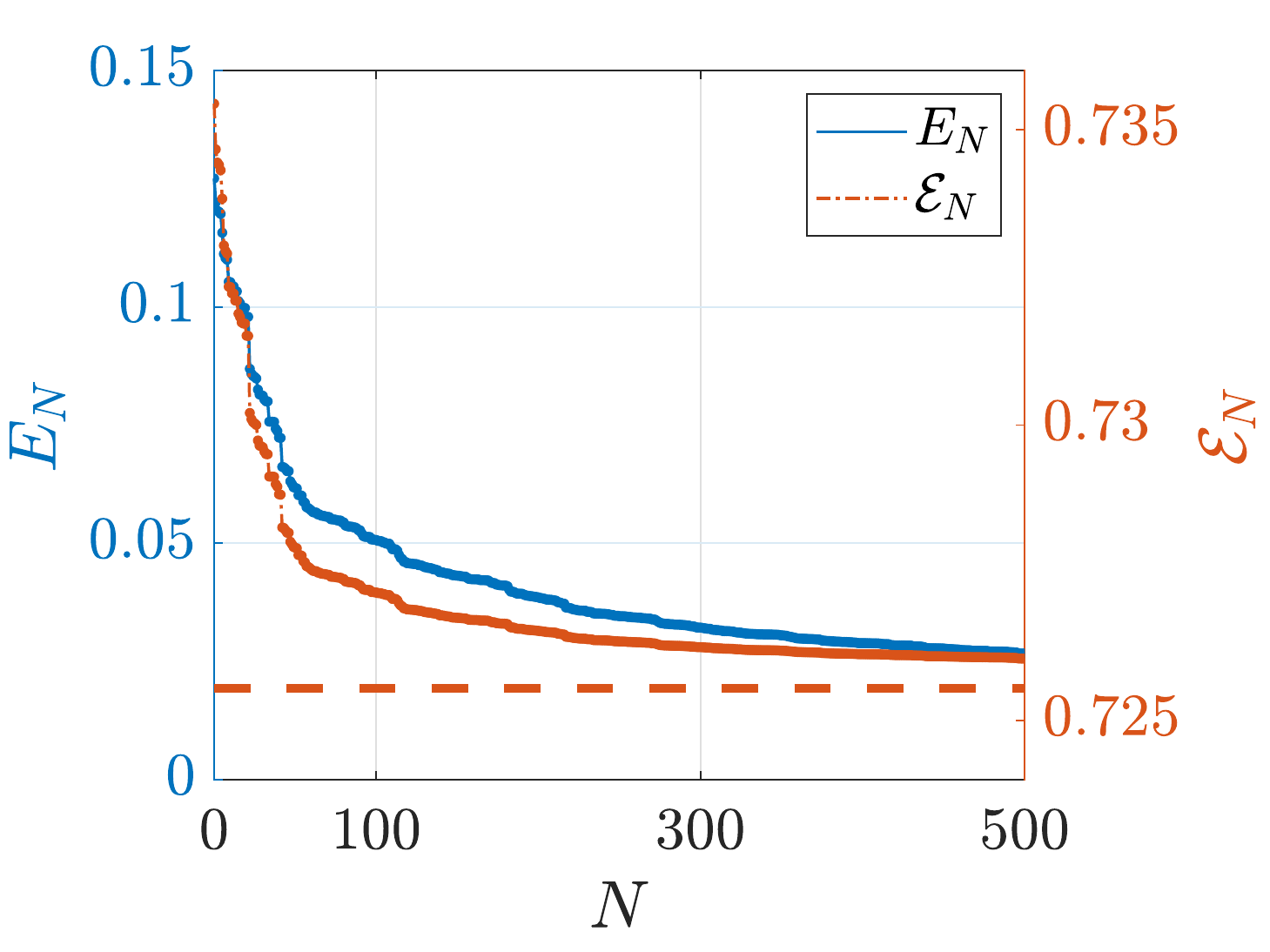}
    \end{subfigure}
    \begin{subfigure}
	\centering
	\includegraphics [width=0.5\linewidth]{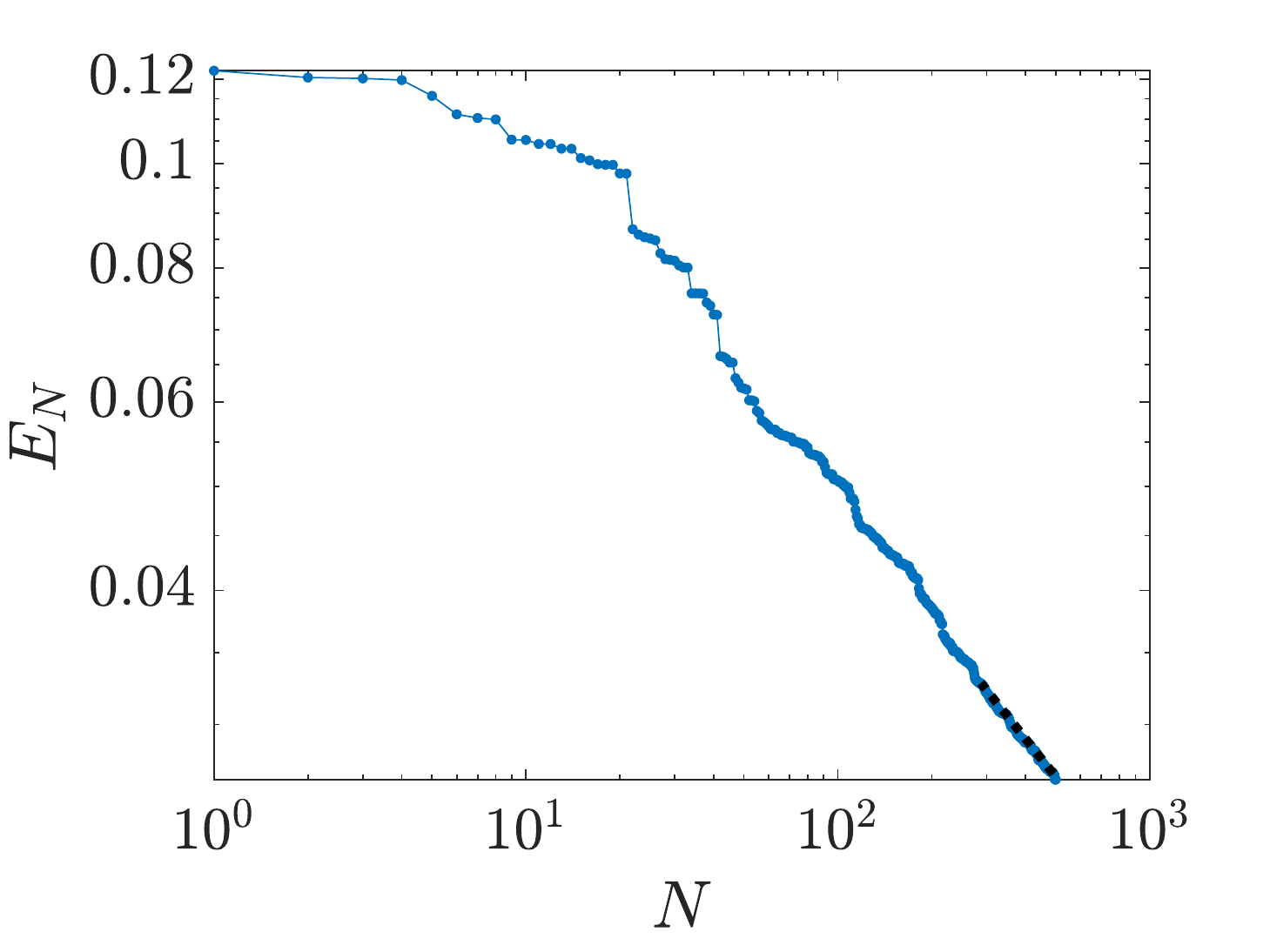}
    \end{subfigure}
\caption{Left: approximation error $E_N$ (in blue) and strain energy corresponding to approximate stress $\mathcal{E}_N$ (in orange) versus $N$ for Example 6; the dashed orange line corresponds to the strain energy of the true stress. Right: $E_N$ on a log-log scale.}
\label{EAniso}
\end{figure}

With this, we conclude the demonstration of our two methods for the solution of planar linear elastic stress problems.

\section{Conclusions} \label{conclusions}
In this paper, we proposed two new methods of solving planar stress problems in linear elasticity. In these methods, the stress to be computed is split into a particular and a homogeneous part, neither of which is required to satisfy strain compatibility. The particular part is {\em any} stress in equilibrium with the loading. The homogeneous part, which is a self-equilibrated traction-free stress, is expanded in basis functions developed in our earlier work \cite{tiwari}. 

In the first method, based on the principle of minimum complementary energy, the coefficients in the expansion are found by minimizing the strain energy over the set of all stresses in equilibrium with the loading. This method is applicable to all linear elastic bodies. The second method is applicable only to the special case of planar homogeneous isotropic bodies subjected to zero net force on each internal hole, if any. In this method, the minimization of the squared $L^2$ norm of the trace of stress (trace energy) gives the true stress. For such bodies, this method yields a computationally cheaper formulation than that corresponding to the first method.

We solved eight examples to demonstrate the working of these methods: six with homogeneous isotropic bodies, and two with other general linear elastic bodies. In these examples, we incorporated sharp corners, multiple-connectedness, non-zero net force and/or non-zero moment on an internal hole, body force, discontinuous surface traction, material inhomogeneity, and anisotropy.

The convergence of the sequence of approximate stresses to the true stress in the strain energy norm was demonstrated numerically in all the examples. It was found that the rate of convergence depends on the regularity of the homogeneous part $\bm{\sigma}_h$; in particular, for discontinuous $\bm{\sigma}_h$, slow convergence was obtained. We also showed that for multiply-connected bodies, our methods do not require explicit imposition of C\'esaro's integral conditions for global strain compatibility, consistent with the comments in \cite{rajagopal}.

In future work, development of a computational method to obtain $\bm{\sigma}_p$ for a general stress problem (with arbitrary geometries and loadings) may be explored. Also, since the basis functions $\bm{\phi}_i$ span all self-equilibrated traction-free stress states without regard to constitutive behaviour, our formulation can in principle be extended to solve any stress-based problem so long as the complementary energy can be expressed solely in terms of stress. Such extensions may be attempted in future. 

However, in our opinion, the importance of this work lies in its two academic contributions, rather than its practical computational utility. Firstly, possibly for the first time, a method that uses tensor-valued stress basis functions to solve solid mechanics problems has been proposed. An expansion with such a basis provides a more natural setting for application of stress-based variational principles, obviating the need to cast the problem in terms of the Airy stress function. Moreover, it allows for a straightforward extension to 3D, unlike an Airy stress function based approach. It also works for cases where an Airy stress function does not exist, viz., multiply-connected bodies with at least one hole subjected to non-zero net force or non-zero net moment. Secondly, we have derived a new variational principle for planar homogeneous isotropic bodies. This new principle is useful for two reasons: firstly, it yields a computationally cheaper formulation than that obtained through the principle of minimum complementary energy; secondly, it shows that in planar homogeneous isotropic bodies, solely the dilatational part of the stress is enough to resolve the issue of strain compatibility.

We hope that this work will lead to new developments in stress-based formulations in solid mechanics. We also hope that the newly discovered planar trace principle will be investigated further and may shed new light on planar elasticity problems.

\appendix

\section{Brief description of the residual stress basis functions \boldmath{$\phi_i$}} \label{abrief}
In this section, we briefly describe the residual stress basis functions $\bm{\phi}_i$ used for expansion of $\bm{\sigma}_h$. 

 We note first that all physically meaningful residual stresses are elements of the closure in $L^2$ norm of the set
\begin{equation}
\mathcal{S}=\biggl\{\bm{\sigma} \left | \bm{\sigma} \in \text{Sym}, \hspace{1mm} \mbox{div} \, \bm{\sigma} = \bm{0}, \hspace{1mm}\left. \bm{\sigma} \bm{n}\right|_{\partial\Omega}=\bm{0}, \hspace{1mm} \int_{\Omega} \bm{\sigma} \cdot \bm{\sigma} \, dA < \infty,\hspace{1mm} \int_{\Omega} \nabla \bm{\sigma} \cdot \nabla \bm{\sigma} \, dA < \infty \right. \biggr\}.
\label{eqdefS}
\end{equation}
We showed in \cite{tiwari} that $\bm{\phi}_i$ are the stationary points of the functional 
\begin{equation}
J_0(\bm{\widetilde{\sigma}})=\frac{1}{2}\int_{\Omega} \nabla \bm{\widetilde{\sigma}} \cdot \nabla \bm{\widetilde{\sigma}} \, dA
\label{extreme}
\end{equation}
on the unit ball in $\mathcal{S}$.
Equivalently, they are the solutions of the eigenvalue problem
\begin{equation}
\begin{array}{cccl}
-\Delta \bm{\sigma} + \nabla_s \bm{\mu}  = \lambda \bm{\sigma} & \text{ and } & \mbox{div} \, \bm{\sigma} = \bm{0} & \text{ in} \hspace{1mm} \Omega,\\
\bm{\sigma n} = \bm{0} & \text{ and } & \nabla_n \bm{\sigma} \cdot (\bm{t} \otimes \bm{t}) = 0 & \text{ on} \hspace{1mm} \partial \Omega,
\end{array}
\label{three_eqns2}
\end{equation}
where $\bm{t}$ is the unit tangent vector at the boundary, $\lambda$ is the constant scalar Lagrange multiplier corresponding to the constraint that the norm be unity, $\bm{\mu}$ is the spatially varying vector Lagrange multiplier corresponding to the pointwise equilibrium constraint $\text{div}\, \bm{\sigma}=\bm{0}$, $\nabla_s \bm{\mu}$ is the symmetric part of $\nabla \bm{\mu}$, and $\nabla_n \bm{\sigma}$ is the normal gradient of $\bm{\sigma}$ at $\partial\Omega$. We arrange the eigenfunctions $\bm{\phi}_i$ in the order of increasing $\lambda_i$.

If $\bm{\phi}_p$ and $\bm{\phi}_q$ are two distinct eigenfunctions, then they satisfy the following orthogonality properties:
\begin{equation}
\int_{\Omega} \bm{\phi}_p \cdot \bm{\phi}_q \, dA =0 ~ ~\mbox{ and } ~ ~ \int_{\Omega} \nabla \bm{\phi}_p \cdot \nabla \bm{\phi}_q \, dA =0.
\label{ort1}
\end{equation}
That is, $\bm{\phi}_i$ are orthogonal with respect to both the $L^2$ and $H^1$ inner products. 

We showed in \cite{tiwari} that $\bm{\phi}_i$ span the closure of $\mathcal{S}$ in the $L^2$ norm. In other words, for any square-integrable residual stress $\tilde{\bm{\sigma}}$, there exists a sequence of constant coefficients ${a_i}$ such that for the stress
$$ \bm{\sigma}_d^N=\tilde{\bm{\sigma}}-\sum_{i=1}^{N} a_i \bm{\phi}_i,$$ 
we have
\begin{equation*}
\displaystyle{\lim_{N \to \infty}} \,\int_{\Omega} \bm{\sigma}_d^N \cdot \bm{\sigma}_d^N \, dA  = 0.
\label{sigmad}
\end{equation*} 

So, 
\begin{equation*}
\tilde{\bm{\sigma}}=\lim_{N \to \infty} \sum_{i=1}^{N} a_i \bm{\phi}_i,
\label{ai}
\end{equation*}
for some sequence $\left(a_i\right)$. Since the closure of $\mathcal{S}$ contains all residual stresses of physical interest, the above equation implies that $\bm{\phi}_i$ can be used to represent any residual stress, without regard to the deformation history that caused it or the material properties of the body.

The eigenvalue problem in Eqs.\ \ref{three_eqns2} can be solved on arbitrary domains using the finite element method (FEM). It can also be solved semi-analytically for special geometries like an annulus and a square. In this paper, we have used $\bm{\phi}_i$ computed using our own FEM routine. The first three $\bm{\phi}_i$ on a rectangle with sides 1 and 1.01, and an annulus with inner and outer radii 0.1 and 0.3, respectively, are shown in Figures \ref{square_modes} and \ref{annular_modes}. Note that we have deliberately chosen slightly different side lengths for the rectangle since, for a square, the second and third modes are degenerate. Consequently, the FEM routine returns arbitrary linear combinations of these two modes, which are not useful for visualization.
\begin{figure}[p]
\centering
\includegraphics [scale=1]{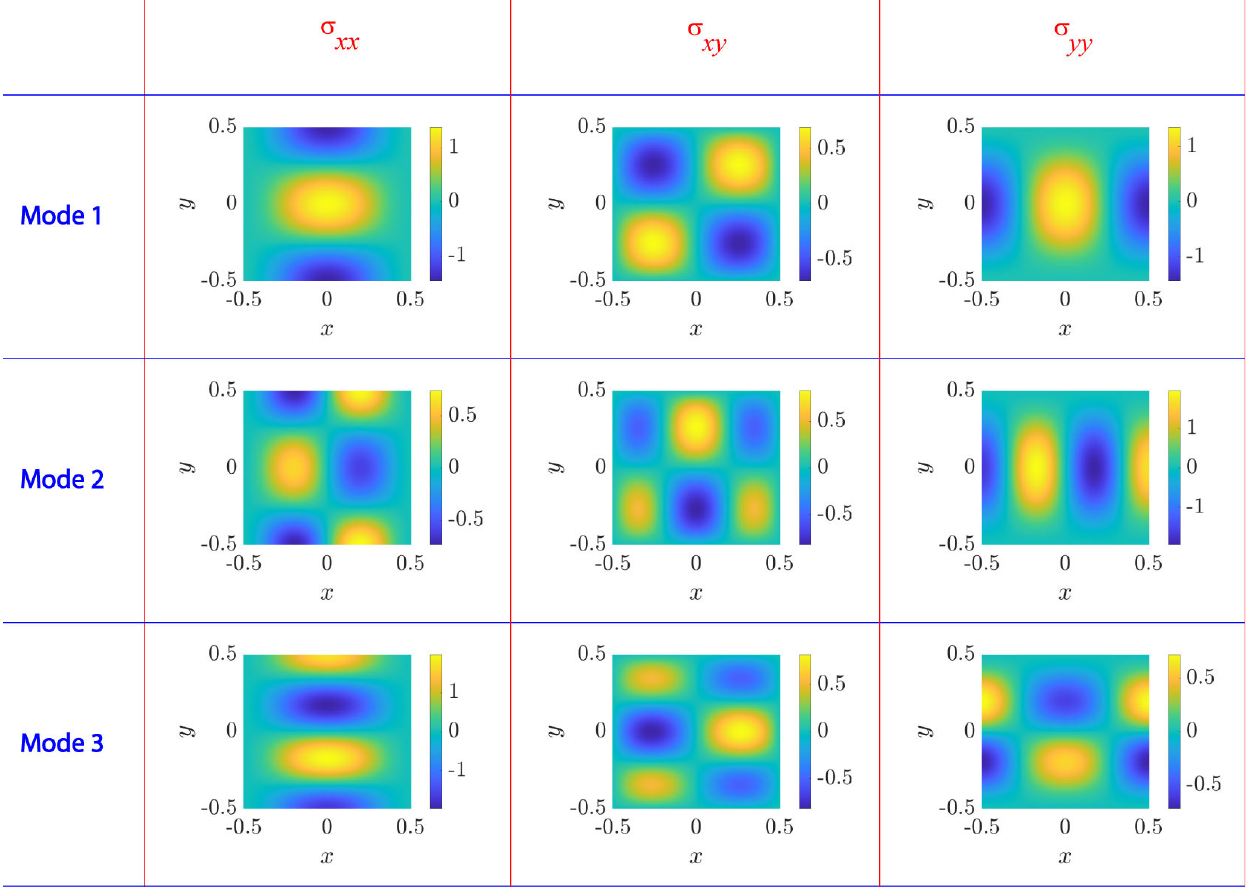}
\caption{First three residual stress basis functions $\bm{\phi}_i$ for a rectangular domain; $\lambda_1=58.54$, $\lambda_2 = 102.37$, $\lambda_3=103.54$. Also in \cite{tiwari}.}
\label{square_modes}
\includegraphics [scale=1]{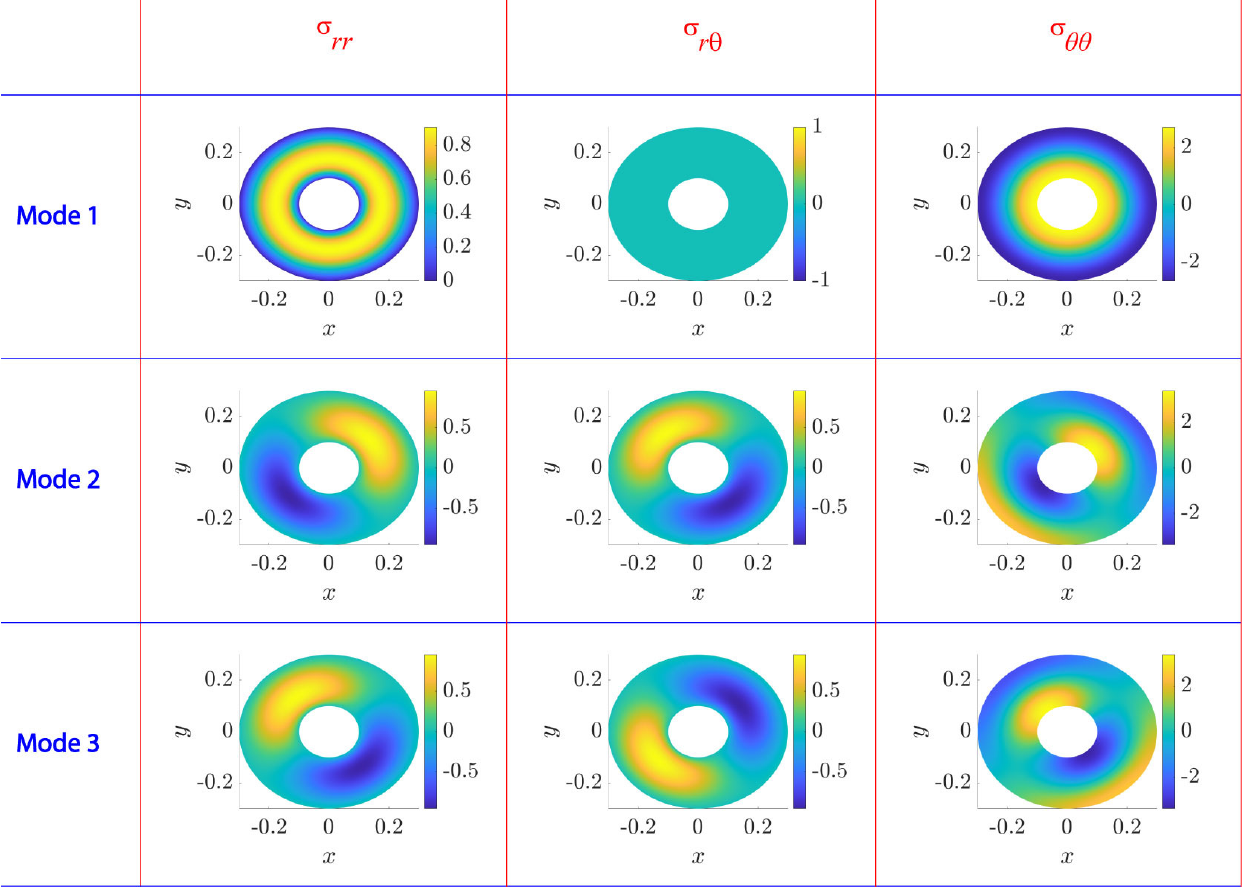}
\caption{First three residual stress basis functions $\bm{\phi}_i$ for an annular domain; $\lambda_1=293.34$, $\lambda_2 = \lambda_3 = 348.76$. Also in \cite{tiwari}.}
\label{annular_modes}
\end{figure}

\section{Proof of Theorem \ref{ipTraces}} \label{ip}
In this section, we prove that the inner product of two planar residual stresses is equal to that of their traces. 

Let the Airy stress functions corresponding to two planar residual stresses $\bm{\sigma}_1$ and $\bm{\sigma}_2$ be $\psi_1$ and $\psi_2$. Since $\bm{\sigma}_1$ and $\bm{\sigma}_2$ are residual stresses, we can assume without loss of generality that \cite{boley}
\begin{equation}
\psi=\psi_{,x}=\psi_{,y}=0 ~~~ \text{on} ~~~ \partial\Omega.
\label{psiBC}
\end{equation} 
Also, we will assume that $\psi_1$ and $\psi_2$ possess the regularity required for all the operations in this proof. 

Consider the quantity $\displaystyle \int_{\Omega} \psi_{1,xy}\, \psi_{2,xy}\, dA$. Using Green's theorem, and boundary conditions on $\psi_1$ and $\psi_2$ (Eq.~\ref{psiBC}), we have
\begin{equation}
\nonumber
\begin{aligned}
\int_{\Omega} \psi_{1,xy}\, \psi_{2,xy}\, dA &= -\int_{\partial\Omega} \psi_{1,xy}\,\psi_{2,x}\, dx - \int_{\Omega} \psi_{1,xyy}\, \psi_{2,x}\, dA\\
& = -\int_{\partial\Omega} \psi_{1,xyy}\,\psi_2\, dy + \int_{\Omega} \psi_{1,xxyy}\,\psi_2\, dA
=\int_{\Omega} \psi_{1,xxyy}\,\psi_2\, dA.
\end{aligned}
\end{equation}
Similarly,
\begin{equation}
\nonumber
\begin{split}
\int_{\Omega} \psi_{1,yy}\, \psi_{2,xx}\, dA = \int_{\partial\Omega} \psi_{1,yy}\,\psi_{2,x}\, dy - \int_{\Omega} \psi_{1,yyx}\, \psi_{2,x}\, dA \\
= -\int_{\partial\Omega} \psi_{1,yyx}\,\psi_2\, dy + \int_{\Omega} \psi_{1,yyxx}\,\psi_2\, dA
=\int_{\Omega} \psi_{1,yyxx}\,\psi_2\, dA=\int_{\Omega} \psi_{1,xxyy}\,\psi_2\, dA.
\end{split}
\end{equation}
From the above two equations, we have that
\begin{equation}
\int_{\Omega} \psi_{1,xy}\, \psi_{2,xy}\, dA = \int_{\Omega} \psi_{1,yy}\, \psi_{2,xx}\, dA.
\label{firstCond}
\end{equation}
We can similarly show that 
\begin{equation}
\int_{\Omega} \psi_{1,xy}\, \psi_{2,xy}\, dA = \int_{\Omega} \psi_{1,xx}\, \psi_{2,yy}\, dA.
\label{secondCond}
\end{equation}

Let us now consider the inner product of $\bm{\sigma}_1$ and $\bm{\sigma}_2$. Using the above two identities, we have
\begin{equation}
\nonumber
\begin{aligned}
\int_{\Omega}\bm{\sigma}_1 \cdot \bm{\sigma}_2 \, dA &= \int_{\Omega} \left( \sigma_{1xx}\sigma_{2xx}+\sigma_{1yy}\sigma_{2yy}+2\sigma_{1xy}\sigma_{2xy}\right)\, dA \\
&= \int_{\Omega} \left( \psi_{1,yy}\psi_{2,yy}+\psi_{1,xx}\psi_{2,xx}+2\psi_{1,xy}\psi_{2,xy}\right)\, dA \\
&= \int_{\Omega} \left( \psi_{1,yy}\psi_{2,yy} + \psi_{1,xx}\psi_{2,xx} + \psi_{1,xx}\psi_{2,yy}+ \psi_{1,yy}\psi_{2,xx}\right)\,dA\\
&=\int_{\Omega}\left( \psi_{1,xx}+\psi_{1,yy}\right)\left( \psi_{2,xx}+\psi_{2,yy}\right)\,dA = \int_{\Omega} \bar{\sigma}_1\bar{\sigma}_2 \, dA.
\end{aligned}
\end{equation}
Thus, we conclude that the planar traces $\bar{\phi}_i$ of the orthonormal basis functions $\bm{\phi}_i$ are also orthonormal. 

\section{Modification of the planar trace principle to include body forces} \label{bf}
In this section, we show that for a planar homogeneous isotropic body obeying linear elasticity subjected to a body force $\bm{b}=-\nabla V$, the minimizer of the functional
\begin{equation}
\label{L2eb}
\mathcal{T}_b(\bm{\sigma})=\int_{\Omega} \left(\bar{\sigma}-\frac{1}{1-\nu}V\right)^2 \, dA
\end{equation}
over the set
\begin{equation*}
\mathcal{Q}_b=\left\{\bm{\sigma} \big| ~ \bm{\sigma}\in \,\text{Sym},~\int_{\Omega} \bm{\sigma}\cdot \bm{\sigma}\,dA<\infty,~\text{div}\,\bm{\sigma}+\bm{b}=\bm{0}, ~ \left.\bm{\sigma n}\right|_{\partial \Omega}=\bm{\tau} \right\}
\end{equation*}
is the true stress. 

Let us first consider simply-connected bodies. We begin by seeking the stationary points of $\mathcal{T}_b$ over $\mathcal{Q}_b$. To that end, we incorporate the constraint $\text{div}\,\bm{\sigma}+\bm{b}=\bm{0}$ through a Lagrange multiplier vector $\bm{\mu}$, and extremize
\begin{equation}
\tilde{\mathcal{T}}_b(\bm{\sigma},\bm{\mu})=\int_{\Omega} \left\{\left(\bar{\sigma}-\frac{1}{1-\nu}V\right)^2- \bm{\mu}\cdot \left(\text{div}\,\bm{\sigma} + \bm{b}\right)\right\}\, dA
\label{tildTb}
\end{equation}
over the set
\begin{equation}
\tilde{\mathcal{Q}}_b=\left\{\bm{\sigma} \big| ~ \bm{\sigma}\in \,\text{Sym},~ \int_{\Omega} \bm{\sigma}\cdot \bm{\sigma}\,dA<\infty,~\left.\bm{\sigma n}\right|_{\partial \Omega}=\bm{\tau} \right\}.
\label{chand}
\end{equation}
Let $(\hat{\bm{\sigma}},\hat{\bm{\mu}})$ be a stationary point and $\bm{\eta}$ be an infinitesimal perturbation such that $\hat{\bm{\sigma}}+\bm{\eta}\in \tilde{\mathcal{Q}}_b$. Then,
$$ \tilde{\mathcal{T}}_b(\bm{\sigma}+\bm{\eta},\bm{\mu}) = \tilde{\mathcal{T}}_b(\bm{\sigma},\bm{\mu})$$
up to first-order terms in $\bm{\eta}$, where we have dropped the hats. Expanding the above using Eq.~\ref{tildTb}, and neglecting the second-order terms in $\bm{\eta}$, we obtain
$$ \int_{\Omega} \left(2\bar{\sigma} \bar{\eta} - \frac{2V\bar{\eta}}{1-\nu}- \bm{\mu}\cdot \text{div}\,\bm{\eta} \right)\, dA=0. $$
Proceeding using routine calculus of variations procedure, we obtain the Euler-Lagrange equation
\begin{equation}
\left\{\bar{\sigma}-\frac{V}{1-\nu}\right\}\bm{I} =- \nabla_s \bm{\mu}/2.
\label{sivab}
\end{equation}
Following the procedure used in Section \ref{HomIso} to eliminate $\bm{\mu}$, we finally obtain 
\begin{equation}
\Delta \left(\bar{\sigma}-\frac{V}{1-\nu}\right)=\Delta \bar{\sigma}+\frac{\text{div}\, \bm{b}}{1-\nu}=0.
\label{traceComp2b}
\end{equation}
Equation \ref{traceComp2b} is a necessary and sufficient condition for compatibility of strain in planar simply-connected bodies \cite{timoshenko}. 

The variation of the Lagrange multiplier $\bm{\mu}$ gives 
$$ \text{div}\,\bm{\sigma}+\bm{b}=\bm{0}.$$
Also, since $\bm{\sigma}\in \tilde{\mathcal{Q}}_b$ (Eq.~\ref{chand}), $\bm{\sigma n}=\bm{\tau}$. So, $\bm{\sigma}$ is in equilibrium with $\bm{\tau}$.

Thus, the stationary point $\bm{\sigma}$ is in equilibrium with $\bm{\tau}$ and satisfies strain compatibility. Therefore, it must be the sought stress solution. Positive definiteness of $\mathcal{T}_b$ implies
that $\bm{\sigma}$ {\em minimizes} $\mathcal{T}_b$ over $\mathcal{Q}_b$.

It can be shown along identical lines to the proof presented in Theorem \ref{PTP} that for multiply-connected bodies, the minimizer of $\mathcal{T}_b$ over $\mathcal{Q}_b$ satisfies global strain compatibility too, provided the net force on each hole is individually zero. Thus, the above variational principle holds for multiply-connected bodies also.

To find an approximate stress $\bm{\sigma}^N$, we write
\begin{equation}
\label{expanseb}
\bm{\sigma}^N=\bm{\sigma}_p + \sum_{j=1}^N a_j \bm{\phi}_j,
\end{equation}
where $\bm{\phi}_j$ are self-equilibrated traction-free basis functions \cite{tiwari}, and $\bm{\sigma}_p$ is any stress satisfying
\begin{equation}
\begin{aligned}
\text{div}\, \bm{\sigma}_p + \bm{b}&= \bm{0} ~~~ \text{in} ~~~ \Omega, \\
\bm{\sigma}_p \bm{n} &= \bm{\tau} ~~~ \text{on} ~~~ \partial\Omega.
\end{aligned}
\label{BVPsigmapb}
\end{equation}
Once a candidate $\bm{\sigma}_p$ is obtained we substitute it in Eq. \ref{expanseb}, which in turn is substituted in the functional $\mathcal{T}_b$ of Eq.~\ref{L2eb} to obtain
\begin{equation}
\label{thatb}
\hat{\mathcal{T}}_b(a_1,a_2,\cdots,a_N)=\displaystyle \int_{\Omega} \left(\bar{\sigma}_p+\sum_{j=1}^{N} a_j \bar{\phi}_j-\frac{V}{1-\nu}\right)^2\, dA. 
\end{equation}
The gradient of $\hat{\mathcal{T}}_b$ with respect to $a_1$ to $a_N$ is set to zero, and orthogonality properties of $\bar{\phi}_j$ are used to finally obtain
\begin{equation*}
a_i = -\int_{\Omega} \left( \bar{\sigma}_p-\frac{V}{1-\nu}\right) \bar{\phi}_i \, dA, ~~~~~~ 1\leq i \leq N.
\end{equation*}
Accordingly, the approximate stress 
\begin{equation*}
\bm{\sigma}^N=\bm{\sigma}_p + \sum_{j=1}^N\left( -\int_{\Omega} \left( \bar{\sigma}_p-\frac{V}{1-\nu}\right) \bar{\phi}_j \, dA\right)\bm{\phi}_j.
\end{equation*}

\bibliographystyle{vancouver} 

\end{document}